\documentclass{amsart}

\usepackage{amsmath, amssymb, amsthm}
\usepackage{graphicx,color}
\usepackage[dvipsnames]{xcolor} %
\usepackage{mathtools}%
\usepackage{bm}%

\usepackage{xargs}%
\usepackage[utf8]{inputenc} %
\usepackage[T1]{fontenc} %
\usepackage{lmodern} %

\usepackage[a4paper, left=2.5cm, right=2.5cm]{geometry} %

\usepackage{caption}
\usepackage{subcaption}

\usepackage[colorlinks,breaklinks]{hyperref} %
\hypersetup{
citecolor=green!50!black,
linkcolor=red!50!black,
}%

\usepackage[noabbrev,capitalise]{cleveref} %

\usepackage[]{todonotes}

\usepackage{nicefrac}%

\usepackage[shortlabels,inline]{enumitem} %

\usepackage[final]{showlabels} %

\usepackage{nameref}
\makeatletter
\let\orgdescriptionlabel\descriptionlabel
\renewcommand*{\descriptionlabel}[1]{%
  \let\orglabel\label
  \let\label\@gobble
  \phantomsection
  \edef\@currentlabel{#1\unskip}%
  \let\label\orglabel
  \orgdescriptionlabel{#1}%
}
\makeatother

\makeatletter
\@namedef{subjclassname@2020}{\textup{2020} Mathematics Subject Classification}
\makeatother

\usepackage{tikz}
\usepackage{tikz-cd} %
\usetikzlibrary{decorations.pathreplacing,decorations.markings,arrows}

\newtheorem{theorem}{Theorem}[section] 
\newtheorem{proposition}[theorem]{Proposition} 
\newtheorem{lemma}[theorem]{Lemma} 
\newtheorem{corollary}[theorem]{Corollary}
\newtheorem{conjecture}[theorem]{Conjecture}

\newtheorem*{theorem*}{Theorem}%
\newtheorem*{corollary*}{Corollary}

\theoremstyle{remark}
\newtheorem{remark}[theorem]{Remark}

\newtheorem{example}[theorem]{Example}
\newtheorem{question}[theorem]{Question}

\theoremstyle{definition}
\newtheorem{definition}[theorem]{Definition}

\newcommand{\defn}[1]{\emph{\color{blue} #1}} %

\def\cC{\mathcal{C}}

\def\cF{\mathcal{F}}
\def\cG{\mathcal{G}}

\def\cL{\mathcal{L}}
\def\cM{\mathcal{M}}

\def\cO{\mathcal{O}}

\def\cQ{\mathcal{Q}}
\def\cR{\mathcal{R}}
\def\cS{\mathcal{S}}
\def\cT{\mathcal{T}}

\def\RR{\mathbb{R}}
\def\SS{\mathbb{S}}

\renewcommand{\b}[1]{\bm{#1}} %
\renewcommand{\c}[1]{\mathcal{#1}} %

\newcommand{\set}[2]{\left\{ #1 \;\middle|\; #2 \right\}} %
\newcommand{\smallset}[2]{\{ #1 \;\big|\; #2 \}} %
\newcommand{\ssm}{\smallsetminus} %
\newcommand{\sprod}[2]{\left\langle  #1 \, , \, #2  \right\rangle} %
\newcommand{\symdif}{\,\triangle\,} %
\newcommand{\norm}[1]{\left\lVert  #1 \right\rVert} %
\newcommand{\one}{\boldsymbol{1}} %
\newcommandx{\ones}[1][1=n]{\one_{#1}} %
\newcommand{\zero}{\boldsymbol{0}} %
\newcommandx{\zeros}[1][1=n]{\zero_{#1}} %
\newcommandx{\Id}[1][1=n]{I_{#1}} %

\DeclareMathOperator{\conv}{conv} %
\DeclareMathOperator{\vol}{vol} %

\DeclareMathOperator{\id}{id} %

\newcommandx{\Sph}[1][1=n]{\SS^{#1}} %
\newcommandx{\Sym}[1][1=n]{\mathfrak{S}_{#1}} %

\newcommandx{\ivl}[1][1=n]{[{#1}]}

\newcommand{\p}[1]{\b{#1}} %

\newcommandx{\pol}[1][1=P]{\b{#1}} %
\newcommandx{\pc}[1][1=A]{\b{#1}} %
\newcommandx{\vc}[1][1=V]{\b{#1}} %
\newcommandx{\hyp}[1][1=H]{\b{#1}} %
\newcommandx{\ha}[1][1=H]{\c{#1}} %

\newcommand{\h}[1]{\bar {#1}} %
\newcommand{\polar}[1]{{#1}^\circ} %
\newcommandx{\Law}[1][1=\h {\pc}]{{\pol[\Lambda]({#1})}} %

\newcommandx{\simp}[1][1=n-1]{\pol[\triangle]_{#1}} %
\newcommandx{\cube}[1][1=n]{\pol[\square]_{#1}}%
\newcommandx{\cpol}[1][1=n]{\pol[\lozenge]_{#1}}%
\newcommandx{\tsimp}[1][1=n-1]{\tilde{\pol[\triangle]}_{#1}} %
\newcommandx{\Hs}[2][1=n,2=k]{\pol[\triangle]_{#1,#2}} %

\newcommandx{\Pn}[1][1=n]{\pol[P]_{#1}}%
\newcommandx{\cPn}[1][1=n]{\pol[P]_{#1}'}%

\newcommandx{\braid}[1][1=n]{\ha[B]_{#1}}%
\newcommandx{\braidB}[1][1=n]{\ha[B]_{#1}^{B}}%
\newcommandx{\braidBB}[1][1=n]{\widetilde{\ha[B]}_{#1}^{B}}%

\newcommand{\lf}{\ell}%
\newcommand{\af}{\gamma}%
\newcommandx{\lfv}[1][1={\p u}]{\sprod{#1}{\cdot\ }}%
\newcommand{\polf}{p}%

\newcommandx{\VO}[1][1={\pc}]{\ha[SH]({#1})}%

\newcommand{\Pset}{\Pi}%
\newcommandx{\PsetA}[1][1={\pc}]{\Pset({#1})}%
\newcommandx{\Scom}[1][1={\pc}]{\overline{\Pi}({#1})}%
\newcommandx{\PScom}[1][1={\OM,T}]{\widetilde{\Pi}({#1})}%
\newcommandx{\PScomF}[2][1={\OM,T},2=F]{\widetilde{\Pi}_{#2}({#1})}%

\newcommandx{\Sp}[1][1={\pc}]{\pol[SP]({#1})}%

\newcommandx{\Z}[1][1={\h {\pc}}]{\pol[Z]({#1})}%

\newcommandx{\sur}[1][1=I]{p_{#1}}%

\newcommandx{\lm}[1][1={\pc}]{M_{#1}}%

\newcommand{\monomials}{\cM}%
\newcommand{\veronese}{\chi}%

\newcommand{\height}{h}
\newcommandx{\HA}[1][1={\h {\pc}}]{\ha_{#1}}%

\DeclareMathOperator{\sign}{sign} %
\newcommand{\plus}{\boldsymbol{+}} %
\newcommandx{\pluses}[1][1=n]{\plus_{#1}} %
\newcommand{\minus}{\boldsymbol{-}} %
\newcommandx{\minuses}[1][1=n]{\minus_{#1}} %

\newcommand{\parallelclass}[1]{\overline{\overline{#1}}}

\newcommand{\OM}{\cM}%

\newcommand{\cov}{\OM}%
\newcommand{\topes}{\cT}%

\newcommandx{\VectOM}[1][1=\vc]{\cov({#1})}%
\newcommandx{\unVectOM}[1][1=\vc]{\un{\cov}({#1})}%

\newcommandx{\LOM}[1][1=\OM]{{#1}^{\textsf{lit}}}%
\newcommandx{\BOM}[1][1=\OM]{{#1}^{\textsf{big}}}%
\newcommandx{\MOM}[1][1=\OM]{{#1}^{\textsf{sw}}}%

\newcommandx{\LOMA}[1][1=\h \pc]{\LOM({#1})}%
\newcommandx{\BOMA}[1][1=\h \pc]{\BOM({#1})}%
\newcommandx{\MOMA}[1][1=\h \pc]{\MOM({#1})}%

\newcommandx{\unLOMA}[1][1=\h \pc]{\un{\LOM}({#1})}%
\newcommandx{\unBOMA}[1][1=\h \pc]{\un{\BOM}({#1})}%
\newcommandx{\unMOMA}[1][1=\h \pc]{\un{\MOM}({#1})}%

\newcommandx{\ipairs}[1][1={\ivl}]{\binom{#1}{2}}%

\newcommandx{\sep}[2][1=X,2=Y]{S({#1},{#2})}%

\DeclareMathOperator{\inv}{inv} %
\DeclareMathOperator{\invset}{inv} %

\newcommand{\reor}[2]{_{-#2}{#1}}
\newcommand{\supp}[1]{\underline{#1}}
\newcommand\restr[2]{{%
  \left.\kern-\nulldelimiterspace %
  #1 %
  \vphantom{\big|} %
  \right|_{#2} %
  }}%
  
\newcommand\contract[2]{#1 / _{#2}}%

\newcommandx{\op}[1][1=X]{I_{#1}}%
\newcommandx{\sv}[1][1=I]{X^{#1}}%

\newcommandx{\Mbraid}[1][1=n]{\braid[{#1}]}%
\newcommandx{\MbraidB}[1][1=n]{\braidB[{#1}]}%
\newcommandx{\MbraidBB}[1][1=n]{\braidBB[{#1}]}%

\newcommandx{\Mbool}[2][1=n, 2=r]{\mathcal{G}^{#2}_{#1}}%

\newcommand{\RG}{\cR\cG} %

\newcommand{\moves}{\cL} %

\newcommand{\idperm}{\id} %

\newcommand{\un}[1]{\underline{#1}} %

\DeclareMathOperator{\rank}{rk}
\newcommand{\rk}[2]{\rank_{#1}(#2)} %

\newcommand{\flats}{\cF} %
\newcommandx{\flatsM}[1][1=\OM]{\flats_{#1}} %

\newcommand{\Dil}[1]{D_1(#1)} %
\newcommand{\unop}{S} %
\newcommand{\unopSet}[1]{\cS\left(#1\right)} %

\newcommandx{\fib}[2][1=\pol,2={\pi}]{\Sigma\left(#1,#2\right)}%

\newcommand{\SC}{\Delta}%
\newcommandx{\OC}[1][1=P]{\SC\left(#1\right)}
\newcommand{\greatest}{\hat{\boldsymbol{1}}}

\title%
{Sweeps, polytopes, oriented matroids, and allowable graphs of permutations}
\date{\today}
\thanks{Supported by the project CAPPS (ANR-17-CE40-0018) of the French National Research Agency ANR, the French\,--\,Austrian project PAGCAP (ANR~21\,CE48\,0020 \& FWF I 5788), the project PID2019-106188GB-I00 of MCIN/AEI/10.13039/501100011033, and the project CLaPPo (21.SI03.64658) of Universidad de Cantabria and Banco Santander.}

\author{Arnau Padrol}
\address[Arnau Padrol]
{Departament de Matem\`atiques i Inform\`atica, Universitat de Barcelona, Gran Via de les Corts Catalanes 585, 08007 Barcelona, Spain.}
\email{arnau.padrol@ub.edu}

\author{Eva Philippe}
\address[Eva Philippe]{Sorbonne Université and Université Paris Cité, CNRS, IMJ-PRG, F-75005 Paris, France.}
\email{eva.philippe@imj-prg.fr}

\begin{document}

\begin{abstract}
A sweep of a point configuration is any ordered partition induced by a linear functional. 
Posets of sweeps of planar point configurations were formalized and abstracted by Goodman and Pollack under the theory of allowable sequences of permutations. We introduce two generalizations that model posets of sweeps of higher dimensional configurations.

Sweeps of a point configuration are in bijection with faces of an associated sweep polytope.
Mimicking the fact that sweep polytopes 
are projections of permutahedra, we define sweep oriented matroids as strong maps of the braid oriented matroid. Allowable sequences are then the sweep oriented matroids of rank~$2$, and many of their properties extend to higher rank. We show strong ties between sweep oriented matroids and both modular hyperplanes and Dilworth truncations from (unoriented) matroid theory. 
Pseudo-sweeps are a generalization of sweeps in which the sweeping hyperplane is allowed to slightly change direction, 
and that can be extended to arbitrary oriented matroids in terms of cellular strings. 
We prove that for sweepable oriented matroids, sweep oriented matroids provide a sphere that is a deformation retract of the poset of pseudo-sweeps. This generalizes a property of sweep polytopes (which can be interpreted as monotone path polytopes of zonotopes), and solves a special case of the strong Generalized Baues Problem for cellular strings.

A second generalization are allowable graphs of permutations: symmetric sets of permutations pairwise connected by allowable sequences. They have the structure of acycloids and include sweep oriented matroids.
\end{abstract}



\maketitle

\section{Introduction}

It is very natural to order a point configuration by the values of a linear functional, and it is not surprising that applications abound in discrete and combinatorial geometry. For example, this is the core of sweep algorithms, a central paradigm in computational geometry (see \cite[Section 2.1]{DCVO08}). 
The simplex methods for linear programming visit vertices of a convex polytope in such a linear order (see for example~\cite{MatousekGartner2006}). Moreover, these orderings are precisely those inducing the Bruggesser-Mani line shellings in the polar polytope~\cite{BruggesserMani1971} 
(see \cite[Lec.~~8]{Ziegler1995}).

The set of all linear orderings of a planar point configuration was already studied by Perrin in 1882 \cite{Per1882}. This was a precursor to
the theory of \defn{allowable sequences}, introduced and developed by Goodman and Pollack~\cite{GP80,GP80b,GP82,GP84,GP93}. 
The idea is the following. Given a configuration~$\pc$ of $n$ points in the plane, for each generic vector $\p u\in\RR^2$, we sweep the plane with a line orthogonal to~$\p u$. The order in which the points are hit by the line gives rise to a permutation $\sigma\in\Sym$ (see \cref{fig:AllowableSequence}).
As~$\p u$ rotates $180^{\circ}$ clockwise, we obtain a sequence of permutations in which:
\begin{enumerate}[(i)]
\item\label{it:AllowSeq1} the move from a permutation to the next one consists of reversing one or more disjoint substrings;
\item\label{it:AllowSeq2} each pair $i,j$ with $1\leq i < j \leq n$ is reversed in exactly one move along the sequence.
\end{enumerate}
An \defn{allowable sequence} is a sequence of permutations from the identity to its reverse ($\sigma,\overline\sigma\in\Sym$ are \defn{reverse} if $\sigma(t)=\overline{\sigma}(n-t+1)$ for all~$t$) fulfilling these two conditions. Contrary to Perrin's claim, Goodman and Pollack showed that there are unrealizable allowable sequences~\cite[Fig.~3 and Thm.~3.1]{GP80}, that is, that do not arise from a point configuration with this construction (c.f.\ \cref{fig:unrealizable_pent}). 

\begin{figure}[htpb]
        \includegraphics[width=\textwidth]{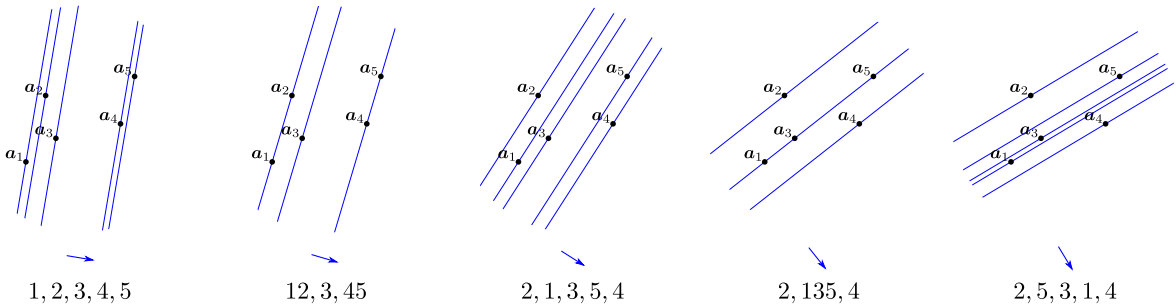}
    \caption{A segment of an allowable sequence. The sweeps between two consecutive permutations in the sequence correspond to ordered partitions.}
    \label{fig:AllowableSequence}
\end{figure}

Allowable sequences are hence purely combinatorial objects abstracting geometric properties of planar point configurations.
They are closely related to pseudoline arrangements and oriented matroids (see~\cite[Sects.~1.10 \& 6.4]{BLSWZ99}), 
although their combinatorial structure is in some senses easier to grasp and manipulate.
In particular, in the \defn{simple} case (where consecutive permutations differ by a transposition), allowable sequences are in correspondence with reduced decompositions of the reverse of the identity and maximal chains in the weak Bruhat order of~$\Sym$, see~\cite[Sec.~6.4]{BLSWZ99}, as well as with (minimal primitive) sorting networks~\cite[Sec.~5.3.4]{KnuthTAOCP3}. This has allowed for their complete enumeration~\cite{Stanley1984,EdelmanGreene1987}, as well as the study of uniform random instances~\cite{AHRV2007,ADHV2019,Dauvergne2019}.

They turned out to be a very effective tool to study problems of geometric combinatorics in the plane, used for example to prove Ungar's theorem (a configuration of $2n$ points not all on a same line determines at least $2n$ slopes)~\cite{Ung82}, to decide the stretchability of arrangements of at most eight pseudolines~\cite{GP80b}, or to estimate the number of $k$-sets and ($\leq k$)-sets~\cite{AlonGyori1986,LovaszVesztergombiWagnerWelzl2004,Welzl1986}. See~\cite[Ch.~V]{GP93} and \cite[Ch.~6]{Felsner2004} for some of their applications.

The construction detailed above extends naturally to any higher dimensional point configuration~$\pc\in\RR^{d\times \ivl}$. Every vector $\p u\in\RR^d$ defines a \defn{sweep}, which is the ordered partition of~$\ivl$ in which the points of~$\pc$ are met when sweeping 
with a hyperplane in direction~$\p u$. Goodman and Pollack already observed that sweeps induce a complex on the unit sphere~$\Sph[d-1]$, \emph{``which has not yet been fully investigated''} (\cite[after Def.~2.3]{GP93}). This was further explored by Edelman~\cite{Edelman2000} and Stanley~\cite{Stan15} who, in particular, presented a tight upper bound for the number of sweeping orders of a $d$-dimensional configuration of $n$~points.

Ordered by refinement, the \defn{poset of sweeps}~$\Scom$ is isomorphic to the face poset of a polyhedral fan 
generated by a hyperplane arrangement~$\VO$, called the \defn{valid order arrangement} by Stanley in a polar formulation~\cite{Stan15}. 
As we discuss in \cref{sec:sweep polytopes}, this is the normal fan of a zonotope: the \defn{sweep polytope}~$\Sp$
(mentionned under the name of \defn{shellotope} by Gritzmann and Sturmfels in~\cite{GritzmannSturmfels1993}).

Posets of sweeps of point configurations are the high-dimensional analogue of realizable allowable sequences. 
However, there is no purely combinatorial description of these objects. Indeed, Hoffmann and Merckx recently adapted the classical Universality Theorem for oriented matroids by Mn\"ev \cite{Mnev88} to give a Universality Theorem for allowable sequences~\cite{HM18}. This shows that already in the plane the problem of deciding whether an allowable sequence arises from a point configuration is very hard (equivalent to the ``existential theory of the reals'', and in particular NP-hard). 

Our main goal is to give a purely combinatorial high-dimensional generalization of allowable sequences that abstracts and encompasses the posets of sweeps of point configurations.
We present two strongly related approaches with two levels of generality (\defn{sweep oriented matroids} and \defn{sweep acycloids}). As we will see, the objects that we introduce fill a gap connecting several topics studied by different communities, providing a new and unified point of view.
We also hope that, beside their intrinsic interest, 
having a purely combinatorial framework without the rigid constraints of realizability
will open the door to new approaches to problems on discrete and combinatorial geometry, as happened in the two-dimensional case.

Our starting point are sweep polytopes. We report alternative constructions that highlight different points of view. On the one hand, sweep polytopes are affine projections of permutahedra. The \defn{$n$-permutahedron}~$\Pn\subset \RR^d$ is a classical polytope whose
normal fan is the braid arrangement~$\braid$.
Up to translation, every affine projection of a permutahedron is a sweep polytope, which gives a natural combinatorial interpretation of permutahedral shadows. 
Moreover, sweep polytopes can be realized as fiber polytopes, and in particular as monotone path polytopes of zonotopes~\cite[Sec.~5]{Edelman2000}.
These are polytopes whose vertices encode the parametric simplex paths induced by a linear functional~\cite{BS92,BilleraKapranovSturmfels1994}.
Conversely, every monotone path polytope of a zonotope is a sweep polytope (under mild technical conditions, see \cref{prop:monotonepathzonotope}). This interpretation of sweep polytopes appears in the study of pivot rules in linear programming~\cite{BLLS22}.

Moreover, this construction naturally reveals a decomposition of sweep polytopes as Minkowski sums of~\defn{$k$-set polytopes}~\cite{AndrzejakWelzl2003,EdelsbrunnerValtrWelzl1997} (see \Cref{rmk:fiberksets}).
After the appearance of the first version of this article, most of these constructions have been generalized to \defn{lineup polytopes}, which encode prefixes of sweeps and are relevant for the $1$-body $N$-representability problem in quantum physics, see~\cite{CLLPPS2021} and references therein.

Inspired by the characterization of sweep polytopes as permutahedral shadows, in \Cref{sec:sweeporientedmatroids} we define \defn{sweep oriented matroids} as strong maps of the oriented matroid of the braid arrangement. 
The strong link between allowable sequences, oriented matroids of rank~$3$, and arrangements of pseudolines is well documented in~\cite[Sects.~1.10 \& 6.4]{BLSWZ99} and explained in terms of \defn{big} and \defn{little oriented matroids}.
These concepts extend to high dimensions too: each sweep oriented matroid of rank~$r$ determines a little and a big oriented matroid of rank~$r+1$ (\Cref{thm:BOMisOM,lem:rank_BOM}). For sweep oriented matroids of rank~$2$, which are equivalent to allowable sequences, we recover the original definitions.
In particular, in the realizable case, the little oriented matroid is the standard oriented matroid associated to the point configuration.

We show that, up to isomorphism, big oriented matroids are characterized by having a \defn{tight modular hyperplane} (\Cref{thm:modularcharacterization}). Modular flats of matroids were introduced by Stanley~\cite{Stanley1971} and play a structural role for matroid constructions~\cite{Brylawski75}. We call a modular hyperplane \defn{tight} if it is no longer modular after the deletion of one of its elements. The operation that determines the big oriented matroid from its sweep oriented matroid extends to all oriented matroids equipped with certain decorations (\Cref{cor:decorations}), and can be seen as an oriented matroid version of~\cite[Thm.~2.1]{Bonin06}. 

We extend the bounds from~\cite{Edelman2000} and~\cite{Stan15} to the non-realizable case (\Cref{thm:boundsweeps}). For this, we show in \Cref{sec:flats} that, at the level of the underlying unoriented matroids, the lattice of flats of a sweep oriented matroid is (a weak map of) the first Dilworth truncation of the lattice of flats of the little oriented matroid (\Cref{thm:Dilworth_unweakmap}).
When one removes all the atoms from a geometric lattice, the resulting poset is no longer a geometric lattice. The first Dilworth truncation is a  lattice obtained by adding the necessary joins in the most generic way to obtain a geometric lattice~\cite{Brylawski1986,Dilworth1944}. 
We can therefore view sufficiently generic sweep oriented matroids as an oriented version of the first Dilworth truncation of the associated little oriented matroid. Unfortunately, in contrast to rank~$3$, not every (little) oriented matroid can be extended to a big oriented matroid (\Cref{thm:unextendable}). The question of characterizing oriented matroids admitting such an extension is open.

In \Cref{sec:pseudosweeps}, we discuss \defn{pseudo-sweeps}, which correspond to sweeps in which the sweeping hyperplane is allowed to change direction (in a controlled monotonous way). Whereas sweeps of a point configuration correspond to the parametric (coherent) monotone paths on an associated zonotope,
pseudo-sweeps take into account all monotone paths. They admit a polar formulation in terms of galleries and cellular strings of pseudo-hyperplane arrangements, which extends to oriented matroids~\cite{Bjorner1992}. This way, for every (little) oriented matroid, even those that cannot be extended to a big oriented matroid, one can define a poset of pseudo-sweeps. In general, an oriented matroid~$\OM$ can be the little oriented matroid of several sweep oriented matroids; each with a different associated poset of sweeps. They are all subposets of the poset of pseudo-sweeps of~$\OM$. A classification of the cases when all pseudo-sweeps are actual sweeps is given in~\cite{EPLM18}.

There is a lot of literature concerning the graphs of pseudo-sweep permutations of oriented matroids. Cordovil and Moreira had shown that they are connected~\cite{CordovilMoreira1993}, extending to oriented matroids results that went back to Tits~\cite{Tits1969} (for reflection arrangements), Deligne~\cite{Deligne1972} (for simplicial arrangements), and Salvetti~\cite{Salvetti1987} (for realizable oriented matroids). 
More results concerning 
graphs of pseudo-sweeps can be found in~\cite{AthanasiadisSantos2001,ReinerRoichman2013}.

The topology of the posets of pseudo-sweeps has been extensively studied as a special case of the \defn{generalized Baues problem}~\cite{BS92,Reiner1999}. Without the trivial sweep, their order complexes have the homotopy type of, but in general are not homeomorphic to, a sphere. In the realizable case, Billera, Kapranov, and Sturmfels proved that the poset of sweeps is a strong deformation retract of the poset of pseudo-sweeps~\cite{BilleraKapranovSturmfels1994}. Their proof uses strongly the geometry of the fiber polytope construction. Bj\"orner~\cite{Bjorner1992} and Athanasiadis, Edelman, and Reiner~\cite{AER2000}
found combinatorial proofs that extend to general oriented matroids,
but only give the homotopy type. Nevertheless, Bj\"orner claims that it is
\emph{``undoubtedly true''} that even for unrealizable oriented matroids there must be a sphere to which the poset of pseudo-sweeps retracts~\cite[below Thm.~2]{Bjorner1992}. However, there were no explicit candidates for these spheres. For oriented matroids that are little oriented matroids, we show in \Cref{thm:retract} that any of the associated sweep oriented matroids can play this role. That is, that the poset of non-trivial sweeps (which is a sphere) is a strong deformation retract of the poset of non-trivial pseudo-sweeps of the little oriented matroid.
This highlights the fact that sweep oriented matroids should be seen as combinatorial analogues of monotone path polytopes of zonotopes; that is, sweep polytopes.
Unfortunately, the existence of oriented matroids that are not little oriented matroids leaves some cases where Bj\"orner's observation remains open.

In \Cref{sec:sweepacycloids} we present a further generalization of sweep oriented matroids in terms of \defn{allowable graphs of permutations}, 
which are closer to the original formulation of allowable sequences.
Allowable graphs of permutations are graphs whose vertex sets are sets of permutations closed under taking reverses
in which
every pair of permutations is connected through a sequence of permutations fulfilling conditions \ref{it:AllowSeq1} and \ref{it:AllowSeq2} above (plus some technical conditions when the moves are not simple). In the simple case, these are antipodal isometric subgraphs of the permutahedron.
Translating back to sign-vectors, we obtain \defn{sweep acycloids} (\Cref{thm:sweeppermutationsetisacycloid}), which have the structure of acycloids~\cite{Handa90}, also known as antipodal partial cubes~\cite{FH93}. 
Again, sweep acycloids (and thus allowable graphs of permutations) of rank~$2$ are equivalent to allowable sequences. Not every acycloid is an oriented matroid~\cite[Sec.~7]{Handa93}, but there are characterizations of those that are~\cite{Handa93,daSilva95,KM20}.
Since sweep acycloids that are oriented matroids are sweep oriented matroids (\Cref{cor:sweepacycloidmatroidissweepmatroid}), these give alternative characterizations of sweep oriented matroids in terms of allowable graphs of permutations (\Cref{cor:characterizations}). So far we could not find any example of a sweep acycloid that is not a sweep oriented matroid, and we leave this question as an open problem.

\subsection{A note concerning the terminology}\label{sec:terminology}

The terms \emph{sweep} and \emph{sweeping} had already been used in the oriented matroids literature in the context of \defn{topological sweepings} of affine oriented matroids and pseudo-hyperplane arrangements. These concepts should not be confused with the notions that we introduce in this paper. 

The two colliding terminologies arise from the two classical dual geometric representations of realizable oriented matroids; namely, point configurations and hyperplane arrangements. Both give rise to a natural definition of \emph{sweep} that generalizes to non-realizable matroids.

On the one hand, our definition of \emph{sweep} is meant to model sweeps of point configurations by parallel hyperplanes. Such a sweep induces an ordering of the points, which are the elements of the underlying oriented matroid. 
When this picture is polarized, the point configuration gives rise to a hyperplane arrangement, but the collection of sweeping hyperplanes becomes a point that travels in a linear direction (the associated sweep permutation records the order in which the point crosses the hyperplanes). This is the formulation studied by Edelman~\cite{Edelman2000} and Stanley~\cite{Stan15}.

On the other hand, one can consider sweeps of hyperplane arrangements by parallel hyperplanes. Such a sweep induces an ordering of the vertices of the arrangement, which are the cocircuits of the underlying oriented matroid. This is the point of view of the literature on \defn{topological sweepings} of pseudo-hyperplane arrangements and oriented matroids (see, for example, \cite[p.172]{BLSWZ99}, \cite{EG89}, \cite{EOS86}, \cite{Hochstattler2016} and \cite{FelsnerWeil2001}), which concerns mostly the rank~$3$ case (pseudoline arrangements).

In rank~$3$, the two notions are strongly related. Indeed, the allowable sequence of a planar point configuration (which is a collection of sweeps in our terminology), can be interpreted as a topological sweep of the dual arrangement of lines. This correspondence exists in rank~$3$ but completely fails in higher rank, as it only works because in an oriented matroid of rank~$3$ the lines (flats of rank~$2$) coincide with the hyperplanes (flats of corank~$1$).

It is worth to note that in this second setup there exist other approaches to generalize allowable sequences to higher dimensions. For example, the \defn{signotopes} described in~\cite{FelsnerWeil2001} (see also~\cite{Felsner2004}). These are strongly related to higher Bruhat orders~\cite{ManinSchechtman1989} and single-element extensions of cyclic hyperplane arrangements~\cite{FelsnerZiegler2001,Ziegler1993}. 
However, as these generalizations are meant to model (topological) sweeps of hyperplane arrangements with a (pseudo) hyperplane, they do not cover the spherical complexes that Goodman and Pollack alluded to in~\cite{GP93} as the natural way to generalize allowable sequences to higher dimensions.

\subsection{Structure of this document}

This paper gravitates around the concept of sweep oriented matroid, which lies in the intersection of the theories of allowable sequences, valid order arrangements, and the generalized Baues problem for cellular strings. Our hope is to provide a unified reference that reflects all these connections. To this end, we give a broad overview of the topic, as we expect readers with diverse backgrounds and motivations to be interested in different aspects. In particular, most of the sections can be read independently.

\Cref{sec:sweepsandpolytopes} serves as an introduction and focuses in the realizable case. We present polytopal constructions that serve as motivation for the upcoming definitions.
Sweep oriented matroids are defined in \Cref{sec:sweeporientedmatroids}. In \Cref{sec:bigandlittle} we show how the structural results on allowable sequences from~\cite{BLSWZ99} generalize to sweep oriented matroids of arbitrary rank. \Cref{sec:flats} demonstrates that the results in~\cite{Edelman2000,Stan15} do not require realizability. \Cref{sec:pseudosweeps} depicts sweep oriented matroids as highlighted spheres inside the poset of cellular strings of oriented matroids whose existence was conjectured by~\cite{Bjorner1992}.
A presentation in terms of permutations, akin to Goodman and Pollack's original formulation of allowable sequences~\cite{GP93}, is given in \Cref{sec:sweepacycloids} under the name of allowable graphs of permutations.

We end by discussing some open problems and further directions of research in \Cref{sec:further}.

\section{Sweeps and sweep polytopes}\label{sec:sweepsandpolytopes}
\subsection{Sweeps of point configurations}

For any integer~$n$, we use \defn{$\ivl$} to denote the set $\{1, \ldots, n\}$, \defn{$\mathfrak{S}_n$} to denote the set of all permutations of~$\ivl$,
and \defn{$\ipairs$}$=\set{(i,j)}{1\leq i < j \leq n}$ to denote the set of non-repeating sorted pairs of elements of~$\ivl$. 
An \defn{ordered partition} of $\ivl$ is an ordered collection of non-empty disjoint subsets $(I_1, \ldots, I_l)$ whose union is~$\ivl$. Ordered partitions where all parts are singletons are identified with permutations. They are the maximal elements in the \defn{refinement order}: we say that $J=(J_1, \ldots , J_l)$ refines $I=(I_1, \ldots , I_k)$, noted
 $J\succeq I$,
if each $I_i$ is the union of some consecutive $J_j$'s. In some proofs, it will be more comfortable to think of an ordered partition $I$ as the surjection $\sur$ from $\ivl$ to $\ivl[l]$ such that $I_k=\sur^{-1}(\{k\})$ for all $1\leq k \leq l$. Note that for a permutation $\sigma$, the ordered partition $I=(\{\sigma(1)\}, \ldots, \{\sigma(n)\})$ corresponds to the bijection~$\sur=\sigma^{-1}$.

We always consider $\RR^d$ as an Euclidean space, equipped with the usual orthogonal scalar product~$\sprod{\cdot}{\cdot}$. A \defn{point configuration} is an ordered sequence $\pc=(\p a_1,\dots,\p a_n)\in\RR^{d\times \ivl}$ of points in~$\RR^d$ indexed by $\ivl$. We do not require the points to be distinct, although it will be often convenient to make this simplification.
For $\p u\in\RR^d$, consider the linear form $\lfv[\p u]:\RR^d\to\RR$ sending $\p x$~to~$\sprod{\p u}{\p x}$. The \defn{sweep} of $\pc$ associated to $\p u$ is the ordered partition $I^{\p u}=(I_1, \ldots, I_l)$ of $\ivl$ that verifies $\sprod{\p u}{\p a_i}=\sprod{\p u}{\p a_j}$ for all $i,j$ in a same part $I_k$, and $\sprod{\p u}{\p a_i} < \sprod{\p u}{\p a_j}$ if $i\in I_r,\, j\in I_s$ with $r<s$.  In particular, $\sprod{\p u}{\p a_i} \leq  \sprod{\p u}{\p a_j}$ if and only if $\sur[{I^{\p u}}](i) \leq  \sur[{I^{\p u} }](j)$. Note that the partition associated to the linear form $\zero$ is the \defn{trivial sweep}~$(\ivl)$.

The \defn{poset of sweeps} of $\pc$, denoted~\defn{$\Scom$}, is the set of all sweeps ordered by refinement. 
Its maximal elements are permutations whenever~$\pc$ does not contain repeated points. 
We will often assume that this is the case, as we can always identify repeated points. Under this assumption, we denote by \defn{$\PsetA$}~$\subseteq \Sym$ the set of its maximal elements, the \defn{sweep permutations} of~$\pc$. If there are repeated points, we will still call the maximal elements \defn{sweep permutations} for brevity.

Sweeps induce an equivalence relation on $\RR^d$, where $\p u\sim \p v$ if they give the same sweep. Its equivalence classes are the cells of the polyhedral fan induced by the \defn{sweep hyperplane arrangement} \defn{$\VO$}; the arrangement of the linear hyperplanes 
$\set{\p u \in~\RR^d}{\sprod{\p u}{\p a_i}=\sprod{\p u}{\p a_j}}$ for all $(i, j) \in \ipairs$. 
Note that the face poset of $\VO$ is isomorphic to the poset $\Scom$, with a bijection that sends each cell $\cC$ of $\VO$ to the sweep $I$ in $\Scom$ that verifies that the relative interior of $\cC$ is $\set{\p u \in \RR^d}{I^{\p u}=I}$. In particular, the cones of dimension $d$ of $\VO$ are indexed by the sweep permutations in $\PsetA$.

We will see in \cref{sec:sweep polytopes} that $\VO$ is the normal fan of a polytope: the \defn{sweep polytope} of~$\pc$, denoted by~\defn{$\Sp$}. Thus, 
the poset of sweeps~$\Scom$ enlarged with a top element is isomorphic to the poset opposite to the face lattice of~$\Sp$, and is in particular a lattice.
This provides a natural labeling of the faces of $\Sp$ by sweeps. In particular, the vertices of~$\Sp$ are labeled by the sweep permutations in~$\PsetA$.

The identification of sweeps with faces of~$\Sp$ reflects the inherent topological structure of the poset of sweeps. This can be made precise in terms of its order complex. The \defn{order complex}~$\OC[P]$ of a poset~$P$ is the simplicial complex whose simplices are the chains of~$P$, see~\cite{Bjorner1995} or~\cite[Sec.~4.7]{BLSWZ99} for some background. In our case, the order complex of~$\Scom\ssm (\ivl)$, the poset of sweeps without the trivial sweep, is just the barycentric subdivision of the boundary of~$\Sp$. We will implicitly identify~$\Scom$ with~$\OC[{\Scom \ssm (\ivl)}]$ whenever we make topological statements about posets of sweeps.

\subsection{Examples}
Before providing constructions for this polytope, we will present two particular examples.

\subsubsection{The simplex and the permutahedron}
If $\pc_n$ is the set of vertices of a standard $(n-1)$-simplex~$\simp[n-1]$, i.e.\ the points $\p a_i$ are the canonical basis vectors $\p e_i$ in $\RR^n$, then $\VO[\pc_n]$ is the \defn{braid arrangement~$\braid$} consisting of the hyperplanes $\set{\p u}{\p u_j-\p u_i=0}$ for all $1\leq i < j \leq n$, the set of sweep permutations is the whole symmetric group $\PsetA[\pc_n]=\Sym$, and the poset of sweeps $\Scom[\pc_n]$ is the poset of all ordered partitions of $\ivl$. Likewise for any set $\pc$ of affinely independent points, up to affine transformation of the braid arrangement.

The braid arrangement~$\braid$ is the normal fan of a polytope, the \defn{$n$-permutahedron~$\Pn$}. It is usually defined as the convex hull of the points $(\sigma(1), \ldots , \sigma(n))\in \RR^n$ for all $\sigma \in \Sym$ (see~\cite[Ex~0.10]{Ziegler1995} or \cite[Ex.~2.2.5]{BLSWZ99}). Thus, it lives in the $(n-1)$-dimensional affine subspace of the sum of coordinates constant equal to~$\tfrac{n(n+1)}{2}$. It can be described as the zonotope: 
\begin{equation}\label{eq:permutahedronzonotope}
\Pn = \tfrac{n+1}{2} \ones + \sum_{1 \leq i < j \leq n} \left[-\frac{\p e_i-\p e_j}{2}, \frac{\p e_i-\p e_j}{2}\right], 
\end{equation}
where $\ones=\sum_{i=1}^n \p e_i$ is the all-ones vector and $[\p p,\p q]\subset \RR^d$ denotes the segment between the points~$\p p$ and~$\p q$, see~\cite[Ex.~7.15]{Ziegler1995}.

\begin{figure}[htpb]
        \vspace{-.3cm}
        \includegraphics[width=\textwidth]{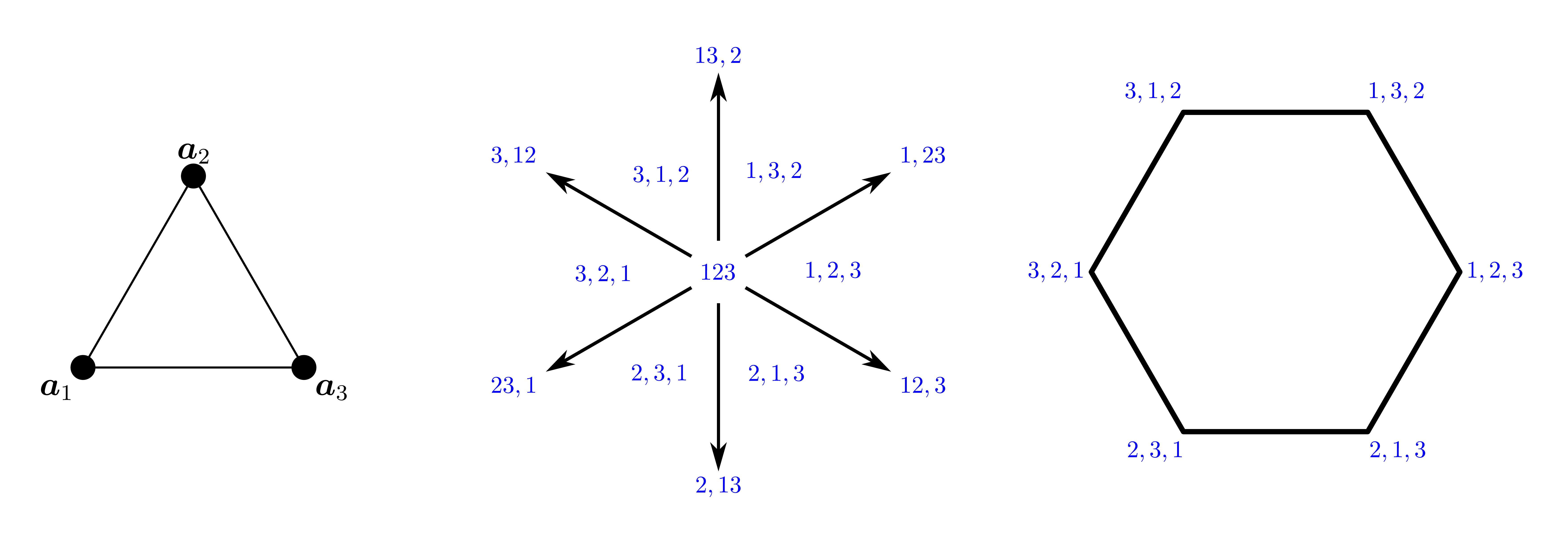}
        \vspace{-1cm}
    \caption{$\pc_3$, its sweep hyperplane arrangement~$\VO[\pc_3]=\braid[3]$ (modulo linearity),
and its sweep polytope $\Sp[{\pc_3}]=\cPn[3]$, the $3$-permutahedron, where each vertex is labeled by the corresponding sweep permutation of $\pc_3$. 
    }
    \label{fig:SweepsA3}
\end{figure}

\begin{figure}[htpb]
        \centering
\input{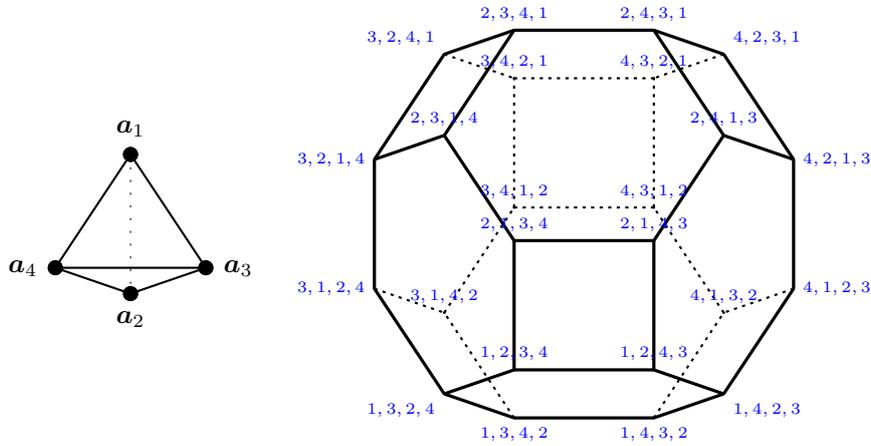}
\caption{$\pc_4$ and its sweep polytope $\Sp[{\pc_4}]=\cPn[4]$, the $4$-permutahedron. 
}\label{fig:SweepsA4}
\end{figure}

The sweep polytope~$\Sp$ associated to the standard simplex is the translation of~$\Pn$ centered at the origin. We will denote this translated permutahedron by~$\cPn$
\begin{equation}\label{eq:centeredpermutahedron}
\cPn=\sum_{1 \leq i < j \leq n} \left[-\frac{\p e_i-\p e_j}{2}, \frac{\p e_i-\p e_j}{2}\right] 
\end{equation}
to distinguish it from the standard realization.
See \cref{fig:SweepsA3,fig:SweepsA4} for the cases~$n=3,4$.

\subsubsection{The cross-polytope and the permutahedron of type~$B$}\label{sec:crosspolytope}

Let $\pc[B]_n$ be the set of vertices of the cross-polytope $\cpol$, that is, the set of standard basis vectors of $\RR^n$ and their opposites. It is convenient to index the points by $[\pm n]=\{-n,\dots,-1,1,\dots,n\}$: $\pc[B]_n=\{  \p b_{-n}=-\p e_n , \ldots, \p b_{- 1}=-\p e_1,  \p b_1=\p e_1, \ldots, \p b_n=\p e_n\}$. 
Then the sweep permutations of $\pc[B]_n$ are the  centrally symmetric permutations of $\Sym[{[\pm n]}]$, which satisfy $\sigma(-i)=-\sigma(i)$ for all $i\in [\pm n]$. By symmetry, the first half determines the whole permutation. This way, they can be represented by signed permutations of $\ivl$, where 
$-k$ is denoted by $\overline k$. We use this notation in \cref{fig:SweepsB2,fig:SweepsB3}.

\enlargethispage{1cm}
\begin{figure}[htpb]
        \vspace{-.3cm}
        \includegraphics[width=\textwidth]{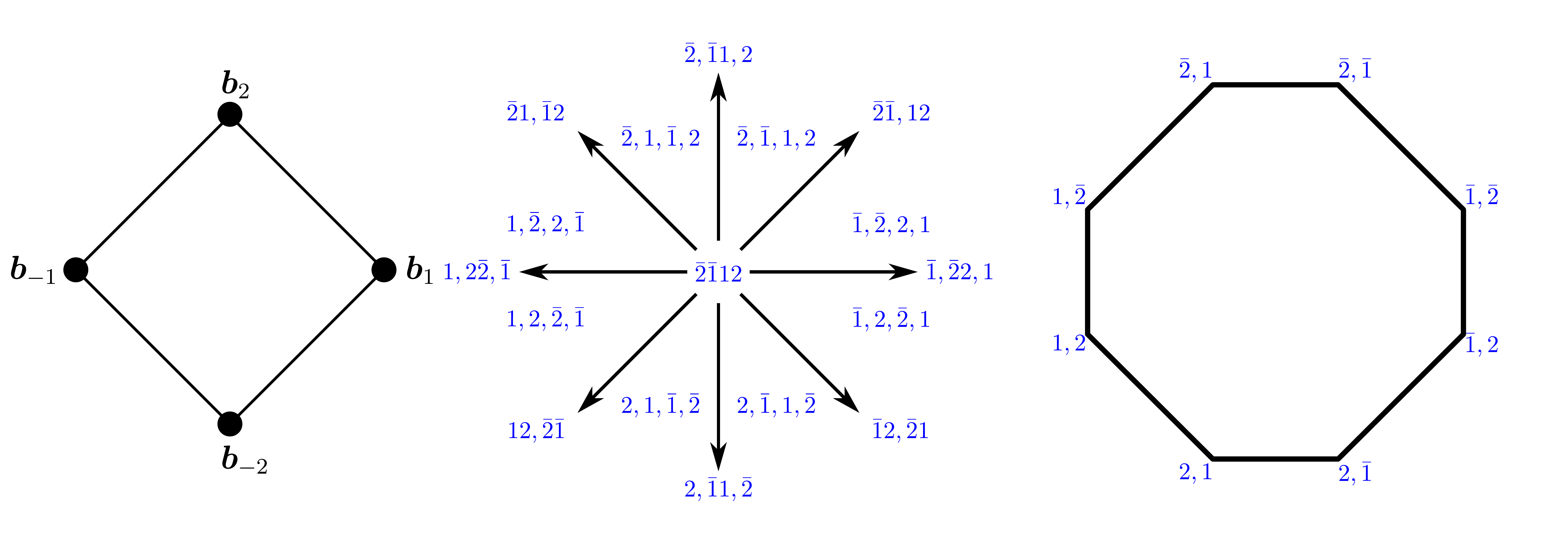}
        \vspace{-1cm}
    \caption{$\pc[B]_2$, its sweep hyperplane arrangement~$\VO[{\pc[B]_2}]$, 
    and its sweep polytope~$\Sp[{\pc[B]_2}]$. }
    \label{fig:SweepsB2}
\end{figure}

\begin{figure}[htpb]
        \centering
\input{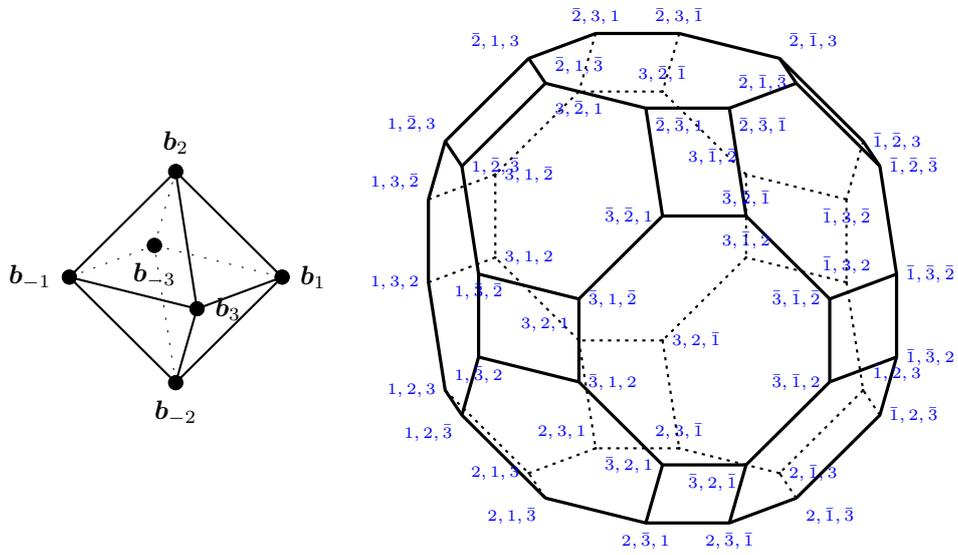}
\caption{$\pc[B]_3$ and its sweep polytope $\Sp[{\pc[B]_3}]$, the $3$-permutahedron of type~$B$.}\label{fig:SweepsB3}
\end{figure}

They are the elements of the Coxeter group of type~$B$, also called hyperoctahedral group. See \cite[Section 8.1]{BB05} for more details on the combinatorics of this group. 
The sweep hyperplane arrangement $\VO[{\pc[B]_n}]$ is the Coxeter arrangement of type~$B$, which consists of the hyperplanes  $\set{\p u\in \RR^n}{\p u_i\pm \p u_j=0}$ for all $1\leq i < j\leq n$ and $\set{\p u\in \RR^n}{\p u_i=0}$ for all $1\leq i \leq n$. The sweeps are the centrally symmetric ordered partitions of $[\pm n]$.
This complex is known as the Coxeter complex of type~$B$, see~\cite[Sec.~2.3(c)]{BLSWZ99}. See~\cref{fig:SweepsB2} for an example.

The associated sweep polytope is the Coxeter permutahedron of type~$B$, also known as the Coxeterhedron of type~$B$~\cite{RZ94}. See \cref{fig:SweepsB2,fig:SweepsB3} for pictures in dimensions $2$ and $3$.

\subsubsection{Sweeping with polynomial functions}

Sweep polytopes can also be used to model sweeps of a point configuration $\pc=(\p a_1, \ldots, \p a_n) \in \RR^{d\times \ivl}$ by polynomial functions $\polf \in \RR[x_1, \ldots, x_d]$ of bounded degree. The \defn{polynomial sweep} of $\pc$ associated to~$\polf$ is the ordered partition of~$\ivl$ induced by the ordered level sets of~$\polf$ on~$\pc$.

Let $\monomials$ be the set of monomials of degree at most~$D$ on variables $x_1,\dots,x_d$. There are $|\monomials|=\binom{D+d}{D}$ elements in $\monomials$.
For a point $\p v=(v_1, \ldots, v_d) \in \RR^d$ and a monomial $M\in \monomials$, denote by $M(\p v) \in \RR$ the evaluation of $M$ on the values $x_1=v_1, \ldots, x_d=v_d$.
The \defn{Veronese mapping} is defined by the map
\[ \veronese :
\begin{cases}
\RR^d &\to \RR^{\monomials} \\
\p v &\mapsto \left( M(\p v) \right)_{M\in \monomials}.
\end{cases}\]

Then, the polynomial sweep of $\pc$ induced by the polynomial $\polf=\sum_{M\in \monomials} c_M M$ exactly corresponds to the sweep of $\veronese(\pc)$ induced by the linear functional $\sprod{\p c}{\cdot}$ for $\p c=(c_M)_{M\in \monomials} \in \RR^{\monomials}$. In particular,  the poset of sweeps of $\veronese(\pc)$ coincides with the poset of polynomial sweeps of $\pc$ induced by polynomials of degre at most~$D$.
Note that if $d=1$, the image $\veronese(\pc)$ is a standard cyclic polytope of dimension~$D$ with $n$ vertices.

Variants of the Veronese mapping can be used for particular families of polynomial sweeps. For example, the embedding 
\[(v_1,\dots,v_d)\mapsto (v_1,\dots,v_d, v_1^2+\cdots+v_d^2)\]
onto the paraboloid models sweeps by families of concentric spheres.

\subsection{Constructions for sweep polytopes}\label{sec:sweep polytopes}

In what follows, we describe three approaches to construct the {sweep polytope}~$\Sp$. Recall that~$\Sp$ is a polytope whose normal fan coincides with the sweep hyperplane arrangement~$\VO$, and whose face poset is opposite to the poset of sweeps~$\Scom$.

\subsubsection{As a zonotope}
The most direct realization is as the Minkowski sum of the segments with directions the differences between the points of the configuration, which is (a translation of) the presentation of sweep polytopes given in~\cite{GritzmannSturmfels1993} (under the name of \defn{shellotopes}). 

\begin{definition}\label{def:sweep polytope}
The \defn{sweep polytope $\Sp$} associated to the configuration $\pc =(\p a_1,\dots,\p a_n)\in\RR^{d\times \ivl}$ is the zonotope:
\[ \Sp = \sum_{1\leq i < j \leq n}\left[-\frac{\p a_i-\p a_j}{2} ,\frac{\p a_i-\p a_j}{2}\right] \subset \RR^d.\]

\end{definition}

The normal fan of a zonotope is the arrangement of the hyperplanes orthogonal to its generators, see for example \cite[Sec.~2]{Ziegler1989} and \cite[Thm.~7.16]{Ziegler1995}. Applied to sweep polytopes, we directly get:

\begin{proposition}
The normal fan of $\Sp$ is the hyperplane arrangement $\VO$.
\end{proposition}

\subsubsection{As a projection of the permutahedron}\label{sec:asPermutahedralShadows}
Our second incarnation is as a projection of the (centered) permutahedron~$\cPn$. For a configuration~$\pc$ of $n$ points~$\p a_1, \ldots , \p a_n$ in~$\RR^d$, let $\lm$ be the linear map
\begin{align}
\lm : \RR^n & \to  \RR^d \label{eq:phiA}\\ 
 \p e_i & \mapsto  \p a_i.\notag
\end{align}
Then it follows from \Cref{def:sweep polytope} and the description of $\cPn$ in~\eqref{eq:centeredpermutahedron} that:
\begin{proposition}\label{prop:sweep polytopeprojection}
\(\Sp = \lm(\cPn).\) 
\end{proposition}

Conversely, all affine images of permutahedra are sweep polytopes, up to translation. This provides a combinatorial interpretation, in terms of sweeps, of the face lattice of any affine projection of a permutahedron (a \defn{permutahedral shadow}).

\begin{corollary}\label{cor:permutahedralshadows}
 Let $\lm[]:\RR^n\to \RR^d$ be a linear map, then $\lm[](\cPn)$ is the sweep polytope of
 the point configuration $\lm[](\p e_1),\dots,\lm[](\p e_n)$.
\end{corollary}

Note that, given a linear map from~$\cPn$ to~$\RR^d$, there is a $d$-dimensional family of ways to extend it to a linear map from~$\RR^n$ to~$\RR^d$. This amounts to the fact that point configurations related by a translation give rise to the same sweep polytope.

\begin{remark}\label{rmk:hypersimplices}
 \Cref{prop:sweep polytopeprojection} follows from the fact that Minkowski sums and linear projections commute.
 This can be exploited also with other decompositions of the permutahedron. For example, the permutahedron~$\Pn$ can be written as the Minkowski sum of the hypersimplices~$\Hs=\set{\p x\in [0,1]^n}{\sum \p x_i=k}$ with $k$ ranging from~$1$ to~$n-1$ (see for example~\cite{Postnikov2009}). Therefore, any sweep polytope can be expressed as a Minkowski sum of projections of hypersimplices. Projections of hypersimplices are studied under the name of \defn{$k$-set polytopes}~\cite{AndrzejakWelzl2003,EdelsbrunnerValtrWelzl1997}, which (up to homothety) can be described as the convex hull of the barycenters of all $k$-subsets of~$\pc$, see~\cite{MartinezSandovalPadrol2020}. The sweep polytope of $\pc$ is thus the Minkowski sum of its $k$-set polytopes, up to translation and homothety. In particular, because $\conv(\pc)=\lm(\Hs[n][1])$, this shows that $\conv(\pc)$ is a Minkowski summand of~$\Sp$.
 See \cref{fig:ksetpolytopes} for an example. Another point of view on this Minkowski decomposition will be discussed in \Cref{rmk:fiberksets}.
\end{remark}

\begin{figure}[htpb]
\input{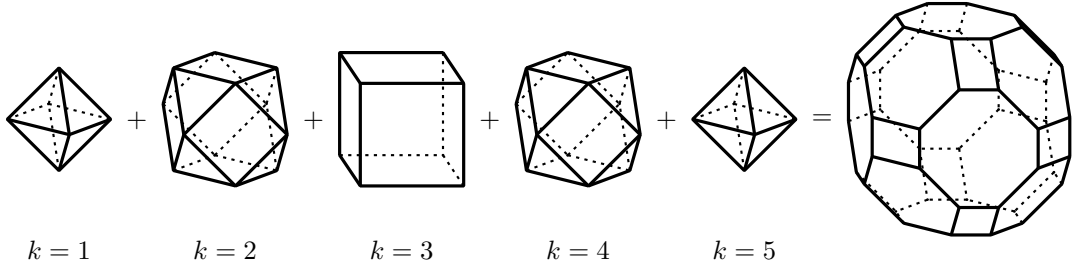}
    \caption{The sweep polytope $\Sp[{\pc[B]_3}]=\cPn[3]$ as a Minkowski sum of the $k$-set polytopes of $\pc[B]_3$ for $k=1,\dots,5$.}
    \label{fig:ksetpolytopes}
\end{figure}

\subsubsection{As a monotone path polytope} \label{sec:asFiberPolytope}

Fiber polytopes are certain polytopes associated to polytope projections. This construction was introduced by Billera and Sturmfels in~\cite{BS92}, generalizing the theory of secondary polytopes in a unified way that encompasses concepts such as monotone path polytopes, zonotopal tiling polytopes and secondary polytopes. We refer to \cite[Lec.~9]{Ziegler1995} and \cite[Sec.~9.1]{DRS10} for gentle introductions to the topic.

Consider polytopes $\pol$ and~$\pol[Q]$ related by a linear surjection $\pi : \pol[P] \to \pol[Q]$. The fibers of $\pi$ over $\pol[Q]$ form a \defn{polytope bundle} $\p y\in \pol[Q]\mapsto \pi^{-1}(\{\p y\})$ whose \emph{Minkowski integral}, after some normalization, is the \defn{fiber polytope $\fib$}:
\[ \fib = \frac{1}{\vol(\pol[Q])}\int_{\pol[Q]}\pi^{-1}(\{\p y\}) d\p y. \]
Fiber polytopes can also be described as a finite Minkowski sum. Namely,
\[ \fib = \frac{1}{\vol(\pol[Q])}\sum_{\pol[C] \in \Gamma(\pol, \pi)} \vol(\pol[C])\ \pi^{-1}(\{\p b_{\pol[C]}\}), \]
where $\Gamma(\pol, \pi)$ is the set of \defn{chambers}: the subsets of $\pol[Q]$ of the form \[\pol[C]_{\p y}~=~\bigcap\limits_{\underset{\p y \in \pi(\pol[F])}{\pol[F]\text{ face of }\pol}} \pi(\pol[F])\]
for $\p y\in \pol[Q]$; and $\p b_{\pol[C]}$ is the barycenter of the chamber $\pol[C]$. 

Note that $\fib$ lies in the fiber over the barycenter of $\pol[Q]$: $\fib \subset \pi^{-1}\left(\frac{1}{\vol(\pol[Q])}\int_{\pol[Q]}\p y\ d\p y \right)$.

An important feature of fiber polytopes is that their face lattice is isomorphic to the poset of $\pi$-coherent subdivisions of $Q$ (ordered by refinement), which are subdivisions of $\pol[Q]$ composed of images of faces of $\pol$ that are \defn{coherently induced} by the map~$\pi$. We refer to the aforementioned sources for the details in the definitions. 
We are particularly interested in a special case of fiber polytopes: monotone path polytopes. 
They will give a new interpretation of sweep polytopes and provide motivation for the definition of pseudo-sweeps, that will be further explored in \cref{sec:pseudosweeps}.

If $\pol[Q]$ is one dimensional and $\pol \subset \RR^n$, then $\pi:\pol\to \pol[Q]$ is a linear form defined by a vector $\p u\in \RR^n$ via $\pi(\p x)=\sprod{\p u}{\p x}$. For simplicity, assume that $\pi$ is generic in the sense that it is not constant along any edge of $\pol$, and let $\p p_m$ and $\p p_M$ be the minimal and maximal vertices of $\pol$ with respect to $\pi$. A \defn{$\pi$-monotone path} is a path from $\p p_m$ to $\p p_M$ composed of edges of~$\pol$ along which $\pi$ is always increasing. One way to obtain $\pi$-monotone paths is to consider some generic vector $\p w$ orthogonal to~$\p u$ and consider the sequence of vertices of $\pol$ that are extreme in the direction $\p w+\lambda \p u$ as $\lambda $ ranges from $-\infty$ to~$\infty$ (see \cref{fig:zonotopeandfibers}). These paths induce the finest $\pi$-coherent subdivisions of $\pol[Q]$, and are known as \defn{parametric simplex paths} in linear programming, where they play an important role as they are the paths followed by the shadow-vertex simplex method~\cite{Borgwardt1987,GassSaaty1955}. 

 More generally, a \defn{cellular string} on $\pol$ with respect to $\pi$ is a sequence of faces $\pol[F]_1,\dots,\pol[F]_k$ of~$\pol$ of dimension at least~$1$ such that $\p p_m\in \pol[F]_1$, $\p p_M\in \pol[F]_k$, and every two adjacent faces $\pol[F]_i,\pol[F]_{i+1}$ meet at a vertex $\p p_i$ such that $\pi(\p x)\leq \pi(\p p_i)\leq \pi(\p y)$ for each $\p x\in \pol[F]_i$ and $\p y\in \pol[F]_{i+1}$. Such a cellular string is $\pi$-coherent if there is some (not-necessarily generic) vector $\p w$ orthogonal to~$\p u$ such that these are the maximal faces of~$\pol$ maximized in a direction of the form $\p w+\lambda \p u$.
 The fiber polytope $\fib$ is called the \defn{monotone path polytope} of $\pol$ and $\pi$. Its vertices are in one-to-one correspondence with the parametric $\pi$-monotone paths of~$\pol$, and its faces are in correspondence with the $\pi$-coherent cellular strings.

\begin{example}[{{\cite[Ex.~5.4]{BS92}, see also \cite[Ex.~9.8]{Ziegler1995}}}]\label{ex:permutahedronfiber}
 Let $\cube = [-1, 1]^n$ be the $n$-dimensional $\pm 1$-hypercube, and let $s: \RR^n\to \RR$ be the linear form that sums the coordinates, i.e.\ the form $s=\lfv[\ones]$ induced by the all-ones vector. Then the fiber polytope~$\fib[\cube][s]=\frac{2}{n} \cPn$ is (homothetic to) the (centered) permutahedron~$\cPn$, and $\fib[\frac{n}{2}\cube][s]=\cPn$.

\end{example}

The following central property of fiber polytopes will be key for our purposes.

\begin{lemma}[{{\cite[Lem.~2.3]{BS92}}}]\label{lem:proj_fiber_pol}
Let $\RR^n\xrightarrow{\theta}\RR^m\xrightarrow{\pi}\RR^d$ be linear maps, and $\pol\subset \RR^n$ a polytope. Then $\fib[\theta(\pol)][\pi]=\theta(\fib[\pol][\pi\circ\theta])$.
\end{lemma}

We need some extra notation. Let~$\pc=(\p a_1,\dots,\p a_n)\in\RR^{d\times \ivl}$ be a point configuration, and consider its \defn{homogenization} $\h \pc=(\h {\p a}_1,\dots,\h {\p a}_n)\in\RR^{(d+1)\times \ivl}$ consisting of the vectors $\h {\p a}_i=(\p a_i,1)$. We define the zonotope~\defn{$\Z$} associated to~$\pc$ as the following Minkowski sum of centrally symmetric segments:
\[\Z = \sum_{i=1}^n [-\h {\p a}_i, \h{\p a}_i].\]
Let $\defn{\height} :\RR^{d+1}  \to \RR$ denote the map that returns the last coordinate of a point, that we call its \defn{height}.

This gives us another point of view on sweep polytopes.

\begin{proposition}\label{prop:spasmonotonepath}
For any point configuration~$\pc$ we have \[\fib[\tfrac{n}{2}\Z][\height]=\Sp\times\{0\},\] and hence $\Sp$ is affinely isomorphic to the monotone path polytope $\fib[\Z][\height]$.
\end{proposition}

\begin{proof}
The projection $\lm[\h \pc] :\mathbb{R}^n  \to \mathbb{R}^{d+1}$ that maps $\p e_i$ to~$\h {\p a}_i=(\p a_i, 1)$ is such that $\Z=\lm[\h \pc](\cube)$ and $s=\height\circ \lm[\h \pc]$, where $s$ is the linear form that sums the coordinates defined in \cref{ex:permutahedronfiber}. 
Hence, by \cref{lem:proj_fiber_pol} and \cref{ex:permutahedronfiber} we have $\fib[\Z][\height]=\lm[\h \pc](\fib[\cube][s])=\lm[\h \pc](\frac{2}{n} \cPn)$. 
Now, $\cPn$ lies in $s^{-1}\left(\frac{1}{\vol(\cube)}\int_{\cube} {\p y} d\p y\right)=s^{-1}(\zeros)$, and thus $\lm[\h \pc](\frac{2}{n} \cPn)$ lies in the kernel of $\height$, which means that $\fib[\Z][\height]=\frac{2}{n}\lm(\cPn)\times \{0\}$. Finally, by \cref{prop:sweep polytopeprojection}, we have \(\lm(\cPn) = \Sp\), and therefore $\fib[\Z][\height]=\frac{2}{n}\Sp\times \{0\}$.
\end{proof}

\begin{remark}\label{rmk:fiberksets}
If we intersect $\Z$ with the hyperplane of height equal to $-n+2$, we obtain
$$\conv(-\sum_{i=1}^n \h{\p a_i} +2\h{\p a_j}, j\in [n])=\conv(-\sum_{i=1}^n {\p a_i} + 2\pc)\times \{-n+2\},$$
which is an embedding of a dilation of the convex hull of $\pc$ in $\RR^{d+1}$. 
Similarly, for any $k\in [n]$ the slice at height $-n+2k$ is an embedding of a dilation of the projection of the hypersimplex~$\Hs$ under the map~$\lm[\pc]$. This is the $k$-set polytope of~$\pc$, see~\cref{rmk:hypersimplices}. The fiber polytope realization reflects the decomposition of the sweep polytope as a sum of $k$-set polytopes.
\end{remark}

\begin{figure}[htpb]
 
\begin{tikzpicture}%
	[x={(0.1cm, 0.25cm)},
	y={(-0.0cm, 1cm)},
	z={(1cm, -0.cm)},
	scale=1.7,
	back/.style={dotted, thick},
	edge/.style={color=black, very thick},
	facet/.style={fill=white,fill opacity=0},
	vertex/.style={inner sep=0pt,circle,black,fill=black,thick,anchor=base},
	etiq/.style={prefix after command= {\pgfextra{\tikzset{every label/.style={font=\tiny,text=blue}}}}},
    fiber/.style={fill=green,fill opacity=0.5},
    on each segment/.style={
        decorate,
        decoration={
        show path construction,
        moveto code={},
        lineto code={
            \path [#1]
            (\tikzinputsegmentfirst) -- (\tikzinputsegmentlast);
        },
        curveto code={
            \path [#1] (\tikzinputsegmentfirst)
            .. controls
            (\tikzinputsegmentsupporta) and (\tikzinputsegmentsupportb)
            ..
            (\tikzinputsegmentlast);
        },
        closepath code={
            \path [#1]
            (\tikzinputsegmentfirst) -- (\tikzinputsegmentlast);
        },
        },
    },
    mid arrow/.style={postaction={decorate,decoration={
        markings,
        mark=at position .5 with {\arrow{>}}
      }}},
    mp/.style={color=red, very thick,round cap-round cap,postaction={on each segment={mid arrow}}}
	]
\coordinate (-1.00000, -1.00000, 2.00000) at (-1.00000, -1.00000, 2.00000);
\coordinate (-1.00000, 0.00000, 1.00000) at (-1.00000, 0.00000, 1.00000);
\coordinate (-1.00000, 0.00000, 3.00000) at (-1.00000, 0.00000, 3.00000);
\coordinate (-1.00000, 1.00000, 2.00000) at (-1.00000, 1.00000, 2.00000);
\coordinate (0.00000, -1.00000, 1.00000) at (0.00000, -1.00000, 1.00000);
\coordinate (0.00000, -1.00000, 3.00000) at (0.00000, -1.00000, 3.00000);
\coordinate (0.00000, 0.00000, 0.00000) at (0.00000, 0.00000, 0.00000);
\coordinate (1.00000, 1.00000, 2.00000) at (1.00000, 1.00000, 2.00000);
\coordinate (0.00000, 0.00000, 4.00000) at (0.00000, 0.00000, 4.00000);
\coordinate (0.00000, 1.00000, 1.00000) at (0.00000, 1.00000, 1.00000);
\coordinate (0.00000, 1.00000, 3.00000) at (0.00000, 1.00000, 3.00000);
\coordinate (1.00000, -1.00000, 2.00000) at (1.00000, -1.00000, 2.00000);
\coordinate (1.00000, 0.00000, 1.00000) at (1.00000, 0.00000, 1.00000);
\coordinate (1.00000, 0.00000, 3.00000) at (1.00000, 0.00000, 3.00000);
\draw[edge,back] (0.00000, -1.00000, 1.00000) -- (1.00000, -1.00000, 2.00000);
\draw[edge,back] (0.00000, -1.00000, 3.00000) -- (1.00000, -1.00000, 2.00000);
\draw[edge,back] (0.00000, 0.00000, 0.00000) -- (1.00000, 0.00000, 1.00000);
\draw[edge,back] (1.00000, 1.00000, 2.00000) -- (1.00000, 0.00000, 1.00000);
\draw[edge,back] (1.00000, 1.00000, 2.00000) -- (1.00000, 0.00000, 3.00000);
\draw[edge,back] (0.00000, 0.00000, 4.00000) -- (1.00000, 0.00000, 3.00000);
\draw[edge,back] (1.00000, -1.00000, 2.00000) -- (1.00000, 0.00000, 1.00000);
\draw[edge,back] (1.00000, -1.00000, 2.00000) -- (1.00000, 0.00000, 3.00000);
\node[vertex] at (1.00000, 0.00000, 3.00000)     {};
\node[vertex] at (1.00000, 0.00000, 1.00000)     {};
\node[vertex] at (1.00000, -1.00000, 2.00000)     {};
\fill[facet] (1.00000, 1.00000, 2.00000) -- (0.00000, 1.00000, 3.00000) -- (-1.00000, 1.00000, 2.00000) -- (0.00000, 1.00000, 1.00000) -- cycle {};
\fill[facet] (-1.00000, 1.00000, 2.00000) -- (-1.00000, 0.00000, 1.00000) -- (-1.00000, -1.00000, 2.00000) -- (-1.00000, 0.00000, 3.00000) -- cycle {};
\fill[facet] (0.00000, 1.00000, 1.00000) -- (-1.00000, 1.00000, 2.00000) -- (-1.00000, 0.00000, 1.00000) -- (0.00000, 0.00000, 0.00000) -- cycle {};
\fill[facet] (0.00000, 1.00000, 3.00000) -- (-1.00000, 1.00000, 2.00000) -- (-1.00000, 0.00000, 3.00000) -- (0.00000, 0.00000, 4.00000) -- cycle {};
\fill[facet] (0.00000, 0.00000, 4.00000) -- (-1.00000, 0.00000, 3.00000) -- (-1.00000, -1.00000, 2.00000) -- (0.00000, -1.00000, 3.00000) -- cycle {};
\fill[facet] (0.00000, 0.00000, 0.00000) -- (-1.00000, 0.00000, 1.00000) -- (-1.00000, -1.00000, 2.00000) -- (0.00000, -1.00000, 1.00000) -- cycle {};

\fill[fiber] (-1.00000, 0.00000, 1.00000) -- (0.00000, -1.00000, 1.00000) -- (1.00000, 0.00000, 1.00000) -- (0.00000, 1.00000, 1.00000) -- cycle {};
\fill[fiber] (-1.00000, -1.00000, 2.00000) -- (-1.00000, 1.00000, 2.00000) -- (1.00000, 1.00000, 2.00000) -- (1.00000, -1.00000, 2.00000) -- cycle {};
\fill[fiber] (-1.00000, 0.00000, 3.00000) -- (0.00000, -1.00000, 3.00000) -- (1.00000, 0.00000, 3.00000) -- (0.00000, 1.00000, 3.00000) -- cycle {};

\draw[edge] (-1.00000, -1.00000, 2.00000) -- (-1.00000, 0.00000, 1.00000);
\draw[edge] (-1.00000, -1.00000, 2.00000) -- (-1.00000, 0.00000, 3.00000);
\draw[edge] (-1.00000, -1.00000, 2.00000) -- (0.00000, -1.00000, 1.00000);
\draw[edge] (-1.00000, -1.00000, 2.00000) -- (0.00000, -1.00000, 3.00000);
\draw[edge] (-1.00000, 0.00000, 1.00000) -- (-1.00000, 1.00000, 2.00000);
\draw[edge] (-1.00000, 0.00000, 1.00000) -- (0.00000, 0.00000, 0.00000);
\draw[edge] (-1.00000, 0.00000, 3.00000) -- (-1.00000, 1.00000, 2.00000);
\draw[edge] (-1.00000, 0.00000, 3.00000) -- (0.00000, 0.00000, 4.00000);
\draw[edge] (-1.00000, 1.00000, 2.00000) -- (0.00000, 1.00000, 1.00000);
\draw[edge] (-1.00000, 1.00000, 2.00000) -- (0.00000, 1.00000, 3.00000);
\draw[edge] (0.00000, -1.00000, 1.00000) -- (0.00000, 0.00000, 0.00000);
\draw[edge] (0.00000, -1.00000, 3.00000) -- (0.00000, 0.00000, 4.00000);
\draw[edge] (0.00000, 0.00000, 0.00000) -- (0.00000, 1.00000, 1.00000);
\draw[edge] (1.00000, 1.00000, 2.00000) -- (0.00000, 1.00000, 1.00000);
\draw[edge] (1.00000, 1.00000, 2.00000) -- (0.00000, 1.00000, 3.00000);
\draw[edge] (0.00000, 0.00000, 4.00000) -- (0.00000, 1.00000, 3.00000);
\node[vertex] at (-1.00000, -1.00000, 2.00000)     {};
\node[vertex] at (-1.00000, 0.00000, 1.00000)     {};
\node[vertex] at (-1.00000, 0.00000, 3.00000)     {};
\node[vertex] at (-1.00000, 1.00000, 2.00000)     {};
\node[vertex] at (0.00000, -1.00000, 1.00000)     {};
\node[vertex] at (0.00000, -1.00000, 3.00000)     {};
\node[vertex] at (0.00000, 0.00000, 0.00000)     {};
\node[vertex] at (1.00000, 1.00000, 2.00000)     {};
\node[vertex] at (0.00000, 0.00000, 4.00000)     {};
\node[vertex] at (0.00000, 1.00000, 1.00000)     {};
\node[vertex] at (0.00000, 1.00000, 3.00000)     {};

\draw[mp] ( 0, 0, 0) -- ( 0,-1, 1) -- (-1,-1, 2) -- ( 0,-1, 3) -- ( 0, 0, 4); %
\draw[mp, color=blue] (0, 0, 0) -- (-1,0, 1) -- (-1,1, 2) -- (0, 1, 3) -- (0, 0, 4); %

\coordinate[label={above:$\bar{\bm{b}}_{1}$}] (b1) at (1, 0, 1);
\coordinate[label={above:$\bar{\bm{b}}_{2}$}] (b2) at (0, 1, 1);
\coordinate[label={above:$\bar{\bm{b}}_{-{1}}$}] (b-1) at (-1, 0, 1);
\coordinate[label={below:$\bar{\bm{b}}_{-{2}}$}] (b-2) at (0, -1, 1);

\draw[-latex, thick] ( 0, -2, 0) -- node [above,midway] {$\height$} ( 0,-2, 4);
\end{tikzpicture}
    \caption{The zonotope $\Z[\h{\pc[B]_2}]$. Three fibers of the height function~$\height$ are highlighted, representing a copy of the convex hull of $\pc[B]_2$, and of its $2$-set and $3$-set polytopes. The lower (red) path represents the coherent monotone path associated to the permutation $(\bar 2, \bar 1, 1, 2)$ (which can be read off the directions of the steps in the path). The upper (blue) path is a monotone path that is not coherent. It is associated to the permutation $(\bar 1,2,1,\bar 2)$, which is not a sweep permutation, but a pseudo-sweep permutation, see \cref{sec:pseudosweeps}.}
    \label{fig:zonotopeandfibers}
\end{figure}
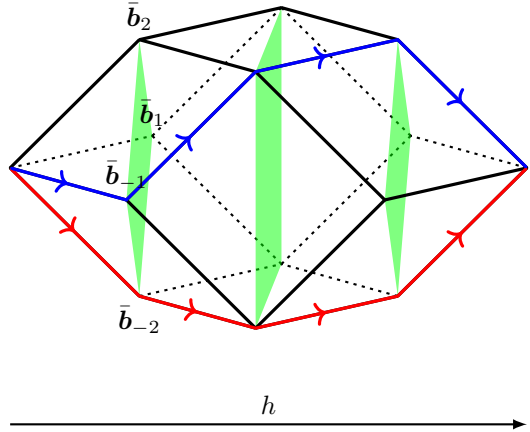

Conversely, monotone path polytopes of zonotopes for nondegenerate functionals are sweep polytopes, up to normal equivalence. 
Two polytopes are called \defn{normally equivalent} if they have the same normal fan, and normal equivalence obviously implies combinatorial equivalence. 

\begin{proposition}\label{prop:monotonepathzonotope}
Let $\pol[Z]\subset \RR^d$ be a zonotope, $\pi:\RR^d\to \RR$ a linear map, and $\pol[Z]^\pi$ the face of~$\pol[Z]$ minimizing~$\pi$. Then the monotone path polytope $\fib[{\pol[Z]}][\pi]$ is normally equivalent to the Minkowski sum of~$\pol[Z]^\pi$ with the sweep polytope~$\Sp$, where $\pc$ consists of the points $ \frac{1}{\pi(\p z_i)} \p z_i$ for the generators $\p z_i$ of $\pol[Z]$ such that $\pi(\p z_i)\neq 0$.
\end{proposition}

\begin{proof}
Let $\p c, \p z_1 \ldots, \p z_m\in\RR^d$ be such that 
\[\pol[Z]= \p c + \sum_{i=1}^m \left[-\p z_i ,\p z_i\right] \subset \RR^d.\]
Then $\pol[Z]$ is normally equivalent to any zonotope $\pol[Z]'= \p c' + \sum_{i=1}^m \left[-\lambda_i \p z_i ,\lambda_i \p z_i\right]$, where $\p c'$ is a vector in $\RR^d$ and the $\lambda_i$ are non-zero scalars.

Up to relabeling the $\p z_i$, one can suppose that $\set{i}{\pi(\p z_i)=0}=\{n+1, \ldots, m\}$ for a certain $n \in \{0, \ldots, m\}$.
Let $\pol[Z]_1$ and $\pol[Z]_2$ be the zonotopes:
\begin{align*}
\pol[Z]_1 &= \sum_{i=1}^n \left[-\frac{1}{\pi(\p z_i)} \p z_i, \frac{1}{\pi(\p z_i)} \p z_i\right], &
\pol[Z]_2 &= \sum_{i=n+1}^m \left[-\p z_i ,\p z_i\right].
\end{align*}

Note that the face $\pol[Z]^{\pi}$ is a translation of $\pol[Z]_2$.

Since $\pol[Z]$ is normally equivalent to the Minkowski sum $\pol[Z]_1 + \pol[Z]_2$, we have that its  monotone path polytope $\fib[{\pol[Z]}][\pi]$
is normally equivalent to the monotone path polytope $\fib[{\pol[Z]_1 + \pol[Z]_2}][\pi]$ by~\cite[Cor.~4.4]{McMullen2003}.

Moreover, $\fib[{\pol[Z]_1 + \pol[Z]_2}][\pi]=\fib[{\pol[Z]_1}][\pi]+\pol[Z]_2$ because $\pi(\pol[Z]_2)=\{0\}$, thus $(\pol[Z]_1 + \pol[Z]_2)\cap \pi^{-1}(\{y\}) = \pol[Z]_1 \cap \pi^{-1}(\{y\}) + \pol[Z]_2$ for any $y\in \RR$. 
If we denote the configuration of points $\p a_1 = \frac{1}{\pi(\p z_1)}\p z_1, \ldots, \p a_n = \frac{1}{\pi(\p z_n)} \p z_n$ in $\RR^d$ by $\pc$, we have exactly $s=\pi \circ \lm$ and $\pol[Z]_1=\lm(\cube)$, with the same notations as in \cref{prop:sweep polytopeprojection} and \cref{ex:permutahedronfiber}. Hence, \cref{lem:proj_fiber_pol} and \cref{ex:permutahedronfiber} give $\fib[{\pol[Z]_1}][\pi] = \lm(\fib[\cube][s])=\lm(\frac{2}{n}\cPn) = \frac{2}{n}\Sp$.

Hence $\fib[{\pol[Z]}][\pi]$ is normally equivalent to the Minkowski sum $\Sp + \pol[Z]^{\pi}$.
\end{proof}

There is an alternative (but strongly related) way to construct sweep polytopes as fiber polytopes. It is not directly used in the sequel, but we present it in \cref{sec:alternativefiber} for completeness.

\section{Sweep oriented matroids}\label{sec:sweeporientedmatroids}

The goal of this section is to provide a purely combinatorial definition of posets of sweeps generalizing allowable sequences to higher dimensions. Since already in the plane not all allowable sequences arise from point configurations, it is clear that our definition has to go beyond the realizable case. We will do it in terms of oriented matroids, which do have enough expressive power to completely describe allowable sequences. However, to motivate our definition, we will start by discussing some oriented matroids associated to point configurations, inspired by~\cite[Sects.~1.10 \& 6.4]{BLSWZ99}. While we will introduce the basic definitions in oriented matroid theory, we refer the reader not familiar with the topic to the introduction in~\cite[Lec.~6]{Ziegler1995}, and to the classical book~\cite{BLSWZ99} for a comprehensive source.

\subsection{Basic notions and notation}\label{sec:OMdefns}

There are several cryptomorphic approaches to oriented matroids. We will use the presentation in terms of the \defn{covector axioms}, which describe oriented matroids in terms of collections of sign-vectors $ \cov\subseteq \{+,-,0\}^E$, called \defn{covectors}, labeled by a finite ground set~$E$. 

For $X\in \cov$ and $e\in E$, $X_e$ denotes the value of $X$ at the coordinate~$e$. The \defn{opposite} $-X$ of $X\in\cov$ is the sign-vector obtained by switching $+$ and $-$ in~$X$; that is, $(-X)_e=-(X_e)$. For $X,Y \in \cov$, the \defn{composition} of $X$ and $Y$ is the sign-vector $X\circ Y \in \{ +, -,0\}^E$ such that $(X\circ Y)_e=X_e$ if $X_e\neq 0$; and $(X\circ Y)_e=Y_e$ otherwise. The \defn{separation set} of $X$ and $Y$, denoted $\sep[X][Y]$, is  the set of elements $e\in E$ such that $(X_e, Y_e) \in \{(+,-),(-,+)\}$.

\begin{definition}[{{cf.~\cite[Def.~4.1.1]{BLSWZ99}}}]\label{def:axioms_OM}
A collection of sign-vectors $\cov \subseteq \{+,-,0\}^E$ is the set of covectors of an \defn{oriented matroid} if it satisfies the following axioms:
\begin{description}
 \item[(V0)\label{it:COVaxiom0}] $\zero \in \cov$,
 \item[(V1)\label{it:COVaxiomSYM}] $X\in \cov$ implies $-X\in \cov$,
 \item[(V2)\label{it:COVaxiomCOMP}] $X,Y \in \cov$ implies $X\circ Y \in \cov$,
 \item[(V3)\label{it:COVaxiomELIM}] if $X,Y\in \cov$ and $e\in \sep[X][Y]$ then there exists $Z \in \cov$ such that $Z_e=0$ and $Z_f=(X\circ Y)_f$ for all $f\notin S(X,Y)$. 
 \end{description}
\end{definition}

The set of covectors of an oriented matroid, with the product partial order induced by $0\prec +,-$ componentwise, forms a poset. It has the structure of a lattice, called the \defn{big face lattice} of the oriented matroid, if a top element~$\greatest$ is adjoined. The \defn{rank} of the oriented matroid is the length of the maximal chains in the poset of covectors.
The minimal non-zero covectors are called \defn{cocircuits}, and they determine the oriented matroid as every non-zero covector is a composition of cocircuits.
The maximal covectors for this partial order are called the \defn{topes} of the oriented matroid. They also determine the oriented matroid, as $X$ is a covector of $\cov$ if and only if $X\circ T$ is a tope for every tope~$T$. 
In fact, the \defn{tope-graph} of $\cov$, whose vertices are the topes and whose edges are given by the covectors covered by exactly two topes, already determines the oriented matroid up to FL-isomorphism, see \cite[Theorem 6.14]{BEZ90} and \cite[Theorem 4.2.14]{BLSWZ99}.

There are several standard notions of {oriented matroid isomorphism}. By \defn{FL-isomorphism}, we mean the coarsest, induced by isomorphism of the big face lattices. 
FL-isomorphism, called just isomorphism in~\cite{FinschiFukuda2002}, is the equivalence relation induced by reorientation, relabeling, and introduction/deletion of loops and parallel elements.

To understand the concepts used in the definition of FL-isomorphism, we need some extra notation. For $X\in \{+,-,0\}^E$ and $F\subseteq E$, we denote by $\reor{X}{F}$ the signed vector $Z$ such that: $Z_f = -X_f$ for $f\in F$ and $Z_e = X_e$ for $e\in E\setminus F$, which we call the \defn{reorientation} of $X$ on $F$. If $\OM$ is an oriented matroid on the ground set~$E$, its \defn{reorientation} on $F$ is the oriented matroid $\reor{\OM}{F}$ with covectors $\reor{X}{F}$ for $X\in \OM$. The \defn{support} of a sign-vector~$X$ is $\supp{X}=\set{e\in E}{X_e\neq 0}$. A \defn{loop} is an element $e\in E$ that does not belong to the support of any covector. 
Two elements $e,f\in E$ are said to be \defn{parallel} if $X_f = X_e$ for all $X\in \cM$ or $X_f=-X_e$ for all $X \in \cM$. This defines an equivalence relation on $E$, whose equivalence classes are called \defn{parallelism classes}. The parallelism class of $e\in E$ is denoted $\parallelclass{e}$. An oriented matroid is called \defn{simple} if it does not contain loops or distinct parallel elements.

For $X\in \{+,-,0\}^E$ and $F\subseteq E$, \defn{the restriction of $X$ to $F$}, denoted $\restr{X}{F}$ is the covector $Z \in \{+,-,0\}^F$ such that $Z_f=X_f$ for all $f\in F$. If $\OM$ is an oriented matroid on the ground set~$E$, the set $\set{\restr{X}{F}}{ X\in \cov}$ forms an oriented matroid, denoted $\restr{\cov}{F}$ and called the \defn{restriction of $\cov$ to~$F$}. 
The set $\set{\restr{X}{E\setminus F}}{ X\in \cov, X_f=0\, \forall f\in F}$ also forms an oriented matroid, denoted $\contract{\cov}{F}$ and called the \defn{contraction of $\cov$ along~$F$}. 

An oriented matroid is called \defn{acyclic} if the all-positive sign-vector~$\pluses$ is a tope.

The standard way to associate an oriented matroid to a real vector configuration~$\vc=(\p v_1,\dots,\p v_n)\in\RR^{d\times \ivl}$ is to consider the set of covectors on the ground set~$\ivl$ induced by the signs of the evaluations of linear functionals on the elements of~$\vc$: 
\begin{equation}
 \label{eq:RealizationOM}\VectOM=\set{(\sign(\sprod{\p u}{\p v_1},\dots,\sign(\sprod{\p u}{\p v_n} )}{\p u\in\RR^n}\subseteq\{+,-,0\}^n,
\end{equation}
where $\sign(x)=\begin{cases}
+ &\text{ if } x>0\\
- &\text{ if } x<0\\
 0 &\text{ if } x=0.
               \end{cases}$
               
That is, to each linear oriented hyperplane, we record which vectors of the configuration lie on the hyperplane, and which lie at the positive and negative sides, respectively. The covectors~$\VectOM$ label the regions of the hyperplane arrangement $\HA[\vc]$ consisting of the hyperplanes orthogonal to the vectors of~$\vc$. Under this labeling, the big face lattice is consistent with the inclusion order of the regions, the topes labeling the maximal cells of the arrangement. Thus, the big face lattice on $\VectOM$ is isomorphic to (the opposite of) the face lattice of the zonotope~$\sum_{i\in \ivl} [\p 0,\p v_i]$. The rank of $\VectOM$ coincides with the dimension of the linear hull of~$\vc$.
We will call this oriented matroid the \defn{oriented matroid associated to~$\vc$}. Oriented matroids that arise this way are called \defn{realizable}. Note that even non-realizable oriented matroids can be geometrically realized by \defn{arrangements of pseudo-spheres}, see \cite[Sec.~1.4.1 \&~5.2]{BLSWZ99}.

\subsection{Three realizable oriented matroids associated to a point configuration}

The construction above extends directly to affine point configurations, by considering evaluations of affine functionals instead. (Or, equivalently, linear functionals on the homogenization~$\h \pc$.) Although this is the standard way to associate an oriented matroid to a point configuration~$\pc$, we will call it the \defn{little oriented matroid} of~$\pc$, which is consistent with the notation in~\cite[Sect.~1.10]{BLSWZ99} for planar configurations. This is to avoid confusion with the other alternative notions of oriented matroid associated to a point configuration that we will introduce. The \defn{big oriented matroid}, which contains more information than the little oriented matroid, is also inspired by~\cite[Sect.~1.10]{BLSWZ99}. We will prefer a more compact presentation, the \defn{sweep oriented matroid}, which was not explicitly introduced there.

\enlargethispage{.5cm}
\begin{definition}\label{def:oms}
Let $\pc=(\p a_1, \ldots,\p a_n)\in\RR^{d\times \ivl}$ be a full-dimensional point configuration (i.e.\ its affine span is the whole space $\RR^d$):
\begin{enumerate}[(i)]
 \item The \defn{little oriented matroid} of $\pc$, denoted \defn{$\LOMA$}, is the oriented matroid of rank~$d+1$ with ground set~$\ivl$ associated to the  $(d+1)$-dimensional homogenized vector configuration $\h \pc = (\h{\p a}_1, \dots, \h{\p a}_n)\in\RR^{(d+1)\times \ivl}$, where $\h{\p a}_i=(\p a_i,1)\in \RR^{d+1}$. This is always an acyclic oriented matroid.
 
 \item
The \defn{sweep oriented matroid} of $\pc$, denoted \defn{$\MOMA$}, is the oriented matroid of rank~$d$ with ground set~$\ipairs=\set{(i,j)}{1\leq i < j \leq n}$ associated to the $d$-dimensional vector configuration \[\textstyle \set{\p a_{(i,j)}=\p a_j-\p a_i}{(i,j)\in\ipairs} \in \RR^{d\times \ipairs}.\]

 \item 
The \defn{big oriented matroid}\footnote{Our definition differs slightly from that in~\cite[Sect.~1.10]{BLSWZ99}. We admit parallel vectors when the configuration is not generic, whereas in \cite[Sect.~1.10]{BLSWZ99} all parallel vectors of the form $\p a_j-\p a_i$ are merged into a single element of the oriented matroid.} of $\pc$, denoted \defn{$\BOMA$},
is the oriented matroid of rank~$d+1$ on the ground set~$\ivl\cup\ipairs$ associated to the $(d+1)$-dimensional vector configuration 
\[\textstyle \h \pc \cup \set{ (\p a_{(i,j)},0)}{(i,j)\in\ipairs } \in \RR^{\left(d+1\right)\times \left(\ivl\cup \ipairs\right)}.\]

\end{enumerate}
\end{definition}

\begin{figure}[htpb]
    \includegraphics[width=.5\textwidth]{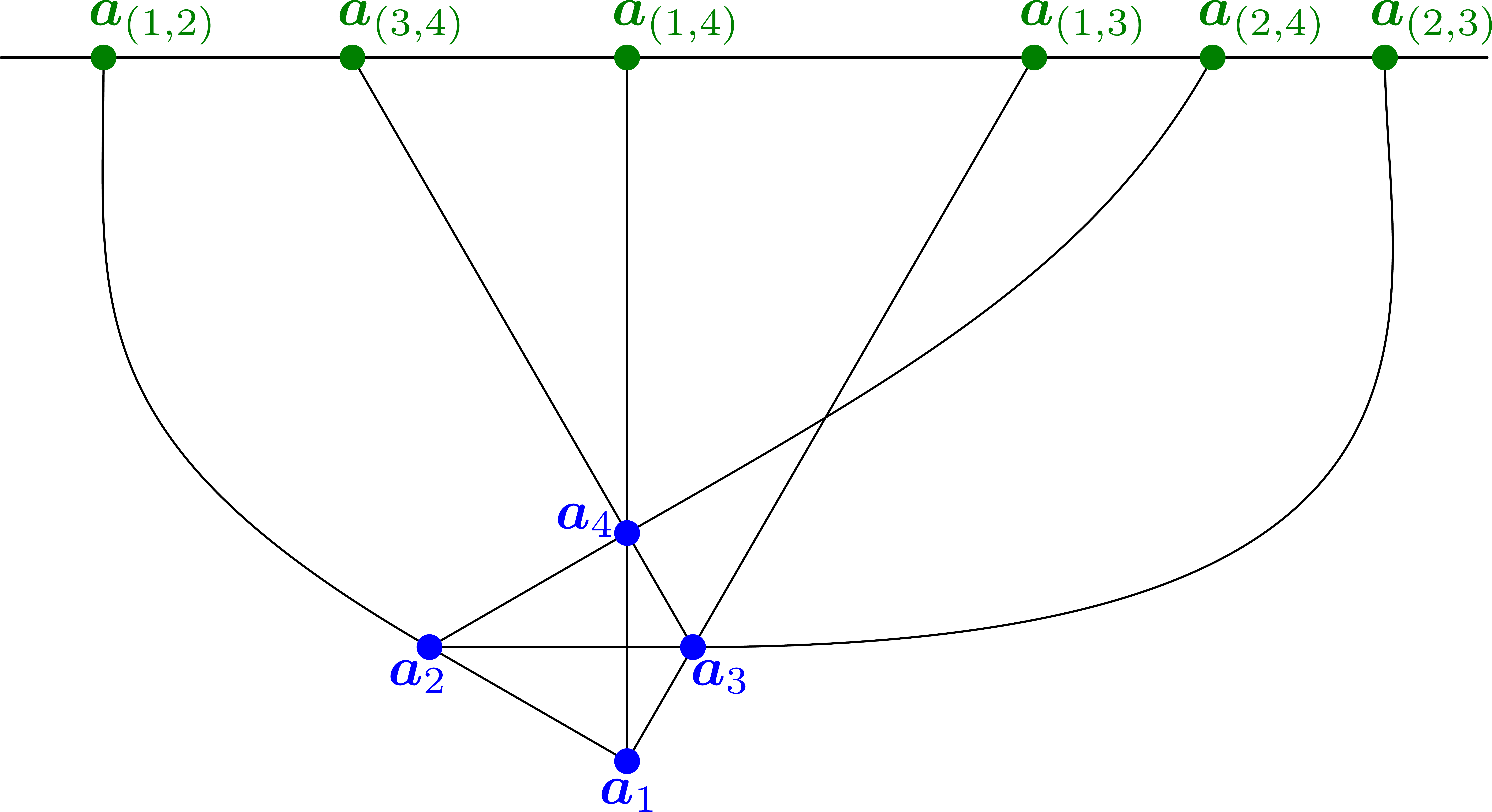}
    \caption{A big oriented matroid (with collinearities indicated). The points in the upper line, which represents the line at infinity, give rise to a sweep oriented matroid, whereas the points below give rise to the associated little oriented matroid.}
    \label{fig:BigOM}
\end{figure}

Little, sweep and big oriented matroids obtained this way from a point configuration will be called \defn{realizable}. In Sections~\ref{subsec:sweeporientedmatroids} and \ref{sec:fromsweeptobig}, we give definitions for abstract sweep, little and big oriented matroids not necessarily arising from point configurations. 
We explain below how these structures are related to each other and to the poset of sweeps and the set of sweep permutations.

For a sweep $I=(I_1, \ldots, I_l)\in \Scom$, corresponding to the surjection $\sur : [n] \to [l]$, we define the sign-vector $\sv\in\{+,-,0\}^{\ipairs}$ such that
\begin{equation}\label{eq:covectorfrompartition}
  X_{(i,j)}^I=\begin{cases}
            + &\text{ if } \sur(i)<\sur(j),\\
            - &\text{ if } \sur(i)>\sur(j),\\
            0 &\text{ if } \sur(i)=\sur(j);
           \end{cases}
\end{equation}
for $(i,j)\in\ipairs$. 

For example, if $I$ is the sweep $(\{1, 3\}, \{2\})$, we have $\sur(1)=\sur(3)=1$, $\sur(2)=2$, and the corresponding covector on the ground set $\{(1,2), (1,3), (2,3)\}$ is $\sv=(+, 0, -)$. Compare \cref{fig:SweepsA3,fig:sweeparrangementA3} to see other examples. As the figures illustrate, this map induces an isomorphism at the level of posets.

\begin{figure}[htpb]
        \includegraphics[width=.8\textwidth]{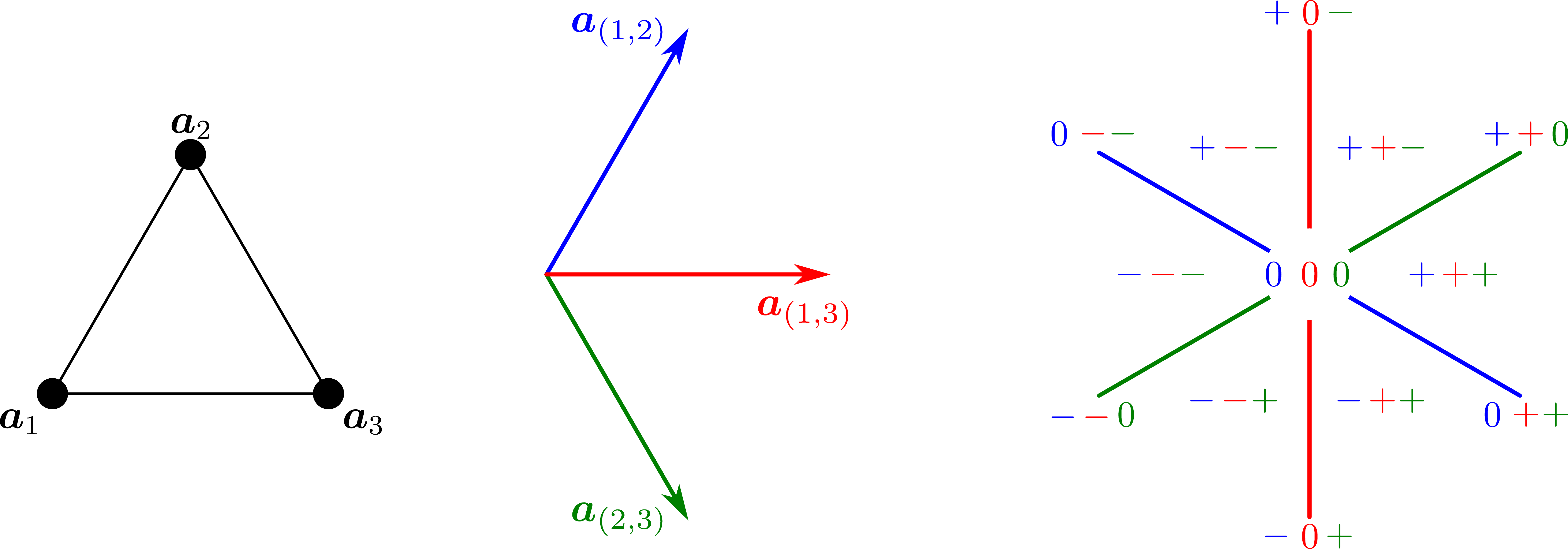}
    \caption{The vector configuration
      $\left\{\p a_{(1,2)},\p a_{(1,3)},\p a_{(2,3)}\right\}$ associated to the point configuration $\pc_3$ from \cref{fig:SweepsA3}. The covectors associated to the regions of the sweep hyperplane are indicated
by sign-vectors of length~$3$ containing the sign of the scalar product of a vector in the region with $\p a_{(1,2)}$, $\p a_{(1,3)}$, and $\p a_{(2,3)}$, respectively. This should be compared with the labeling of the regions of the sweep hyperplane arrangement in terms of partitions in \cref{fig:SweepsA3}.}\label{fig:sweeparrangementA3}
\end{figure}

\begin{lemma}\label{lem:equiv_sweepcomplex_sweepOM}
The map $I \mapsto \sv$ induces a poset isomorphism between the poset of sweeps $\Scom$ and the poset of covectors of the sweep oriented matroid~$\MOMA$.
\end{lemma}

In particular, $\Scom\cup\greatest$, where $\greatest$ is an additional top element, is isomorphic to the big face lattice of $\MOMA$, which is the opposite of the face lattice of the zonotope~$\Sp$ (cf.\ \cite[Cor.~7.17]{Ziegler1995}). 

\begin{proof}
Let $I$ be an ordered partition in $\Scom$, with corresponding surjection $p_I$, and associated to the linear form $u\in \RR^d$. 
This linear form $u$ is also associated to a covector $X$ of $\MOMA$ that is exactly the image of $I$ by the above bijection:

\begin{eqnarray*} 
X_{(i,j)} = 0 &\Leftrightarrow \sprod{u}{a_j-a_i}=0 &\Leftrightarrow p_I(i)=p_I(j), \\
X_{(i,j)} = + &\Leftrightarrow \sprod{u}{a_j-a_i}>0&\Leftrightarrow p_I(i)<p_I(j), \\
X_{(i,j)} = - &\Leftrightarrow \sprod{u}{a_j-a_i}<0&\Leftrightarrow p_I(i)>p_I(j).
\end{eqnarray*}

Hence both the sweeps of $\Scom$ and the covectors of $\MOMA$ are in bijection with the cells of the hyperplane arrangement $\VO$ and the bijections induce poset isomorphisms.
\end{proof}

It follows from the previous lemma that the set of sweep permutations $\PsetA$ is in bijection with the topes of the sweep oriented matroid $\MOMA$. Since the topes of an oriented matroid completely determine it (cf.~\cite[Proposition~3.8.2]{BLSWZ99}), this implies:

\begin{corollary}\label{cor:permutationsetsuffices}
The set of sweep permutations $\PsetA$ determines the whole poset of sweeps~$\Scom$.
\end{corollary}

The structures we have introduced are related by the following hierarchy (whose proof depends on the upcoming~\cref{prop:realizablesweeptobig}):

\begin{theorem}\label{thm:hierarchy}
Let $\pc \in\RR^{d\times \ivl}$ be a point configuration. 
Then the set of sweep permutations $\PsetA$, the poset of sweeps $\Scom$, the sweep oriented matroid $\MOMA$ and the big oriented matroid $\BOMA$ (cryptomorphically) determine each other.
They determine the little oriented matroid $\LOMA$, which does not always determine them.

\end{theorem}

In particular, the sweep oriented matroid is a combinatorial invariant of a point configuration that is finer than the order type (given by the little oriented matroid).

\begin{proof}
The fact that $\PsetA$ and $\Scom$ determine each other follows from \cref{cor:permutationsetsuffices}. 
The equivalence between $\Scom$ and $\MOMA$ follows from \cref{lem:equiv_sweepcomplex_sweepOM}. 
The equivalence between $\MOMA$ and $\BOMA$ will be proved later, as a consequence of \cref{def:def_big_OM} and \cref{prop:realizablesweeptobig}. 

Finally, $\BOMA$ determines $\LOMA$ by restriction to the ground set $[n]$ but this operation is not injective. 
Examples of planar configurations with different sets of sweep permutations but the same little oriented matroid can be found in~\cite[Section~1.10]{BLSWZ99}.
\end{proof}

\subsection{Sweep oriented matroids}\label{subsec:sweeporientedmatroids}

The main insight for expanding the notion of sweep oriented matroids from Definition~\ref{def:oms} beyond the realizable case is to note that a configuration of vectors of the form $\p a_j-\p a_i$ for $(i,j)\in\ipairs$ is just the projection of the \defn{braid configuration} $\smallset{\p e_j-\p e_i}{(i,j)\in\ipairs } \in \RR^{n\times \ipairs}$ (the set of positive roots of the Coxeter root system~$A_{n-1}$) under the linear map~$\lm$ defined in~\eqref{eq:phiA}.

 Consider the oriented matroid~\defn{$\Mbraid$} associated to the braid configuration, that is, the graphic oriented matroid of the complete graph $K_n$ with the acyclic orientation induced by the usual order on~$\ivl$. We will use the same notation~$\Mbraid$ as with the hyperplane arrangement, as it will be always clear from the context whether we are considering the hyperplane arrangement or the associated oriented matroid.
Note that, since the configuration of the $\p a_j-\p a_i$ is a linear projection of the braid configuration, every covector of $\MOMA$ is a covector of the braid oriented matroid, as we can pull back linear forms with $\lm^*$.
 
The oriented matroid analogues of linear projections are \defn{strong maps}. For two oriented matroids  $\OM_1$ and $\OM_2$ on the same ground set, we say that there is a \defn{strong map} from $\OM_1$ to $\OM_2$, denoted $\OM_1\to \OM_2$, if every covector of $\OM_2$ is a covector of $\OM_1$ (see~\cite[Sec.~7.7]{BLSWZ99}). This will be the starting point for our definition.

\begin{definition}
An oriented matroid~$\OM$ on the ground set~$\ipairs$ is a \defn{sweep oriented matroid} if there is a strong map $\Mbraid\to \OM$ from $\Mbraid$ to $\OM$, i.e.\ if all covectors of $\OM$ are covectors of $\Mbraid$.
\end{definition}

\begin{remark}\label{rmk:bijection}
Note that, if $\OM$ is a sweep oriented matroid, then we can interpret its covectors as covectors of the braid arrangement, and hence each covector can be uniquely identified  with an ordered partition via the bijection inverse to~\eqref{eq:covectorfrompartition}. 
For a covector $X\in\cov$ of a sweep oriented matroid, we will denote by~\defn{$\op$} the associated ordered partition.
\end{remark}

Our next result characterizes sweep oriented matroids via a $3$-term orthogonality condition on covectors (c.f.\ \cite[Sec.~3.4]{BLSWZ99}) that provides an explicit test for deciding whether an oriented matroid is a sweep oriented matroid.
It will be relevant later in the context of sweep acycloids in \cref{sec:sweepacycloids}.

Recall that the \defn{support} of a sign-vector~$X\in\{+,-,0\}^E$ is $\supp{X}=\set{e\in E}{X_e\neq 0}$. Two sign-vectors $X,Y\in\{+,-,0\}^E$ are said to be \defn{orthogonal} if either $\supp{X}\cap \supp{Y}=\emptyset$, or the restrictions of $X$ and $Y$ to $\supp{X}\cap \supp{Y}$ are neither equal nor opposite (i.e., there are $i,j$ with $X_i=Y_i\neq 0$ and $X_j=-Y_j\neq 0$).

\begin{lemma}\label{lem:OMtransitivity}

An oriented matroid $\OM$ on~$\ipairs$ is a sweep oriented matroid if and only if for every covector~$X$ and every choice of $1\leq i < j < k \leq n$, the triple $(X_{(i,j)},X_{(j,k)},X_{(i,k)})$ is orthogonal to the sign vector $(+,+,-)$.
 
 Equivalently, $\OM$ is a sweep oriented matroid if and only if for any covector~$X$, and for  $1\leq i < j < k \leq n$, the triple $(X_{(i,j)},X_{(j,k)},X_{(i,k)})$ does not belong to the following list of forbidden patterns:
 \[
 \left\{
 \begin{tabular}{ccccccc}
 $(+,+,-)$,&$(-,-,+)$,&$(0,+,-)$,&$(0,-,+)$,&$(+,0,-)$,&$(-,0,+)$,&$(+,+,0)$,\\
 $(-,-,0)$,&$(0,0,-)$,&$(0,0,+)$,&$(0,+,0)$,&$(0,-,0)$,&$(+,0,0)$,&$(-,0,0)$
 \end{tabular}
 \right\}.
 \]

\end{lemma}

\begin{proof}
There is a strong map $\Mbraid\to\OM$  if and only if all the covectors of $\OM$ are covectors of $\Mbraid$, which is equivalent to the condition that all the covectors of $\OM$ are orthogonal to all circuits of~$\Mbraid$ (see \cite[Prop.~7.7.1]{BLSWZ99}).

The circuits of $\Mbraid$ are induced by cycles of~$K_n$.
They are of the form~$C^{i_1, \ldots, i_r}$ for any collection $i_1, \ldots, i_r$ of at least~$3$ distinct elements of~$\ivl$, with $C^{i_1, \ldots, i_r}_{(i_k, i_{k+1})}=+$ if $i_k<i_{k+1}$ and $C^{i_1, \ldots, i_r}_{(i_{k+1}, i_{k})}=-$ if $i_k>i_{k+1}$  for all $1\leq k \leq r$ (with the convention $i_{r+1}=i_1$), and $C^{i_1, \ldots, i_r}_{(h,l)}=0$ for any other pair. 

An easy induction shows that the orthogonality to the circuit $C^{i_1, \ldots, i_r}$ is implied by the orthogonality to all circuits $C^{i_1, i_k, i_{k+1}}$ for $2\leq k \leq r-1$, which is equivalent to our statement. 
\end{proof}

This condition is actually a reformulation of the transitivity of the partial order induced by an ordered partition~$I$ (namely $i \preceq j$ if and only if $p_I(i)\leq p_I(j)$).
For example, forbidding the patterns $(+,+, -)$ and $(+, +, 0)$ is equivalent to stating that $i\prec j \prec k$ implies $i\prec k$, and so on.
 This is why we refer to it as the \defn{transitivity condition} on sweep oriented matroids.

The \defn{poset of sweeps} of a sweep oriented matroid~$\OM$ is the partially ordered set $\Scom[\OM]$ of the ordered partitions $\op$ for the covectors $X\in \OM$, ordered by refinement. Enlarged with a  top element~$\greatest$, this poset is isomorphic to the big face lattice of~$\OM$. 
The topology of such complexes is well known~\cite[Thm.~4.3.3]{BLSWZ99}.
We describe it in the following proposition. Note that there is some ambiguity in the literature concerning the definition of the poset of faces of cell complexes, in particular whether it should be augmented by a bottom element or not (compare \cite[Fig.~2]{Bjorner1984} and \cite[Fig.~2]{Bjorner1995}). We follow~\cite{Bjorner1995} and~\cite{BLSWZ99} and do not include an additional bottom element in the definition of the \defn{face poset} of a cell complex. 

\begin{proposition}[{\cite[Thm.~4.3.3]{BLSWZ99}}]\label{cor:topologyPosetSweeps}
 The poset of sweeps $\Scom[\OM]\ssm(\ivl)$ of a sweep oriented matroid~$\OM$ of rank~$r$ without the trivial sweep is isomorphic to the face poset of a shellable regular cell decomposition of the $(r-1)$-sphere. 
 In particular, the order complex $\OC[{\Scom[\OM]\ssm(\ivl)}]$ triangulates the $(r-1)$-sphere.
\end{proposition}

\section{Big and little oriented matroids}\label{sec:bigandlittle}

In this section we show how the big and little oriented matroids of a point configuration (Definition~\ref{def:oms}) are completely determined by its sweep oriented matroid. Actually, the construction of these matroids can be extended to any abstract sweep oriented matroid, providing definitions beyond the realizable case. This generalizes the results for rank~$3$ proved in~\cite[Sec.~1.10]{BLSWZ99}. 

\subsection{Big and little oriented matroids associated to sweep oriented matroids}\label{sec:fromsweeptobig}
First, we will show how to extend any sweep oriented matroid to what will be called a big oriented matroid. For a covector~$X$ of a sweep oriented matroid, let $\sur[X]:\ivl\to [l_X]$ be the surjection associated to the corresponding ordered partition.
For each $1\leq k \leq  2l_X+1$, let $X^k\in \{+,-,0\}^{\ivl\cup\ipairs}$ be the sign-vector:
\begin{align*}
{X}^k_i &= 
\begin{cases}
- &\text{ if } p_X(i)\leq \lfloor \frac{k-1}{2} \rfloor, \\
+ &\text{ if } p_X(i)>\lfloor \frac{k}{2} \rfloor, \\
0 &\text{ if $k$ is even and } p_X(i)=\frac{k}{2}.
\end{cases}&&\text{ for }1\leq i\leq n;\\
{X}^k_{(i,j)} &= X_{(i,j)} &&\text{ for all } 1\leq i < j \leq n.
\end{align*}

We defer the details of checking that the transitivity condition from Lemma~\ref{lem:OMtransitivity} implies the oriented matroid axioms for these covectors to \cref{sec:appendix}. They are easy, but tedious.

\begin{theorem}\label{thm:BOMisOM}
If $\cov$ is the set of covectors of a sweep oriented matroid, then 
\[\BOM[\cov] = \set{{X}^k}{ X\in \cov, \, 1\leq k\leq 2l_X+1}\] is the set of covectors of an oriented matroid.
\end{theorem}

\begin{definition}\label{def:def_big_OM}
Let $\OM$ be a sweep oriented matroid. The oriented matroid \defn{$\BOM[\cov]$} is the \defn{big oriented matroid} of $\OM$; and the oriented matroid~\defn{$\LOM$} obtained by deleting all pairs $(i,j)$ from $\BOM$ is the \defn{little oriented matroid} of~$\OM$.
\end{definition}

These definitions are indeed coherent with the realizable case, as the following proposition shows. This proves that the sweep oriented matroid of a point configuration determines its big and little oriented matroids, concluding the proof of \cref{thm:hierarchy}.

\begin{proposition}\label{prop:realizablesweeptobig}
 The big and little oriented matroids of a point configuration are the big and little oriented matroids associated to its sweep oriented matroid.
\end{proposition}
\begin{proof}
 Let $\pc=(\p a_1,\dots,\p a_n)\in\RR^{d\times\ivl}$ be a $d$-dimensional point configuration. Every vector $\p u\in\RR^d$ induces an ordering of $\pc$, which is encoded in a covector $X$ of $\MOMA$. 
 For $c\in \RR$, the partition 
 \[\set{i}{\sprod{\p u}{\p a_i}<c}, \set{i}{\sprod{\p u}{\p a_i}=c}, \set{i}{\sprod{\p u}{\p a_i}>c}\] only depends on which, or between which pair, of the $l_X$ values attained by $\lfv$ on $\pc$ does $c$ lie. These $2l_X+1$ distinct partitions are precisely those encoded by the covectors $X^k$ defining the big oriented matroid of $\MOMA$.
\end{proof}

Note that, by the definition of the big oriented matroid of $\OM$, the zero covector~$\zero$
of $\OM$ induces the all-positive tope~$\pluses$ in $\LOM$, which is hence an acyclic oriented matroid.

The following lemma concerning the ranks of the big and little oriented matroids will be needed later.

\begin{lemma}\label{lem:rank_BOM}
If the sweep oriented matroid $\OM$ is of rank~$r$, then $\BOM$ and $\LOM$ are of rank~\mbox{$r+1$}. 
\end{lemma}

\begin{proof}
To justify that $\BOM$ has rank $r+1$, it is sufficient to notice that if $\zeros[\ipairs]=Y^0\prec Y^1\prec \cdots\prec Y^{r}$ is a maximal chain of covectors of $\OM$, then $\zeros[\ivl\cup \ipairs]=Z^{-1}\prec Z^0\prec Z^1\prec \cdots\prec Z^{r}$ is a maximal chain of covectors of $\BOM$, where for any $k\in \{0, \ldots, r\}$, we define $Z^k$ by $\restr{Z^k}{\ipairs}=Y^k$ and $\restr{Z^k}{\ivl}=\pluses$. 
Indeed, we cannot add a covector $Z$ in the big oriented matroid between $Z^{-1}$ and $Z^0$ because if $Z_i=0$ and $Z_j=+$ we necessarily have $Z_{(i,j)}\neq 0$ since $i$ and $j$ are not in the same part of the ordered partition~$l_Z$. We cannot add a covector strictly between $Z^k$ and $Z^{k+1}$ either because its restriction to $\ipairs$ would give a covector of $\OM$ strictly between $Y^k$ and $Y^{k+1}$.

We prove that $\LOM$ also has rank $r+1$ by induction on~$r$. 
If $\OM$ is of rank~$r=0$, then $\zeros[\ipairs]$ is its only covector. It induces the little oriented matroid of rank~$1$ consisting of the covectors $\minuses$, $\zeros$, and~$\pluses$.

Now, suppose that $\OM$ is a sweep oriented matroid on ground set $\ipairs$ that has rank $r\geq 1$. Up to relabelling, we can suppose that $(n-1,n)$ is not a loop. Then the contraction of $\OM$ along $\{(n-1,n)\}$ has rank $r-1$. Under the bijection~\eqref{eq:covectorfrompartition}, the covectors of this contraction $\contract{\OM}{\{(n-1,n)\}}$ correspond to the partitions associated to covectors of $\OM$ such that $n-1$ and $n$ are in the same part. This implies that for all $i\leq n-2$, the pairs $(i,n-1)$ and $(i,n)$ are parallel. By deleting all the pairs $(i,n)$ we obtain an oriented matroid $\OM'$ on $\ipairs[{[n-1]}]$ isomorphic to $\contract{\OM}{\{(n-1,n)\}}$. The transitivity condition from~\cref{lem:OMtransitivity} is preserved, and hence $\OM'$ is a sweep oriented matroid of rank~$r-1$ and $\LOM[\OM']$ has rank~$r$, by induction.
A maximal chain of the contraction $\LOM[\OM']/(n-1)$ induces a chain $\zeros=X^0 \prec \cdots\prec X^{r-1}$ of $\LOM$ in which $(X^i)_{n-1}=(X^i)_n=0$ for all $0\leq i\leq r-1$ and that is maximal with this property. Since $n-1$ and $n$ are not parallel (because $(n-1,n)$ is not a loop), there is a covector $Y$ of $\LOM$ such that $Y_{n-1}=+$ and $Y_n=0$. Setting $X^{r}=X^{r-1}\circ Y$, and $X^{r+1}=X^{r}\circ \pluses$, we obtain a chain \[\zeros=X^0 \prec \cdots\prec  X^{r-1}\prec  X^{r}\prec  X^{r+1}\] of lenght~$r+1$ of covectors of $\LOM$. 
Moreover, the restriction operation on oriented matroids cannot increase the rank, thus the rank of $\LOM$ cannot be bigger than the rank of $\BOM$.
Hence $\LOM$ also has rank~$r+1$.
\end{proof}

\begin{example}[The braid oriented matroids in types~$A$ and $B$]\leavevmode\label{ex:braidBOM}
The study of big oriented matroids of Coxeter hyperplane arrangements in types~$A$ and~$B$ unveils a recursive decomposition that, in view of the upcoming \Cref{sec:modularhyperplanes}, explains the existence of a maximal chain of modular flats. This important property was first studied by Stanley under the name of \defn{supersolvability}~\cite{Stanley1972}. 

\textbf{Type $\mathbf{A}$.} The big oriented matroid of the braid oriented matroid~$\Mbraid$ is the braid oriented matroid~$\Mbraid[n+1]$. More precisely, if we relabel the elements $i\in \ivl$ by $(1,i+1)$ and the elements $(i,j)\in \ipairs$ by $(i+1,j+1)$, then we recover the braid oriented matroid~$\Mbraid[n+1]$. Indeed, the topes of $\BOM[\Mbraid]$ are of the form $X^{2k+1}$ where $X$ is a tope of~$\Mbraid$ and $0\leq k\leq n$. If $X$ corresponds to the permutation $(\sigma(1),\dots,\sigma(n))\in \Sym$, then $X^{2k+1}$ corresponds to the permutation in $\Sym[n+1]$:
\[(\sigma(1)+1,\dots,\sigma(k)+1,1,\sigma(k+1)+1,\dots,\sigma(n)+1).\]

 \textbf{Type $\mathbf{B}$.} Consider the type~$B$ braid oriented matroid~$\MbraidB$ from \cref{sec:crosspolytope}, indexed by the elements in~$\ipairs[{[\pm n]}]$. That is, $\MbraidB$ is the sweep oriented matroid of the vertex set of the cross-polytope. Then its big oriented matroid~$\BOM[{(\MbraidB)}]$ is FL-isomorphic to~$\MbraidB[n+1]$ without one element (of those of the form $(-i,i)$%
 ).
 
 To see it, it is easier to consider first an enlarged version, with base elements \[\ivl[-n,n]=\{-n,\dots,-1,0,1,\dots,n\}\] corresponding to the point configuration \[\widetilde {\pc[B]}_n=(- \p e_{n}, \ldots, - \p e_1, \zero , \p e_1, \ldots, \p e_n)\] that contains the vertices of the cross-polytope together with the origin. The FL-isomorphism class of the sweep oriented matroid does not change, but we get some new parallel elements. Namely, the elements labeled $(-i,i)$, $(-i,0)$, and  $(0,i)$ become parallel (with the same orientation) in the enlarged sweep oriented matroid~$\MbraidBB=\MOMA[\widetilde {\pc[B]}_n]$. Now, relabel the elements $\ivl[-n,n]\cup \ipairs[{\ivl[-n,n]}]$ to~$\ipairs[{\ivl[-n-1,n+1]}]$ by sending each $i\in \ivl[\pm n]$ to the pair of parallel elements $(-n-1,-i), (i,n+1)$; $0$ to the triple of parallel elements $(-n-1,n+1),(-n-1,0),(0,n+1)$; and leaving the pairs in $\ipairs[{[-n,n]}]$ unchanged.
 Each tope~$X$ of the sweep oriented matroid~$\MbraidBB$ is represented by a centrally symmetric permutation~$\sigma$ of~$\ivl[-n,n]$:
 \[(-\sigma(n), \dots, -\sigma(1),0,\sigma(1),\dots,\sigma(n)).\]
 Under the relabeling we can read the topes $X^{2k+1}$ of the big oriented matroid~$\BOM[{(\MbraidBB)}]$ as centrally symmetric permutations of $\ivl[-n-1,n+1]$ representing topes of~$\MbraidBB[n+1]$. Namely,
 for $0\leq k\leq n+1$, the tope $X^{2k+1}$ corresponds to the centrally symmetric permutation:
 \[(-\sigma(n), \dots, -\sigma({n-k+1}), -n-1, -\sigma({n-k}), \dots, \sigma({n-k}),n+1,\sigma({n+1-k}),\dots \sigma({n})).\]
whereas for $n+2\leq k\leq 2n+2$ it corresponds to:
\[(-\sigma(n), \dots, -\sigma({n-k+1}), n+1, -\sigma({n-k}), \dots,  \sigma({n-k}),-n-1,\sigma({n+1-k}),\dots \sigma({n})).\]
This shows that~$\BOM[{(\MbraidBB)}]$ is FL-isomorphic to~$\MbraidBB[n+1]$, and hence to~$\MbraidB[n+1]$.

If we want to consider the original configuration without the origin, we simply need to remove all the elements of the big oriented matroid that involve a label
using~$0$. Every parallelism class conserves at least one representative except for the singleton $0$, which was sent to the triple $(-n-1,n+1),(-n-1,0),(0,n+1)$ with our relabeling. This shows that $\BOM[{(\MbraidB)}]$ is FL-isomorphic to $\MbraidB[n+1]\ssm(-n-1,n+1)$.\qed
\end{example}

\begin{remark}[On labeling and isomorphism]\label{rmk:isomorphism}
The labeling plays an important role in the definition of a sweep oriented matroid and in~\cref{thm:hierarchy}.
Indeed, non-isomorphic big oriented matroids might arise from isomorphic sweep oriented matroids. (Here, we mean FL-isomorphism, but the statement is also true for the other standard notions of oriented matroid isomorphism.) For example, all sufficiently generic planar $n$-point configurations give rise to FL-isomorphic sweep oriented matroids but their big oriented matroids are not FL-isomorphic.
\end{remark}

\begin{remark}[On realizability]\label{rmk:realizability}
Note that, for a big oriented matroid $\OM$, realizability as an oriented matroid (i.e.\ in the sense of~\eqref{eq:RealizationOM}) is equivalent to realizability as a big oriented matroid (i.e.\ in the sense of \cref{def:oms}). Indeed, any point configuration $\pc$ such that $\BOMA=\OM$ can be extended (with the corresponding points at infinity) to an oriented matroid realization of~$\OM$. And reciprocally, the restriction of any oriented matroid realization of $\OM$ to the elements indexed by $\ivl$ can be sent, after a suitable projective transformation and dehomogenization, to a point configuration $\pc$ such that $\BOMA=\OM$.

\begin{figure}[htpb]
    \includegraphics[width=.35\textwidth]{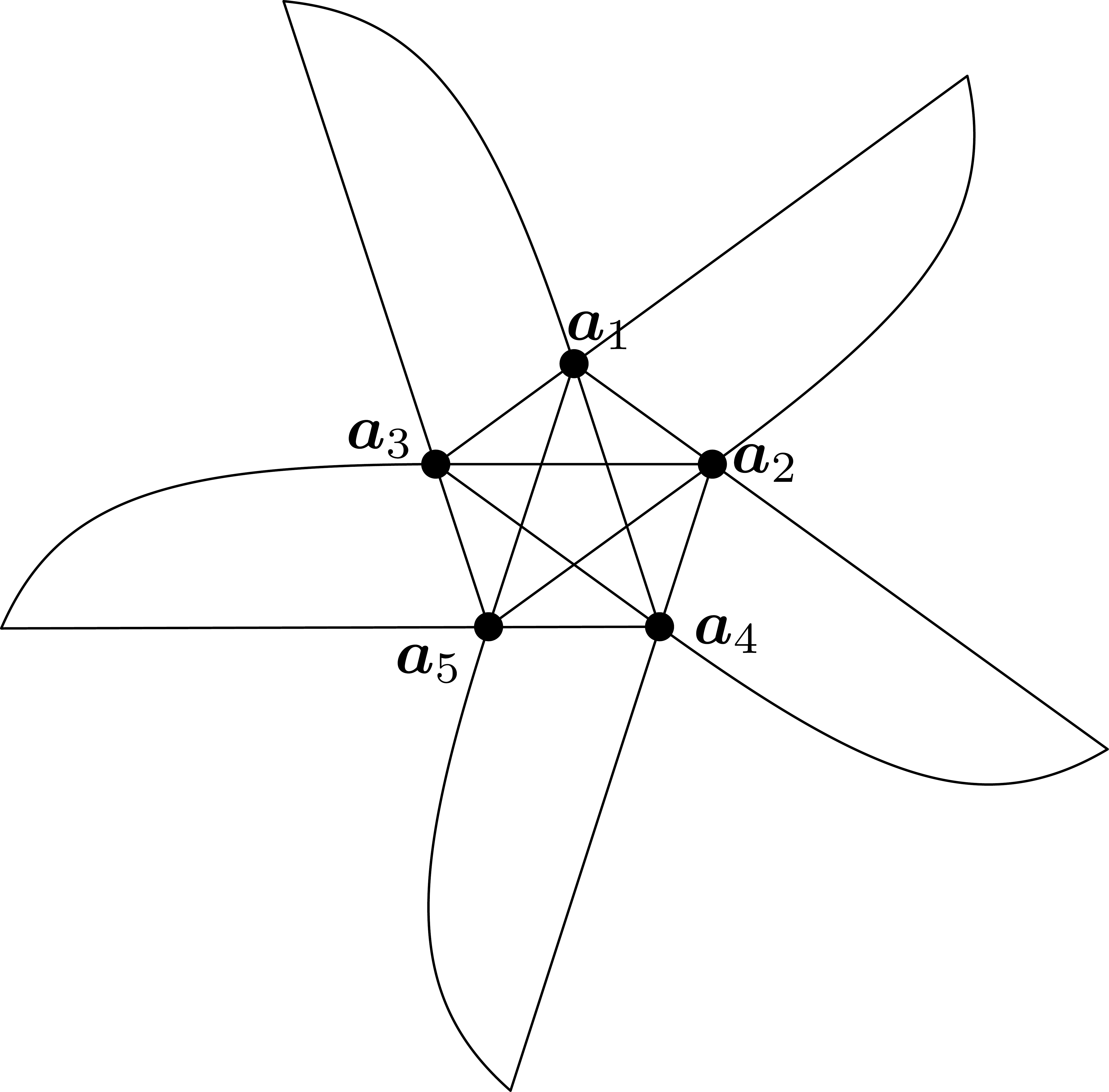}
        \caption{The allowable sequence $(1,2,3,4,5) \rightarrow (1,2,4,3,5)\rightarrow (2,1,4,3,5) \rightarrow (2,1,4,5,3) \rightarrow (2,4,1,5,3) \rightarrow (2,4,5,1,3) \rightarrow (4,2,5,1,3) \rightarrow (4,5,2,1,3) \rightarrow (4,5,2,3,1) \rightarrow (4,5,3,2,1) \rightarrow (5,4,3,2,1) $ cannot be realized by a point configuration, because it would necessarily be a pentagon whose sides and ``parallel diagonals'' meet as in the above picture, which is geometrically impossible~\cite{GP80}.}
    \label{fig:unrealizable_pent}
\end{figure}

In contrast, there are sweep oriented matroids that are realizable as an oriented matroid but that are not of the form $\MOMA$ for any point configuration~$\pc$. Indeed, the non-realizable pentagon of~\cite{GP80} (see \cref{fig:unrealizable_pent}) gives rise to a non-realizable allowable sequence; that is, to a non-realizable big oriented oriented matroid of rank~$3$. The associated sweep oriented matroid is an oriented matroid of rank~$2$, and thus realizable (as an oriented matroid)~\cite[Cor.~8.2.3]{BLSWZ99}. However, it is not the sweep oriented matroid of a point configuration, because the corresponding big oriented matroid is not realizable.

We end this remark by noting that the Universality Theorem for allowable sequences of Hoffmann and Merckx~\cite{HM18} implies that it is $(\exists\RR)$-hard to decide whether a big oriented matroid is realizable.
\end{remark}

\subsection{Big oriented matroids and tight modular hyperplanes}\label{sec:modularhyperplanes}

In this section we provide an alternative characterization of the FL-isomorphism classes of big oriented matroids, and hence of sweep oriented matroids. It is purely structural, without relying on the labeling of the elements. We show that they are closely related to the concept of {modular hyperplanes}.

According to our definition, every big oriented matroid~$\BOM$ on $\ivl\cup \ipairs$ contains the cocircuit $Z=(\pluses,\zeros[\binom{n}{2}])$. 
Moreover, $X_{(i,j)}=0$ for any covector $X$ such that $X_i=X_j=0$; which is equivalent to the fact that for any $i,j$ not in the same parallelism class, the restriction of~$\BOM$ to the set $\{i,j,(i,j)\}\subset E$ has rank $2$. These two properties show that the set of indices $\ipairs$ form a \defn{modular hyperplane}. 

The \defn{flats} of an oriented matroid~$\OM$ of rank~$r$ on~$E$ are the flats of its underlying (unoriented) matroid~$\un\OM$; that is, the zero-sets of its covectors. The poset of flats ordered by inclusion forms a geometric lattice~\cite[4.1.13]{BLSWZ99}. The \defn{hyperplanes} are the flats of rank~$r-1$, and they arise as zero-sets of cocircuits. A flat~$F$ is called \defn{modular} if $\rk{}{F}+\rk{}{G}=\rk{}{F\wedge G}+\rk{}{F\vee G}$ for any other flat~$G$, where $\rk{}{\cdot}$ is the rank function of the geometric lattice (for a flat~$F$,  $\rk{}{F}$ coincides with the rank of the oriented matroid~$\restr{\OM}{F}$). Modular flats have many interesting properties, and play an important role in the theory of matroids, see~\cite{Stanley1971} and \cite{Brylawski75}.

Hence, a \defn{modular hyperplane} is a hyperplane $F\subset E$ such that $\rk{}{F\wedge G}=\rk{}{F\cap G}=\rk{}{G}-1$ for any flat~$G$ not contained in $F$. Said differently, $F\cap G$ is a hyperplane in $\restr{\OM}{G}$. In~\cite[Cor.~3.4]{Brylawski75} it is shown that a hyperplane is modular if and only if it intersects every line (flat~$G$ with $\rk{}{G}=2$). Equivalently, if for every pair of elements $x,y\in E\ssm F$ that are not parallel nor a loop, there is some element $z\in F$ such that for every covector $X$ with $X_x=X_y=0$ we have $X_z=0$. We will say that a modular hyperplane $F$ is \defn{tight} if there is no $z\in F$ such that $F\ssm z$ is a modular hyperplane of $\restr{\OM}{E\ssm z}$.%

The following result gives a characterization of big oriented matroids similar to the one given in \cite[Sect.~6.4]{BLSWZ99} for the rank~$3$ case.

\begin{proposition}\label{prop:characterizationbig}
Let $\OM$ be an oriented matroid on ground set $E=\ivl\cup \ipairs$ such that:
\begin{enumerate}
\item there exists a cocircuit $Z$ of $\cM$ such that $\{e\in E \, |\, Z_e=0\}=\ipairs$ (i.e.\ $\underline{Z}=[n]$),\label{cond:modular}
\item for any $(i,j)\in \ipairs$, for any covector $X$ of $\cM$, if two coordinates among $X_i$, $X_j$, $X_{(i,j)}$ are zero, then the third one is zero too. \label{cond:rank2}
\end{enumerate}
Then, up to reorientation, $\OM$ is the big oriented matroid of the sweep oriented matroid $\restr{\OM}{\ipairs}$. 
\end{proposition}

In a realizable setting, and without parallel elements and loops, the conditions on $\cM$ amount to asking that the real vector representing $(i,j)$ is in the intersection of the $2$-plane spanned by the real vectors representing $i$ and $j$ and the hyperplane given by the cocircuit~$Z$ (which contains all the vectors corresponding to elements in~$\ipairs$). 
One can check that the example depicted in Figure~\ref{fig:BigOM} satisfies this condition. 

\begin{proof}
We need to prove that, after the reorientation of some elements of the ground set, the restriction $\restr{\OM}{\ipairs}$ is a sweep oriented matroid, i.e.\ it satisfies \cref{lem:OMtransitivity}, and the covectors of $\OM$ are exactly those obtained from the covectors of $\restr{\OM}{\ipairs}$ as in \cref{thm:BOMisOM}.

We can assume that, after a suitable reorientation of $\OM$ we have  that $Z=(\pluses,\zeros[\ipairs])$. 
Note that $\restr{\OM}{\ivl}$ cannot have loops, as witnessed by~$Z$; and that if $i$ and $j$ are parallel, then $(i,j)$ must be a loop. We will from now on assume that $\OM$ does not have parallel elements, as it simplifies the exposition.

Let us show that for any two covectors $X, Y\in \OM$ such that $X_i=Y_i=-$, $X_j=Y_j=+$ we have $X_{(i,j)}=Y_{(i,j)}\neq 0$. 
Assume the contrary. Then the axiom~\ref{it:COVaxiomELIM} on oriented matroids would imply the existence of a covector $T\in \OM$ such that $T_i=-$, $T_j=+$ and $T_{(i,j)}=0$. 
A second application of the axiom~\ref{it:COVaxiomELIM} between $T$ and $Z$ would give the existence of a covector $T'\in \OM$ such that $T'_i=T'_{(i,j)}=0$ and $T'_j=+$, which contradicts the second assumption on $\OM$. 
Hence, we can  reorient $(i,j)$ so that for any covector $X$ of $\cM$ with $X_i = -$ and $X_j = +$, we have $X_{(i,j)} = +$.

To check that $\restr{\OM}{\ipairs}$ is a sweep oriented matroid, it suffices to look at all restrictions of the form \[\restr{\OM}{\{i, j, k, (i,j), (j,k), (i,k)\}}\] for $1\leq i< j < k \leq n$. This gives an oriented matroid of rank at most~$3$. One can easily check that with our conditions there are only three possible configurations, none of which violates the condition from \cref{lem:OMtransitivity}.

Moreover, it is clear that any covector $X$ of $\OM$ can be obtained from the covector $\restr{X}{\ipairs}$ of $\restr{\OM}{\ipairs}$ by the method described at the beginning of \cref{sec:fromsweeptobig}. Indeed, our reorientation on $\ipairs$ implies that the ordered partition of $\ivl$ given by $(I_-=\{i \, |\, X_i=-\}, \, I_0=\{i \, |\, X_i=0\}, I_+=\{i \, |\, X_i=+\})$ is refined by the ordered partition $J$ induced by $\restr{X}{\ipairs}$, in such a way that either $I_0=\emptyset$ or $I_0$ is an entire part of $J$. Thus $X$ is of the form $(\restr{X}{\ipairs})^k$ for some $k$.

It remains to check that, for every covector~$Y\in \restr{\OM}{\ipairs}$, all covectors~$Y^k$ obtained by the method described in \cref{sec:fromsweeptobig} are indeed covectors of $\OM$. We do it by induction on $k$.
Observe first that $Y^1=Z\circ \tilde{Y}$, where $\tilde{Y}$ is any covector in $\OM$ whose restriction to $\ipairs$ gives $Y$. Thus, we have $Y^1\in \cM$. 

Now, for an odd $k_0\in \ivl[2 l_Y]$, we apply the Elimination Axiom~\ref{it:COVaxiomELIM} to the covectors $Y^{k_0}$ and $(-Z)\circ Y^{k_0}$, and the smallest element $i_0\in\sur[Y]^{-1}(\{\frac{k_0 + 1}{2}\})$ to obtain a covector $T$. 
We claim that $T=Y^{k_0+1}$. Indeed, for all $i$ where $\sur[Y](i)<\frac{k_0 + 1}{2}$ we have $T_i=Y^{k_0}_i=(-Z)_i=-$. For all $i\in \sur[Y]^{-1}(\{\frac{k_0 + 1}{2}\})$, we have $T_{i_0}=0$ and $T_{(i_0,i)}=Y^{k_0}_{(i_0,i)}=(-Z)_{(i_0,i)}=0$, so the second hypothesis on $\OM$ implies that $T_{i}=0$. Let $i$ where $\sur[Y](i)>\frac{k_0 + 1}{2}$. We assume that $i>i_0$, the other case is analogous. We have that $T_{(i_0,i)}=Y^{k_0}_{(i_0,i)}\neq 0$ and $T_{i_0}=0$, so $T_{i}\neq 0$ by the second hypothesis. This forces that $T_i=+$ as otherwise $T\circ Z$ would satisfy $(T\circ Z)_{i_0}=-(T\circ Z)_{i}=(T\circ Z)_{(i_0,i)}$, which contradicts our assumption on the reorientation.

To conclude, if $k_0$ is even, then $Y^{k_0+1} = Y^{k_0}\circ (-Z)$.
\end{proof}

We get the following characterization as a direct corollary.
\begin{theorem}\label{thm:modularcharacterization}
 A simple oriented matroid~$\OM$ is FL-isomorphic to a big oriented matroid if and only if it has a tight modular hyperplane.
\end{theorem}

 \begin{proof}
 It is straightforward to check that in a big oriented matroid the elements indexed by~$\ipairs$ form a modular hyperplane that is tight up to the simplification of parallel elements.
 
 For the converse, let $E$ be the ground set of $\OM$, and $F\subseteq E$ a tight modular hyperplane. We will relabel the elements of $E\ssm F$ by $\ivl$, where $n=|E\ssm F|$. Now, for each $(i,j)\in \ipairs$ there is an element $z\in F$ in the line spanned by $i$ and $j$ by the modularity of~$F$. We add to $\OM$ an element parallel to $z$ labeled by $(i,j)\in \ipairs$. We obtain this way an isomorphic oriented matroid~$\OM'$. Note that, since the modular hyperplane $F\subseteq E$ is tight, for each $z\in F$ there are some $i,j\in E\ssm F$ such that $i,j,z$ are collinear. Hence, $z$ is parallel to $(i,j)$ and $\OM'\ssm z$ is isomorphic to~$\OM$. We conclude that $\restr{\OM'}{\ivl\cup\ipairs}$ is isomorphic to $\OM$. It satisfies the conditions of~\cref{prop:characterizationbig} and is hence isomorphic to a big oriented matroid.
 
 \end{proof}

A consequence of this observation is that we can extend the process to determine the big oriented matroid from the sweep oriented matroid to any oriented matroid with a modular hyperplane (not necessarily tight). For sweep oriented matroids, this relies on the labeling of the elements (see \cref{rmk:isomorphism}). Arbitrary modular hyperplanes also need a similar extra information. Let $\OM$ be an oriented matroid on a ground set~$E$  with a modular hyperplane~$F$. To simplify the exposition, we will assume that~$\OM$ is simple (no loops or parallel elements), that $E\ssm F=\ivl$, that $F\cap \ipairs=\emptyset$, and that all the elements of~$E\ssm F$ lie in a common halfspace defined by~$F$. (We could omit this simplification by adding information to the decoration, but it unnecessarily complicates the notation.)

We will \defn{decorate} the elements in~$F$ by constructing maps $\delta: F\to 2^{\ipairs}$ and $\epsilon:\ipairs\to\{+,-\}$ that associate a subset of elements of $\ipairs$ to each $f\in F$ and a sign to each pair in $\ipairs$. 
This is done with the following algorithm.
We start decorating each element in~$F$ with an empty set. For every $(i,j)\in \ipairs$, let $f\in F$ be the element of~$F$ in the flat spanned by $i$ and~$j$. We add to the decoration~$\delta(f)$ of~$f$ the ordered pair $(i,j)$; and we set $\epsilon(i,j)=+$ if there is a covector $X\in \OM$ such that $X_i=0$ and $X_j=X_f\neq 0$,
or  $\epsilon(i,j)=-$ otherwise. We will call this information the~\defn{decoration of $F$ induced by $\OM$}. 

We will show that we can recover $\cM$ from $\OM'=\restr{\OM}{F}$, its restriction to~$F$, and the decoration. To state our result, we introduce \defn{valid decorations}, which are those that can be obtained with the procedure above. For any simple oriented matroid~$\OM'$ on the ground set~$F$, we call a \defn{valid decoration} a couple of maps~$\delta:F\to 2^{\ipairs}$ and $\epsilon:\ipairs\to\{+,-\}$
for a certain~$n$, such that:
\begin{itemize}[leftmargin=0.5cm]
\item the decorations form a partition of $\ipairs$, with empty parts accepted: $\ipairs = \bigcup_{f\in F} \delta(f)$ with $ \delta(f)\cap \delta(f')=\emptyset$ whenever $f\neq  f'$; and
\item the covectors $X\in\OM$, seen as elements of $\{+,-,0\}^{\ipairs}$ by considering $X_{(i,j)}=\epsilon(i,j)X_f$ if $(i,j)\in \delta(f)$, satisfy the transitivity condition from \cref{lem:OMtransitivity}.
\end{itemize}

The following result should be seen as the oriented version of \cite[Thm.~2.1]{Bonin06}, which similarly characterizes when an (unoriented) matroid can be extended so that its ground set is a modular hyperplane of the larger matroid. We have deferred its proof to \Cref{sec:appendix}, since it relies on the proof of \cref{thm:BOMisOM}.

\begin{corollary}\label{cor:decorations}
If $\OM'$ is a simple oriented matroid on~$F$ with a valid decoration~$(\delta,\epsilon)$, then $\OM'$ can be extended to a unique oriented matroid~$\OM$ for which $F$ is a modular hyperplane and $(\delta,\epsilon)$ is the decoration of $F$ induced by $\OM$. 

In particular, an oriented matroid $\OM$ with a modular hyperplane~$F$ is completely determined by $\restr{\OM}{F}$ together with the decoration of~$F$ induced by~$\OM$. 
\end{corollary}

\subsection{Not every oriented matroid is a little oriented matroid}\label{sec:extendability}

Little oriented matroids are always acyclic, meaning that $\pluses$ is a tope. 
A first guess could be that all acyclic oriented matroids can be extended to a big oriented matroid. After all, this is trivially the case for realizable oriented matroids. Moreover, it is also true for rank~$3$ oriented matroids. Although stated in a different language, this follows directly from \cite[Thm.~6.3.3]{BLSWZ99} and \cite[Lemma~1]{FelsnerWeil2001}\footnote{This is usually presented in the context of ``topological sweepings'' of arrangements of pseudolines, for example in~\cite{FelsnerWeil2001,Felsner2004}. Note that the notation in these references collides slightly with ours, see~\cref{sec:terminology}.}%
, which was first proved in the uniform case in~\cite{SnoeyinkHershberger1991}. (Actually, their result is stronger, as the sweep oriented matroid they construct is Dilworth in the sense of the upcoming~\cref{sec:Dilworth}.)

\begin{theorem}[{\cite[Thm.~6.3.3]{BLSWZ99}}]\label{thm:sweepabilityRank3}
Every loopless acyclic oriented matroid $\OM$ of rank~$3$ is the little oriented matroid of a sweep oriented matroid.
\end{theorem}

However, contrary to the rank~$3$ case, starting at rank~$4$ there exist acyclic oriented matroids that cannot be extended to big oriented matroids. The proof of \cref{thm:sweepabilityRank3} in~\cite{BLSWZ99} uses Levi's extension lemma, that states that every arrangement of pseudolines can be extended with an extra pseudoline through two given points. We use a famous counterexample to the analogous statement in rank~$4$ by Richter-Gebert~\cite{RichterGebert1993} to present an acyclic oriented matroid that cannot be extended to a big oriented matroid.

\begin{theorem}[{\cite[Cor.~3.4]{RichterGebert1993}}]
 There is an oriented matroid $\RG$ of rank $4$ with ground set~$\ivl[12]$ with two %
topes $U$ and $T$ such that no extending pseudoplane intersects $U$ and $T$ simultaneously.
\end{theorem}

This means that if $\RG'$ is an oriented matroid on $\ivl[12]\cup \{f\}$ such that $\restr{\RG'}{\ivl[12]}=\RG$, then it cannot contain covectors $U', T'\in \RG'$ such that $\restr{U'}{\ivl[12]}\preceq U$ and $\restr{T'}{\ivl[12]}\preceq T$ but $U'_f=T'_f=0$.

\begin{theorem}\label{thm:unextendable}
The reorientation of $\RG$ sending $U$ to $\pluses[12]$ is acyclic, but it is not %
the little oriented matroid of any sweep oriented matroid.
\end{theorem}

\begin{proof}
After a suitable reorientation, assume that $U=\pluses[12]$.
Suppose that there is a big oriented matroid~$\OM$ on $\ivl[12]\cup \ipairs[{\ivl[12]}]$ such that $\restr{\OM}{\ivl[12]}=\RG$. It contains a cocircuit $U'\in \OM$ with $U'_i=U_i=+$ for all $i\in \ivl[12]$ and $U'_{(i,j)}=0$ for all $(i,j)\in \ipairs[{\ivl[12]}]$. 

Let $X$ be a covector in $\RG$ such that $[X, T]$ forms an interval of length $2$ in the face lattice of $\RG$. 
This means that there are $1\leq i_0 < j_0 \leq 12$ such that $X_{i_0}=X_{j_0}=0$ and $X_i\preceq T_i$ for all $i\in \ivl[12]\setminus\{i_0, j_0\}$.
Let $X'$ be a covector in $\OM$ such that $\restr{X'}{\ivl[12]}=X$. Hence, we have  $X'_{(i_0,j_0)}=0$ and $\restr{X'}{\ivl[12]}\preceq T$. 
Hence $\RG'=\restr{\OM}{\ivl[12]\cup \{(i_0,j_0)\}}$ is an extension of $\RG$ whose covectors $\restr{U'}{\ivl[12]\cup\{(i_0,j_0)\}}$ and $\restr{X'}{\ivl[12]\cup\{(i_0,j_0)\}}$ contradict the special property of $\RG$.
\end{proof}

\section{Lattices of flats of sweep oriented matroids}\label{sec:flats}
\subsection{Dilworth sweep oriented matroids}\label{sec:Dilworth}

It is also interesting to understand the underlying (unoriented) matroid $\un{\MOM}$ associated to a sweep oriented matroid~$\MOM$. In particular, because it plays an essential role in the enumeration of sweeps~\cite[Sec.~4.6]{BLSWZ99}. In the realizable case, this was done by Edelman~\cite{Edelman2000} and Stanley~\cite{Stan15}, who showed that, under certain genericity constraint, $\un{\MOM}$ can be obtained from $\un{\LOM}$ via the operation of \defn{Dilworth truncation}.

We will work directly with the axiomatic of (unoriented) matroids in terms of \defn{geometric lattices} of flats, which was already mentioned in \cref{sec:modularhyperplanes}. We refer to \cite{White86} for a comprehensive reference on (unoriented) matroids.

Recall that if $\OM$ is an oriented matroid on ground set $E$, a \defn{flat} of $\cov$ is a subset $F\subseteq E$ that is the zero-set of a covector of $\cov$ (there is $X\in \cov$ such that $F=\{e\in E \, |\, X_e=0\}$). The set~$\flatsM$ of all flats of $\cov$, ordered by inclusion, has the special structure of a \defn{geometric lattice}; that is, a finite atomistic semimodular lattice. If $\OM$ has no loop, its minimal element is $\emptyset$. (Note that this order is reversed from the order on the covectors in the face lattice of $\cov$.) Conversely, any geometric lattice can be seen as the lattice of flats of a matroid. Let $S \subseteq E$. There is only one minimal flat $F$ that contains~$S$. The \defn{rank} of~$S$ is the length of any maximal chain from $\emptyset$ to $F$ in~$\flatsM$. It is denoted $\rk{\cov}{S}$, or $\rk{\un{\cov}}{S}$. The rank function satisfies the \defn{submodular inequality}: \[\rk{\cov}{A}+\rk{\cov}{B}\geq \rk{\cov}{A\cap B}+\rk{\cov}{A\cup B}.\]

The flats and the rank function give two cryptomorphic ways to define the underlying (unoriented) matroid~\defn{$\un{\OM}$} of the oriented matroid~$\OM$.
If $\unVectOM$ is the matroid associated to a real vector configuration $\vc=(\p v_1, \ldots, \p v_n)$, the flats correspond to the sets of vectors in a same linear subspace and the rank of $S\subseteq E$ is the dimension of the linear subspace generated by $\{\p v_i \, |\, i\in S\}$.  

The flats of the braid arrangement~$\Mbraid$ are in correspondence with the (unordered) partitions of~$\ivl$, and the lattice of flats of~$\Mbraid$ is just the lattice of partitions of~$\ivl$. Similarly, each flat of a sweep oriented matroid can be associated to a partition, and the sweeps corresponding to orderings of this partition correspond to the covectors with this zero-pattern.

We will need the oriented and unoriented notions of weak maps, which are the matroidal version of perturbing a configuration to a more special position. If $\OM$ and $\OM'$ are two oriented matroids on the same ground set~$E$, we say that there is a \defn{weak map} from $\OM$ to $\OM'$ if for every covector $X \in \OM'$, there is a covector $Y  \in \OM$ such that $X \preceq Y$. Note that every strong map is also a weak map, but not the other way round (the definition of strong maps is given in \cref{subsec:sweeporientedmatroids}). If $\un{\OM}$ and $\un{\OM}'$ are two unoriented matroids on the same ground set $E$, we say that there is a \defn{weak map} from $\un{\OM}$ to $\un{\OM}'$ if for any subset $F\subseteq E$ we have $\rk{\un{\OM}'}{F}\leq \rk{\un{\OM}}{F}$. Note that a weak map between oriented matroids induces a weak map on the underlying unoriented matroids (cf.\ \cite[Cor.~7.7.7]{BLSWZ99}).

The idea behind the Dilworth truncation is the following: if $\flats$ is a geometric lattice and we remove the elements of rank~$1$, we obtain a poset $\flats'$ that is not necessarily a geometric lattice. The most generic way to augment it with all the joins needed to fulfill the semimodularity condition gives rise to a matroid called the \defn{first Dilworth truncation} of~$\flats$. The construction works in more generality when the elements of rank $\leq k$ are removed, giving rise to the $k$th Dilworth truncation, but we will not need it in such generality (\cite{Dilworth1944}, see also~\cite{Brylawski1986}).

\begin{definition}[{\cite[Prop.~7.7.5]{Brylawski1986}}]\label{def:Dilworth_trunc}
Let $\un{\cov}$ be a matroid on ground set $E$. The \defn{first Dilworth truncation} of $\un{\cov}$, denoted \defn{$\Dil{\un{\cov}}$}, is defined on the ground set $\ipairs[E]$ and its rank function is given by:
\begin{align*}
\rk{\Dil{\un{\cov}}}{\emptyset}&=0,\\
\rk{\Dil{\un{\cov}}}{F} &= \min_{\unop \in \unopSet{F}} r_{\unop}(F) & \text{for } \emptyset \neq F \subseteq \ipairs[E],
\end{align*}
where $\unopSet{F}$ is the set of (unordered) partitions $\unop=\{F_1, \ldots, F_l\}$ of~$F$  ($F=F_1\cup\cdots\cup F_l$, $F_k\neq \emptyset$ for all $k\in [l]$, and $F_k\cap F_h=\emptyset$ for all $k\neq h$) and $r_{\unop}(F)=\big(\sum_{k=1}^l \rk{\un{\cov}}{\bigcup \{i, j\, |\, (i,j)\in F_k\}}\big) -l$.

\end{definition}

The flats of rank~$1$ of $\Dil{\un{\cov}}$  are exactly the flats of rank~$2$ (i.e.\ the lines) of $\un{\cov}$. As noted by Brylawski~\cite{Brylawski1986} and Mason (\cite[Sec.~2.1]{Mason1977}), in the realizable case the Dilworth truncation can be geometrically realized by intersecting all the lines of $\un{\cov}$ with a generic affine hyperplane. If $\pc$ is generic enough (in the sense that incomparable flats spanned by its subsets are never parallel), then the hyperplane at infinity fulfills this genericity condition and $\unMOMA$ is the first Dilworth truncation of $\unLOMA$. Otherwise, we only get a weak map of $\Dil{\unLOMA}$, as $\unMOMA$ will be in less general position. This result extends to (not necessary realizable) sweep oriented matroids.

\begin{theorem}\label{thm:Dilworth_unweakmap}
Let $\cov$ be a sweep oriented matroid on $\ipairs$. Then there is a weak map from $\Dil{\un{\LOM}$} to $\un{\cov}$.
\end{theorem}

The proof needs an auxiliary lemma.

\begin{lemma}\label{lem:flats_lit_sweep}
Let $\LOM$ be the little oriented matroid of the sweep oriented matroid~$\OM$. If $I$ is a flat of $\LOM$ of rank at least two, and $J$ is the minimal flat in $\cov$ that contains $\set{(i,j)}{i,j\in I}$, then $\rk{\OM}{J}= \rk{\LOM}{I}-1$.
\end{lemma}
\begin{proof}
 Let $I'= \smallset{(i,j)\in\ipairs}{i,j\in I}$. Then $\restr{\OM}{I'}$ is a sweep oriented matroid with little oriented matroid~$\restr{\LOM}{I}$, and their respective ranks are $\rk{\LOM}{I}-1$ and $\rk{\LOM}{I}$ by \Cref{lem:rank_BOM}.
 Therefore, $\rk{\OM}{J}=\rk{\LOM}{I}-1$,
 because the rank function of a restriction is just the restriction of the rank function, see \cite[Prop~7.3.1]{Brylawski1986}. 
\end{proof}

\begin{proof}[Proof of \cref{thm:Dilworth_unweakmap}]
We want to show that $\rk{\OM}{G} \leq \rk{\Dil{\un{\LOM}}}{G}$ for every $G\subseteq \ipairs$. 
Let $F$ be a minimal flat of $\Dil{\un{\LOM}}$ that contains $G$, so that $\rk{\Dil{\un{\LOM}}}{F}=\rk{\Dil{\un{\LOM}}}{G}$. 
Then there exists an unordered partition $\{I_1, \ldots, I_l\}$ of a subset of~$\ivl$ into flats of $\LOM$ of rank at least two such that $F=\bigsqcup_{k=1}^l \set{(i,j)}{i,j\in I_k}$ and $\rk{\Dil{\un{\LOM}}}{F}=\sum_{k=1}^l (\rk{\LOM}{I_k} -1)$.

Indeed, let $\unop=\{F_1,\dots,F_l\}$ be a partition of~$F$ that minimizes $r_{\unop}(F)$, and let $I_k=\bigcup\set{i,j}{(i,j)\in F_k}$. The submodular inequality shows that $\rk{\LOM}{I_1\cup I_2} -1\leq \rk{\LOM}{I_1} + \rk{\LOM}{I_1} -2$ whenever $I_1\cap I_2\neq \emptyset$. We can therefore assume that the $I_k$'s are disjoint.
Moreover, these parts $I_k$ have to be flats of $\LOM$. Otherwise, if there was some $e\notin I_k$ such that $\rk{\LOM}{I_k}=\rk{\LOM}{I_k\cup \{e\}}$, then we could add to $F$ all the pairs $(i,e)$ and $(e,i)$ with $i\in I_k$ without augmenting its rank, but $F$ was taken to be a flat.

Let $J_k$ be the minimal flat in $\cov$ that contains $\set{(i,j)}{i,j\in I_k}$; and let $J$ be the join of all the $J_k$ in the lattice of flats of $\OM$.  The submodularity of geometric lattices implies that $\rk{\OM}{J}\leq \sum_{k=1}^l \rk{\OM}{J_k}$. Moreover, such a $J$ contains all the $J_k$, hence it contains $F$, which contains $G$; and therefore $\rk{\OM}{G} \leq \rk{\OM}{J}$.  We conclude by \cref{lem:flats_lit_sweep}, that implies that for any $k$, $\rk{\OM}{J_k}=\rk{\LOM}{I_k}-1$.
\end{proof}

In view of this result, we will say that a sweep oriented matroid~$\OM$ is \defn{Dilworth} if the weak map predicted by \cref{thm:Dilworth_unweakmap} is actually an equality and we have $\un{\cov}=\Dil{\un{\LOM}}$.

This is the case if for any flat $F$ of $\cov$ associated to a partition $I=(I_1, \ldots, I_l)$ we have 
\begin{equation}\label{eq:Dilworthrank}\rk{\cov}{F} =\left(\sum_{k=1}^l \rk{\LOM}{I_k}\right)-l.\end{equation}
                                                                                   
In other words, coplanarities in $\cov$ are induced by coplanarities in $\LOM$. 
For sweep oriented matroids that come from a point configuration, it prevents the case where some subspaces spanned by disjoint subsets of points are parallel.

Note that Dilworth sweep oriented matroids provide an oriented version of the matroid operation of Dilworth truncation. However, contrary to the unoriented case, such a truncation is often not unique and may even not exist, as shown by \cref{thm:unextendable}.

Even if \cref{thm:Dilworth_unweakmap} only works at the level of unoriented matroids, we expect that a stronger statement holds at the level of oriented matroids. The following conjecture is true for sweep oriented matroids of rank~$2$ (by \cite[Thm.~6.3.3]{BLSWZ99}), and for sweep oriented matroids arising from point configurations (it suffices to make a generic projective perturbation that removes unwanted parallelisms).

\begin{conjecture}\label{conj:Dilworth}
For any sweep oriented matroid~$\OM$ there is a Dilworth sweep oriented matroid~$\OM'$ such that there is a weak map from~$\OM'$ to~$\OM$, and $\OM$ and $\OM'$ have the same little oriented matroid.
\end{conjecture}

\subsection{Bounds on the number of sweep permutations}

One motivation for studying the lattice of flats of an oriented matroid is that it completely determines its $f$-vector, as shown by the celebrated Las Vergnas-Zaslavsky Theorem \cite[Thm~4.6.4]{BLSWZ99}.

\begin{theorem}\label{thm:nbtopes}
The number of topes of an oriented matroid $\OM$ only depends on its lattice of flats $\flats$. More precisely, this number is:
\[(-1)^r\chi_{\flats}(-1),\]
where $r$ is the rank of~$\OM$, and $\chi_{\flats}$ is the characteristic polynomial of~$\flats$.
\end{theorem}

We can therefore adapt \cite[Thm.~3.4]{Edelman2000}\footnote{There is a small typo in the statement of \cite[Thm.~3.4]{Edelman2000}, but the correct statement can be recovered from \cite[Cor.~3.2]{Edelman2000} with $d=n-k-1$.} 
and \cite[Thm.~7]{Stan15} to oriented matroids. As noted by Stanley in \cite{Stan15}, for fixed~$r$ the bound is a polynomial in~$n$ of degree~$2(r-1)$.

\begin{theorem}\label{thm:boundsweeps}
Let $\OM$ be a sweep oriented matroid on $\ipairs$ of rank $r$. Then its number of sweep permutations is bounded from above by:
\[\left|\PsetA[\OM]\right|\leq \sum_{i=0}^{\lfloor \frac{r-1}{2}\rfloor} 2 c(n, n-r+1+2i),\]
where the $c(n,n-i)$ are the unsigned Stirling numbers of the first kind.

The equality is obtained for example for realizable sweep oriented matroids that come from generic configurations of~$n$ points in~$\RR^{r-1}$.
\end{theorem}

\begin{proof}
We demonstrate how the proof of \cite[Thm.~3.4]{Edelman2000} and \cite[Thm.~7]{Stan15} extends to our set-up. We repeat the main ideas for the reader's convenience and refer to these references for more details.
We denote by $\Mbool$ the geometric lattice obtained by removing all elements of rank greater than~$r$ from the Boolean lattice on~$\ivl$ and adding a top element. This is the lattice of flats of any generic point configuration of $n$ points in $\RR^{r-1}$. The computation and evaluation of the characteristic polynomial of~$\Dil{\Mbool}$ gives the right hand side of the inequality (see \cite[Co.~3.2]{Edelman2000}), which is the number of topes of any oriented matroid whose lattice of flats is~$\Dil{\Mbool}$ via \cref{thm:nbtopes}. This is the case for the sweep oriented matroids arising from generic configurations.

By~\cite[Cor.~9.3.7]{KungNguyen1986}, it suffices to show that there is a weak map from $\Dil{\Mbool}$ to $\un{\OM}$, because this implies that the coefficients of the characteristic polynomial of $\un{\OM}$ are bounded by those of the characteristic polynomial of $\Dil{\Mbool}$. 
Note that for any subset $F\subseteq \ivl$, we have $\rk{\Mbool}{F}=\min(|F|, r)$. Like in any matroid, $\un{\LOM}$ satisfies $\rk{\un{\LOM}}{F}\leq |F|$, and hence there is a weak map from $\Mbool$ to $\un{\LOM}$. It follows from Definition~\ref{def:Dilworth_trunc} of the Dilworth truncation by its rank function that this induces a weak map from $\Dil{\Mbool}$ to $\Dil{\un{\LOM}}$. 
It follows from \cref{thm:Dilworth_unweakmap} that there is a weak map from $\Dil{\Mbool}$ to $\un{\cov}$.
\end{proof}

\section{Pseudo-sweeps}\label{sec:pseudosweeps}

Even if the little oriented matroid does not change, the poset of sweeps of a point configuration is not invariant under admissible projective transformations (in the sense of~\cite[App.~2.6]{Ziegler1995}). In this section we describe a larger poset, the \defn{poset of pseudo-sweeps}, that contains the sweeps with respect to all possible choices of ``hyperplane at infinity''. It is a poset of cellular strings, and as such it can be defined at the level of oriented matroids. Thus it exists even for those oriented matroids that are not little oriented matroids of any sweep oriented matroid.

\subsection{Pseudo-sweeps}

With the presentation of $\Sp$ as a monotone path polytope introduced in \cref{sec:asFiberPolytope}, we know that sweep permutations of a point configuration~$\pc$ can be interpreted as coherent monotone paths of the zonotope $\Z$ with respect to a linear form (which we called the height). Non-coherent monotone paths also give rise to permutations of the elements of~$\pc$, which we will call \defn{pseudo-sweep permutations}.
They can be read in terms of $k$-sets. A \defn{$k$-set} of~$\pc$ is a $k$-element subset $\pc[S]\subseteq \pc$ for which there is an affine hyperplane strictly separating $\pc[S]$ from $\pc\ssm \pc[S]$. See~\cite[Ch.~11]{Matousek2002} for background.

For simplicity, assume that $\pc =(\p a_1,\dots,\p a_n)\in \RR^{d\times \ivl}$ does not contain repeated points. A \defn{pseudo-sweep permutation} of $\pc$ is a permutation $\sigma\in\Sym$ such that $\set{\p a_{\sigma(i)}}{1\leq i \leq k}$ is a $k$-set for all $1\leq k\leq n$. Note that we are still sweeping with a hyperplane, although we are allowed to slightly change its direction every time the hyperplane hits a point, as long as the new hyperplane does not cross one of the already visited points. 

This point of view can be extended to obtain ordered partitions (and lift the constraint of not having repeated points). Consider a sequence of affine functionals $\af_r(\p x)=\sprod{\p u_r}{\p x}-c_r$ for $1\leq r \leq m$ such that
for each point $\p a_i\in \pc$ there is an $r$ with $\af_r(\p a_i)= 0$, $\af_s(\p a_i)> 0$ for all $s<r$, and $\af_s(\p a_i)< 0$ for all $r<s$; and such that for each $1\leq r\leq m$ there is some $i$ such that $\af_r(\p a_i)= 0$. The sets $I_r=\set{i}{\af_r(\p a_i)=0}$ with $1\leq r\leq m$ form an ordered partition of $\ivl$, which we call a \defn{pseudo-sweep} of~$\pc$.

There is another way to interpret pseudo-sweeps of $\pc$ and monotone paths/cellular strings of~$\Z$ in terms of hyperplane arrangements, which extends to oriented matroids. 

A \defn{gallery} of a hyperplane arrangement (without parallels) is a sequence of chambers (topes) such that adjacent chambers are separated by exactly one hyperplane. More generally, a \defn{gallery} of 
an oriented matroid is a collection of topes $T^0,\dots,T^{m+1}$ such that $\sep[T^i][T^{i+1}]$ is a parallelism class for all~$i$. A gallery is \defn{minimal} if no parallelism class is crossed twice. 
We will work with acyclic oriented matroids
and we will be interested in their minimal galleries from $\pluses$ to its opposite~$\minuses$.

This definition can be relaxed to accept paths that go accross some covectors (other than subtopes). 
A \defn{cellular string} of $\OM$ with respect to $\pluses$ is a sequence of non-tope covectors $(X^1,\dots, X^m)$ that are such that $X^1\circ \pluses=\pluses$,  $X^m\circ \minuses=\minuses$, and $X^i\circ \minuses=X^{i+1}\circ \pluses$ for all~$i$. This notation is consistent with the notion of cellular string for a polytope with respect to a linear functional given in \cref{sec:asFiberPolytope}. Indeed, for a hyperplane arrangement which is the normal fan of a zonotope~$\pol[Z]$, its cellular strings are equivalent to the cellular strings of $\pol[Z]$ with respect to a linear functional that is minimized at the vertex corresponding to~$\pluses$. (Minimal galleries are in correspondence with monotone paths.)

Note that an allowable sequence is just a cellular string on the braid arrangement based at the tope indexed by the permutation $\idperm =(1,2,\dots,n)$, and that its galleries correspond to simple allowable sequences.

\begin{figure}[htpb]
    \centering
    
    \includegraphics[width=.9\linewidth]{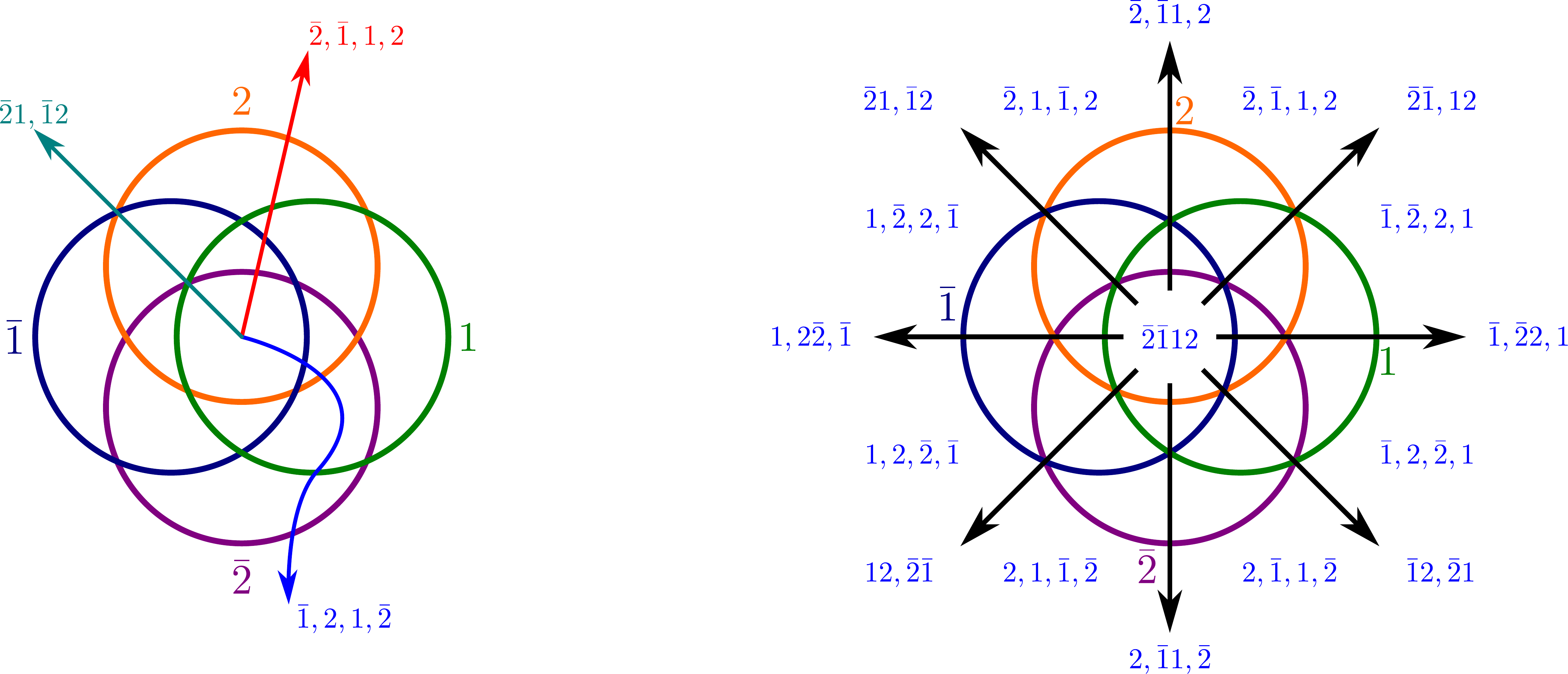}
    \caption{The hyperplane arrangement $\HA[\h{\pc[B]_2}]$. To depict the arrangement, it is intersected with the unit sphere and stereographically projected from the south pole $(0,0,-1)$. We obtain an arrangement of circles, oriented so that the positive side is the interior. Two sweeps, corresponding to the permutation $\bar 2,\bar 1, 1,2$ and the ordered partition $\bar 21,\bar 12$ are depicted; and also the pseudo-sweep that is not a sweep corresponding to the permutation $\bar 1,2,1,\bar 2$. (This resumes the example of \cref{fig:zonotopeandfibers}, where the monotone paths corresponding to these two permutations were depicted.) To represent these pseudo-sweeps, an oriented ray from the all-positive tope (containing the origin) to its opposite (at infinity) is depicted. The order in which the circles are crossed gives the corresponding permutation. If the ray meets more than one circle at the same time, then one recovers an ordered partition. Note that this gives an alternative method to construct the sweep hyperplane arrangement~$\VO[{\pc[B]_2}]$. Indeed, it is not hard to see that when one does this procedure (intersection of $\HA$ with the unit sphere plus stereographic projection), the hyperplanes spanned by the origin and the intersections of all possible pairs of spheres are precisely those of~$\VO$. This is why, under this representation, sweeps correspond to straight rays emanating from the origin. 
    }
    \label{fig:pseudosweepsB2}

\end{figure}

The following lemma sums up the relations between these objects in the realizable case. 
It is illustrated in \cref{fig:pseudosweepsB2}, where the example of~$\pc[B]_2$ from \cref{fig:SweepsB2,fig:zonotopeandfibers} is revisited.

\begin{lemma}
Let $\pc=(\p a_1,\dots,\p a_n)\in\RR^{d\times \ivl}$ be a point configuration; let $\HA$ be the hyperplane arrangement in~$\RR^{d+1}$ composed of the linear hyperplanes $\hyp_{i}=\set{\p x\in\RR^{d+1}}{\sprod{\p x}{\h {\p a}_i}=0}$ (oriented towards $\h {\p a}_i$) for $\p a_i\in \pc$, where $\h {\p a}=(\p a,1)$; and let $\Z = \sum_{i=1}^n [-\h {\p a}_i, \h {\p a}_i]$ be the associated zonotope.

There is a bijection between:
\begin{enumerate}[(i)]
 \item pseudo-sweeps of $\pc$,
 \item cellular strings of $\HA$ with respect to the all-positive tope $\pluses$, and 
 \item $\height$-monotone cellular strings of $\Z$ ($\height$-coherent subdivisions of $\height(\Z)$);
\end{enumerate}
and if moreover $\pc$ does not have repeated points, then there is a bijection between:
\begin{enumerate}[(i)]
 \item pseudo-sweep permutations of $\pc$,
  \item minimal galleries of $\HA$ from the tope $\pluses$ to its opposite $\minuses$, and
 \item $\height$-monotone paths of $\Z$.
\end{enumerate}
\end{lemma}

\begin{proof}
The proof amounts simply to translate between definitions (the definition of cellular strings induced by a projection was given in \cref{sec:asFiberPolytope}). We omit the details and only give some indications.

To a sequence of affine functionals $\af_r(\p x)=\sprod{\p u_r}{\p x}-c_r$ for $1\leq r \leq m$ such that
for each point $\p a_i\in \pc$ there is an $r$ with $\af_r(\p a_i)= 0$, $\af_s(\p a_i)> 0$ for all $s<r$, and $\af_s(\p a_i)< 0$ for all $r<s$; we can associate
\begin{enumerate}[(i)]
 \item the ordered partition $I_1,\dots,I_m$ of $\ivl$ given by $I_r=\set{i}{\af_r(\p a_i)=0}$, 
 \item the sequence of non-tope covectors $X^1,\dots,X^m$ obtained by considering the sign of evaluating $\af_r$ on each of the points of $\pc$, and
 \item the sequence $\pol[F]_1,\dots,\pol[F]_m$  of faces of~$\Z$, where $\pol[F]_r$ is the face of~$\Z$ minimized by the linear functional $\lf_r:\RR^{d+1}\to \RR$ given by $(\p x,x_{d+1})\mapsto \sprod{\p u_r}{\p x}-c_rx_{d+1}$.
\end{enumerate}
One can easily check that the conditions imposed on $\af_1,\dots,\af_m$ imply that these sequences are a pseudo-sweep of $\pc$, a cellular string of $\HA$ with respect to the all-positive tope $\pluses$, and a $\height$-monotone cellular string of $\Z$, respectively. 
And conversely, for any pseudo-sweep or cellular string of $\HA$ or $\Z$, one can find such a sequence of affine functionals. This is direct for pseudo-sweeps and cellular strings of $\HA$. For cellular strings $\pol[F]_1,\dots,\pol[F]_m$ of $\Z$, we associate to each face $\pol[F]_r$ an affine map $\af_r$ obtained by restricting 
the linear functional minimized by $\pol[F]_r$ in $\Z$ to the hyperplane $x_{d+1}=1$.

The map that associates the partition $I_1,\dots,I_m$ to the sequence $X^1,\dots,X^m$  with $(X^r)_i=0$ if $i\in I_r$, $(X^r)_i=-$ if $i\in I_s$ with $s<r$ and $(X^r)_i=+$ if $i\in I_s$ with $s>r$, is hence a bijection between pseudo-sweeps and cellular strings of~$\HA$. And similarly the map that sends a cellular string $X^1,\dots,X^m$ of $\HA$ to the cellular string $\pol[F]_1,\dots,\pol[F]_m$ of $\Z$ given by
\[\pol[F]_r=\sum_{(X^r)_i=+}\{-\h {\p a}_i\}+\sum_{(X^r)_i=-}\{\h {\p a}_i\}+\sum_{(X^r)_i=0}[-\h {\p a}_i, \h {\p a}_i]\]
is also a bijection.

The second part of the statement arises from the observation that these bijections are order-preserving.
\end{proof}

In particular, we can define pseudo-sweeps of a realizable oriented matroid in terms of its cellular strings. We extend this definition to abstract oriented matroids.

\begin{definition}
A \defn{pseudo-sweep} of an acyclic oriented matroid~$\OM$ is an ordered partition $(I_1,\dots, I_m)$ arising from a cellular string $(X^1,\dots,X^m)$ of~$\OM$ via $I_i=\sep[X^i\circ \pluses][X^i\circ \minuses]$, that is, $I_i$ is the set of zeros of $X^i$.
\end{definition}

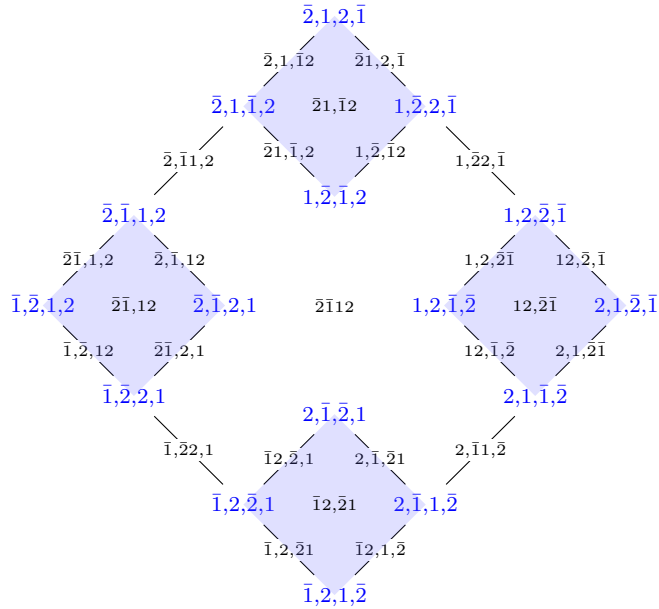
\begin{figure}[htpb]
 \centering
 
\begin{tikzpicture}[scale=.8,vertex/.style={color=blue, font=\footnotesize}, edge/.style={color=black, font=\tiny}]
     \node[vertex] (-2-112) at (0.00,0.00) [draw,draw=none] {$\bar{2},\! \bar{1},\! 1,\! 2$};
     \node[vertex] (-2-121) at (1.50,-1.50) [draw,draw=none] {$\bar{2},\! \bar{1},\! 2,\! 1$};
     \node[vertex] (-21-12) at (1.80,1.80) [draw,draw=none] {$\bar{2},\! 1,\! \bar{1},\! 2$};
     \node[vertex] (-212-1) at (3.30,3.30) [draw,draw=none] {$\bar{2},\! 1,\! 2,\! \bar{1}$};
     \node[vertex] (-1-212) at (-1.50,-1.50) [draw,draw=none] {$\bar{1},\! \bar{2},\! 1,\! 2$};
     \node[vertex] (-1-221) at (0.00,-3.00) [draw,draw=none] {$\bar{1},\! \bar{2},\! 2,\! 1$};
     \node[vertex] (-12-21) at (1.80,-4.80) [draw,draw=none] {$\bar{1},\! 2,\! \bar{2},\! 1$};
     \node[vertex] (-121-2) at (3.30,-6.30) [draw,draw=none] {$\bar{1},\! 2,\! 1,\! \bar{2}$};
     \node[vertex] (1-2-12) at (3.30,0.30) [draw,draw=none] {$1,\! \bar{2},\! \bar{1},\! 2$};
     \node[vertex] (1-22-1) at (4.80,1.80) [draw,draw=none] {$1,\! \bar{2},\! 2,\! \bar{1}$};
     \node[vertex] (12-2-1) at (6.60,-0.00) [draw,draw=none] {$1,\! 2,\! \bar{2},\! \bar{1}$};
     \node[vertex] (12-1-2) at (5.10,-1.50) [draw,draw=none] {$1,\! 2,\! \bar{1},\! \bar{2}$};
     \node[vertex] (2-1-21) at (3.30,-3.30) [draw,draw=none] {$2,\! \bar{1},\! \bar{2},\! 1$};
     \node[vertex] (2-11-2) at (4.80,-4.80) [draw,draw=none] {$2,\! \bar{1},\! 1,\! \bar{2}$};
     \node[vertex] (21-2-1) at (8.10,-1.50) [draw,draw=none] {$2,\! 1,\! \bar{2},\! \bar{1}$};
     \node[vertex] (21-1-2) at (6.60,-3.00) [draw,draw=none] {$2,\! 1,\! \bar{1},\! \bar{2}$};
     \draw[edge] (-2-112) --  node[pos=0.5,color=black,fill=white, inner sep=1pt] {$\bar{2},\! \bar{1},\! 12$} (-2-121);
     \draw[edge] (-2-112) --  node[pos=0.5,color=black,fill=white, inner sep=1pt] {$\bar{2},\! \bar{1}1,\! 2$} (-21-12);
     \draw[edge] (-2-112) --  node[pos=0.5,color=black,fill=white, inner sep=1pt] {$\bar{2}\bar{1},\! 1,\! 2$} (-1-212);
     \draw[edge] (-2-121) --  node[pos=0.5,color=black,fill=white, inner sep=1pt] {$\bar{2}\bar{1},\! 2,\! 1$} (-1-221);
     \draw[edge] (-21-12) --  node[pos=0.5,color=black,fill=white, inner sep=1pt] {$\bar{2},\! 1,\! \bar{1}2$} (-212-1);
     \draw[edge] (-21-12) --  node[pos=0.5,color=black,fill=white, inner sep=1pt] {$\bar{2}1,\! \bar{1},\! 2$} (1-2-12);
     \draw[edge] (-212-1) --  node[pos=0.5,color=black,fill=white, inner sep=1pt] {$\bar{2}1,\! 2,\! \bar{1}$} (1-22-1);
     \draw[edge] (-1-212) --  node[pos=0.5,color=black,fill=white, inner sep=1pt] {$\bar{1},\! \bar{2},\! 12$} (-1-221);
     \draw[edge] (-1-221) --  node[pos=0.5,color=black,fill=white, inner sep=1pt] {$\bar{1},\! \bar{2}2,\! 1$} (-12-21);
     \draw[edge] (-12-21) --  node[pos=0.5,color=black,fill=white, inner sep=1pt] {$\bar{1},\! 2,\! \bar{2}1$} (-121-2);
     \draw[edge] (-12-21) --  node[pos=0.5,color=black,fill=white, inner sep=1pt] {$\bar{1}2,\! \bar{2},\! 1$} (2-1-21);
     \draw[edge] (-121-2) --  node[pos=0.5,color=black,fill=white, inner sep=1pt] {$\bar{1}2,\! 1,\! \bar{2}$} (2-11-2);
     \draw[edge] (1-2-12) --  node[pos=0.5,color=black,fill=white, inner sep=1pt] {$1,\! \bar{2},\! \bar{1}2$} (1-22-1);
     \draw[edge] (1-22-1) --  node[pos=0.5,color=black,fill=white, inner sep=1pt] {$1,\! \bar{2}2,\! \bar{1}$} (12-2-1);
     \draw[edge] (12-2-1) --  node[pos=0.5,color=black,fill=white, inner sep=1pt] {$1,\! 2,\! \bar{2}\bar{1}$} (12-1-2);
     \draw[edge] (12-2-1) --  node[pos=0.5,color=black,fill=white, inner sep=1pt] {$12,\! \bar{2},\! \bar{1}$} (21-2-1);
     \draw[edge] (12-1-2) --  node[pos=0.5,color=black,fill=white, inner sep=1pt] {$12,\! \bar{1},\! \bar{2}$} (21-1-2);
     \draw[edge] (2-1-21) --  node[pos=0.5,color=black,fill=white, inner sep=1pt] {$2,\! \bar{1},\! \bar{2}1$} (2-11-2);
     \draw[edge] (2-11-2) --  node[pos=0.5,color=black,fill=white, inner sep=1pt] {$2,\! \bar{1}1,\! \bar{2}$} (21-1-2);
     \draw[edge] (21-2-1) --  node[pos=0.5,color=black,fill=white, inner sep=1pt] {$2,\! 1,\! \bar{2}\bar{1}$} (21-1-2);

     \fill [blue!60, opacity=0.2] (-1-212.center) -- (-2-112.center) -- (-2-121.center) -- (-1-221.center) --  cycle;
     \fill [blue!60, opacity=0.2] (-121-2.center) -- (2-11-2.center) -- (2-1-21.center) -- (-12-21.center) --  cycle;
     \fill [blue!60, opacity=0.2] (1-2-12.center) -- (-21-12.center) -- (-212-1.center) -- (1-22-1.center) --  cycle;
     \fill [blue!60, opacity=0.2] (12-1-2.center) -- (21-1-2.center) -- (21-2-1.center) -- (12-2-1.center) --  cycle;

     \path (-1-212) -- (-2-121) node[midway,edge] (-2-1,12) {$\bar{2}\bar{1},\! 12$};
     \path (-12-21) -- (2-11-2) node[midway,edge] (-12,-21) {$\bar{1}2,\! \bar{2}1$};
     \path (12-1-2) -- (21-2-1) node[midway,edge] (12,-2-1) {$12,\! \bar{2}\bar{1}$};
     \path (-21-12) -- (1-22-1) node[midway,edge] (-21,-12) {$\bar{2}1,\! \bar{1}2$};

     \path (-1-212) -- (21-2-1) node[midway,edge] (-2-112) {$\bar{2}\bar{1}12$};

\end{tikzpicture}
\caption{The pseudo-sweeps of the point configuration~$\pc[B]_2$. Without the trivial sweep, they index a non-pure cellular complex that retracts to the boundary of the sweep polytope $\Sp[{\pc[B]_2}]$ from \cref{fig:SweepsB2}, a $1$-sphere.}
\end{figure}

\begin{remark}
 If $\pc'$ is a (full-dimensional) admissible projective transformation of $\pc$, then any sweep of~$\pc'$ gives rise to a pseudo-sweep of $\pc$. Indeed, under an admissible projective transformation a pencil of parallel hyperplanes is mapped into a pencil of hyperplanes containing a codimension~$2$ flat that does not intersect $\conv(\pc)$. The $k$-sets defined by these hyperplanes clearly give rise to a pseudo-sweep. However, not all pseudo-sweeps arise this way. For example, if $\{\p a_1,\dots, \p a_6\}$ are the vertices of a regular hexagon in cyclic order, then 
 $[1,2,3,6,5,4]$ is a pseudo-sweep permutation that is not a sweep of any of its projective transformations. (Because in every realization the vector $\p a_6-\p a_3$ is a positive linear combination of the vectors $\p a_1-\p a_2$ and $\p a_5-\p a_4$.)
\end{remark}

\begin{remark}[{Pseudo-sweeps and shellings}]

One of Stanley's motivations for studying sweep permutations in~\cite{Stan15} is that they are in correspondence with Bruggesser-Mani line-shelling orders of polytopes~\cite{BruggesserMani1971}. For a convex polytope $\pol$ and a line~$\ell$ through its interior, this is the order in which the facets of $\pol$ become visible to a point following~$\ell$ from the interior of~$\pol$ to infinity, plus the order in which the remaining facets lose visibility when the point returns from the opposite side to the interior of~$\pol$ along~$\ell$. Now, let $\polar{\pol}$ be the polar of $\pol$ with respect to an interior point~$\p p$ of $\pol$, and let $\ell$ be a line through $\p p$. (Here, we are considering the usual projective polarity, as in~\cite[Sec.~5.1]{Matousek2002}, but after a translation by $-\p p$.) Since $\ell$ contains $\p p$, which is mapped to the hyperplane at infinity by polarity, the set of points in $\ell$ corresponds to a family of parallel affine hyperplanes orthogonal to a common direction. The shelling order given by~$\ell$ coincides with the sweep permutation of the vertices of $\polar{\pol}$ with respect to this direction. Thus, sweep permutations of a point configuration in convex position are in bijection with line shelling orders of the polar polyhedron for lines that go through the center of polarity (here, the origin, which is the image of the hyperplane at infinity).

Actually, not only sweeps, but all pseudo-sweeps, give rise to shelling orders. And this is true in the more general level of oriented matroids. Indeed, every pseudo-sweep of $\OM$ induces a shelling order of the (Edmonds-Mandel) face lattice of the tope~$\pluses$~\cite[Sec.~3.VI]{MandelPhD}, see also~\cite[Sec.~4.3]{BLSWZ99}. (To the best of our knowledge, it is still an open problem whether the opposite of this lattice, called the Las Vergnas face lattice, is shellable.) Pseudo-sweep shellings have been recently rediscovered by Heaton and Samper in the special case of matroid polytopes under the name of \defn{broken line shellings}~\cite{HeatonSamper2020}.
\end{remark}

\subsection{The poset of pseudo-sweeps and the generalized Baues problem}
Just like sweeps, pseudo-sweeps can be naturally ordered by refinement. We denote by \defn{$\PScom$} the poset of pseudo-sweeps of~$\OM$.
Topological properties of this poset have been studied in the context of a special case of the \emph{generalized Baues problem} (GBP) of Billera and Sturmfels~\cite{BS92} concerning the homotopy of the poset of subdivisions induced by a projection of polytopes; see~\cite{Reiner1999} for a nice survey.
We recall that by the topology of a poset~$P$ we mean the topology of its \defn{order complex}~$\OC[P]$: the simplicial complex whose simplices are the chains of~$P$ (see~\cite{Bjorner1995} or~\cite[Sec.~4.7]{BLSWZ99}).

Billera, Kapranov and Sturmfels~\cite[Thm.~2.3]{BilleraKapranovSturmfels1994} showed that the strong version of the GBP, as considered in \cite[Q.~2.3]{Reiner1999}, holds for monotone paths of polytopes. This implies that, in the realizable case, the poset of sweeps of a point configuration is a deformation retract of the poset of pseudo-sweeps. For the case of zonotopes, Bj\"orner~\cite[Thm.~2]{Bjorner1992} gave an alternative combinatorial proof for the weak version of the GBP (in the sense of~\cite[Q.~2.2]{Reiner1999}) that extends to oriented matroids. Namely, he proved that the poset of pseudo-sweeps of an oriented matroid is homotopy equivalent to a sphere (once the trivial sweep $(\ivl)$ is removed). A further generalization to shellable CW-spheres, for an appropriate definition of cellular strings induced by shellings, was proven in~\cite{AER2000}.

\begin{theorem}[{\cite[Thm.~2]{Bjorner1992}}]\label{thm:homotopypseudo-sweeps}
 The poset of pseudo-sweeps of an oriented matroid~$\OM$ of rank~$r$ with respect to a tope~$T$ without the trivial sweep has the homotopy type of an $(r-2)$-sphere.
\end{theorem}

Note that, by \cref{cor:topologyPosetSweeps}, for oriented matroids that admit a sweep oriented matroid (in the sense that they are the little oriented matroid of some sweep oriented matroid) the poset of sweeps is an explicit $(r-2)$-sphere embedded in the poset of pseudo-sweeps. 
We will show that it is in fact a deformation retract; thus proving the strong GBP for cellular strings of little oriented matroids.
In the realizable case, this holds by~\cite[Thm.~2.3]{BilleraKapranovSturmfels1994}. In the more general case, Bj\"orner also remarks that he expects the poset of pseudo-sweeps to retract to a subcomplex homeomorphic to a $(r-2)$-sphere~\cite[below Thm.~2]{Bjorner1992}, but does not provide a candidate subcomplex.

\begin{theorem}\label{thm:retract}
Let $\LOM$ be the little oriented matroid of a sweep oriented matroid $\MOM$. Then the poset of sweeps of $\MOM$ is a strong deformation retract of the poset of pseudo-sweeps of~$\LOM$; and the poset of non-trivial sweeps is a strong deformation retract of the poset of non-trivial pseudo-sweeps.
\end{theorem}

The proof of~\Cref{thm:retract} needs some auxiliary results concerning (combinatorial) homotopy theorems. We refer to~\cite{Bjorner1995} for a very good introduction to the topic. 
First, we present a result that allows us to weaken the statement to prove, as a consequence of the fact that the \defn{homotopy extension property} holds for order complexes of subposets (c.f.~\cite[Ch.~0]{Hatcher2002}). Then we recall three results on the homotopy type of posets: the Carrier Lemma, Quillen's Fiber Theorem and Babson's Lemma (the last two being corollaries of the first one). Next, inspired by~\cite{AER2000}, we use the function that returns the first part of an ordered partition to show the contractibility of some subsets of pseudo-sweeps and sweeps, thanks to Babson's Lemma. Finally, we combine all these results to prove that the inclusion induces a homotopy equivalence.

The first result that we need shows that it suffices to prove a weaker statement, namely that the inclusion is a homotopy equivalence. A \defn{CW pair} of a cell complex (such as a simplicial complex) is a pair $(X,A)$ consisting of a cell complex~$X$ and a subcomplex~$A$. In particular, if~$S$ is a subposet of~$P$, then $(\OC[P],\OC[S])$ is a CW pair.

\begin{lemma}[{{\cite[Prop.~0.16 and Cor.~0.20]{Hatcher2002}}}]\label{lem:inclusionretract}
If $(X,A)$ is a CW pair and the inclusion $A\hookrightarrow X$ is a homotopy equivalence, then~$A$ is a strong deformation retract of~$X$.
\end{lemma}

We will use the following version of the Carrier Lemma, from~\cite{Bjorner1995}. For a simplicial complex~$\SC$ and a space~$T$, let $C:\SC\to2^T$ be an order-preserving map ($C(\sigma)\subseteq C(\tau)$ for all $\sigma\subseteq \tau$). A mapping $f:\norm{\SC}\to T$ is \defn{carried} by $C$ if $f(\norm{\sigma})\subseteq C(\sigma)$ for all $\sigma\in \SC$, where $\norm{\cdot}$ denotes the associated geometric realization of the simplicial complex.

\begin{lemma}[{Carrier Lemma~\cite[Lem.~10.1]{Bjorner1995}}]\label{lem:carrier}
Let $C:\SC\to2^T$ be an order-preserving map such that $C(\sigma)$ is contractible for all $\sigma\in\SC$. If $f,g:\norm{\SC}\to T$ are both carried by~$C$, then $f$ and~$g$ are homotopy equivalent, $f\sim g$.
\end{lemma}

We will also need Quillen's Fiber Theorem~\cite{Quillen1978}. 
For a poset $Q$ and $x\in Q$, let $Q_{\geq x}=\set{y\in Q}{y\geq x}$. For the claim about the carrier, see the proof in~\cite[Thm.~10.5]{Bjorner1995}.

\begin{theorem}[Quillen's Fiber Theorem~\cite{Quillen1978}]\label{thm:Quillen}
Let $f:P\to Q$ be an order-preserving map of posets. If $f^{-1}(Q_{\geq x})$ is contractible for all $x\in Q$, then $f$ induces a homotopy equivalence between~$\OC[P]$ and~$\OC[Q]$ whose homotopy inverse is carried by $C(\sigma)=f^{-1}(Q_{\geq \min \sigma})$.
\end{theorem}

For this variant of Quillen's Fiber Theorem, known as Babson's Lemma~\cite[Lem.~1 in Sec. 0.4.3]{BabsonPhD}, see also~\cite[Lem.~3.2]{SturmfelsZiegler1993}.

\begin{lemma}[{Babson's Lemma~\cite{BabsonPhD}}]\label{lem:babson}
If an order-preserving map of posets~$f : P \rightarrow Q$ fulfills
\begin{enumerate}
\item[(i)] $f^{-1}(x)$ is contractible for all $x \in Q$, and
\item[(ii)] $f^{-1}(x) \cap P_{\geq y}$ is contractible for all $x \in Q$ and~$y \in P$ with~$f(y) \leq x$,
\end{enumerate}
then $f$ induces a homotopy equivalence between $\OC[P]$ and $\OC[Q]$.
\end{lemma}

Moreover, we will need the following lemmas certifying the contractibility of certain subsets of pseudo-sweeps and sweeps.
If $F\subseteq \ivl$ is the zero-set of a non-negative covector~$Z$ of~$\LOM$, we denote by $\Scom[\MOM]_{\subseteq F}$ the sets of sweeps $(I_1,\dots,I_m)$ with $I_1\subseteq F$. Similarly, we denote by $\PScom[\LOM,\pluses]_{\subseteq F}$ the sets of pseudo-sweeps $(I_1,\dots,I_m)$ with $I_1\subseteq F$.

\begin{lemma}\label{lem:contractible}
Let $F\subseteq \ivl$ be the zero-set of a non-negative covector~$Z$ of $\LOM$, then $\PScom[\LOM,\pluses]_{\subseteq F}$ is contractible.
\end{lemma}

\begin{proof}
The proof of \cite[Lem.~5.5]{AER2000} can be adapted to prove that $\PScom[\LOM,\pluses]_{\subseteq F}$ is contractible. 
First, we note that with the same proof we can make a slightly stronger statement. Namely, they define a map 
$f:\omega(P, \cO, a)\to D(P, \cO,a)$, between certain posets $\omega(P, \cO, a)$ and $D(P, \cO,a)$ that we describe below, and show that it induces a homotopy equivalence. However, the exact same proof also shows that $f:\omega(P, \cO, a)\cap f^{-1}(I)\to I$ induces a homotopy equivalence for any order ideal (lower set)~$I$ of~$D(P, \cO,a)$.

To match their notations, we call $P$ the poset opposite to the big face lattice of $\LOM$ (the atoms of $P$ are the topes of $\LOM$ and its $1$-skeleton is the tope graph) and $\cO$ the orientation of the tope graph that goes from $\minuses$ to $\pluses$. 
For a tope $a$, the poset $\omega(P, \cO, a)$ is the poset of partial cellular strings ending at $a$ (i.e.\ sequences of non-tope covectors $(X^1,\dots, X^m)$ such that $X^1\circ \minuses=\minuses$,  $X^m\circ \pluses=a$, and $X^i\circ \pluses=X^{i+1}\circ \minuses$ for all~$i$). 
Therefore taking $a=a_{max}=\pluses$ we have that $\omega(P,\cO, a_{max})=\omega(P,\cO)$ is exactly the poset of cellular strings of~$\LOM$ with respect to~$\minuses$, which is in bijection with~$\PScom[\LOM, \pluses]$. However, their partial order is the opposite of our refinement order and the cellular strings have to be read in reverse order.
The poset $D(P, \cO, a)$ is the poset of the non-tope covectors $X$ such that $X\circ \pluses = a$. Therefore, $D(P, \cO, a_{max})$ corresponds to the half-interval $[\zero,\pluses)$ in the face lattice of $\LOM$.

If we take $I$ the lower set of $D(P, \cO, a_{max})$ corresponding to the interval $[Z, \pluses)$, their function $f:\omega(P, \cO, a_{max})\cap f^{-1}(I)\to I$ corresponds to the function that sends the pseudo-sweep~$(I_1, \ldots, I_m) \in \PScom[\LOM, \pluses]_{\subseteq F}$ to the non-negative covector $Y\in [Z, \pluses)$ with zero-set $I_1$. 
Hence it induces a homotopy equivalence from $\PScom[\LOM,\pluses]_{\subseteq F}$ to $[Z,\pluses)$, which has a contractible order poset because it has a unique minimal element. 
\end{proof}

We wish to prove the same when restricted to sweeps. For this, we use an auxiliary result from~\cite{BCK18}. Let $\cov \subseteq \{+,-, 0\}^E$ be the set of covectors of an oriented matroid on~$E$. Then, each element $e\in E$ defines two \defn{halfspaces} $\set{X\in \cov}{X_e=+}$ and~$\set{X\in \cov}{X_e=-}$, and a \defn{hyperplane} $\set{X\in \cov}{X_e=0}$.

\begin{lemma}\label{lem:supertopes}
Let $\cov$ be the set of covectors of an oriented matroid. Then, any non-empty intersection of one or more halfspaces and hyperplanes, seen as a subposet of the face lattice, is contractible.
\end{lemma}

\begin{proof}
This is a consequence of~\cite[Prop.~15]{BCK18}. Indeed, an intersection of halfspaces and hyperplanes is a COM, because it satisfies Face symmetry and Strong elimination, see~\cite[Def.~1]{BCK18}.
\end{proof}

\begin{lemma}\label{lem:contractible_sweeps}
Let $F\subseteq \ivl$ be the zero-set of a non-negative covector~$Z$ of $\LOM$, then $\Scom[\MOM]_{\subseteq F}$ is contractible.
\end{lemma}

\begin{proof}
Inspired by the proof of \cite[Lem.~5.5]{AER2000}, we apply Babson's \Cref{lem:babson} with the function~$f$ from the subposet of sweeps $\Scom[\MOM]_{\subseteq F}$ to the half-open interval of the face lattice $[Z,\pluses)$ that sends a sweep $(I_1, \ldots, I_m)$ to the non-negative covector with zero-set $I_1$. 

Let $Y$ be a covector in $[Z, \pluses)$, with zero-set $G\subseteq F$.
\begin{enumerate}
\item[(i)] $f^{-1}(Y)$ is the set of sweeps whose first part is $G$. It is not empty because $Y$ must be of the form $\tilde{Y}^1$ for a covector $\tilde{Y}\in \MOM$ (in the sense of \Cref{def:def_big_OM}), and such $\tilde{Y}$ corresponds to a sweep with first part $G$. Moreover, $f^{-1}(Y)$ is the intersection of halfspaces $\set{X\in \MOM}{X_{(i,j)}=+}$ for all $i\in G$, $j\notin G$ and $i<j$, and $\set{X\in \MOM}{X_{(i,j)}=-}$ for all $i\in G$, $j\notin G$ and $i>j$. By \Cref{lem:supertopes} it is contractible.

\item[(ii)] Let $J=(J_1, \ldots, J_r)$ be a sweep in $\Scom[\MOM]_{\subseteq F}$ such that $f(J)\leq Y$, i.e.\ $G\subseteq J_1$. 
The intersection $f^{-1}(Y)\cap (\Scom[\MOM]_{\subseteq F})_{\geq J}$ is the set of sweeps that refine $J$ and whose first part is $G$. As for $f^{-1}(Y)$, this set is an intersection of halfspaces. It is not empty because it contains the sweep corresponding to $J\circ \tilde{Y}$. Hence it is contractible.
\end{enumerate}

It follows from Babson's Lemma that $\Scom[\MOM]_{\subseteq F}$ is homotopy equivalent to $[Z, \pluses)$, which has a contractible order poset because it has a unique minimal element.
\end{proof}

\begin{proof}[Proof of \Cref{thm:retract}]
To simplify the exposition, we denote by $P=\PScom[\LOM,\pluses]$ the poset of pseudo-sweeps of $\LOM$ with respect to $\pluses$, by $S=\Scom[\MOM]$ the poset of sweeps of $\MOM$, and by $Q=[\zero,\pluses)$ the half-open interval between $\zero$ and~$\pluses$ in the face lattice of~$\LOM$ (this is its Edmonds-Mandel lattice without the top element).

By \Cref{lem:inclusionretract} it suffices to show that the inclusion map $\iota: S\hookrightarrow P$ induces a homotopy equivalence. As in the proof of \Cref{lem:contractible,lem:contractible_sweeps}, let~$f:P\to Q$ be the map that sends a pseudo-sweep $(I_1,\dots, I_m)$ to the non-negative covector with zero-set~$I_1$.

\begin{center}
\begin{tikzcd}
S=\Scom[\MOM] \arrow[hookrightarrow,r, "\iota" ] & P \arrow[d, "f"]=\PScom[\LOM,\pluses]\\ 
& Q=[\zero,\pluses) \arrow[lu, "g" ]
\end{tikzcd}
\end{center}

For any covector $Z\in Q$ with zero-set $F$, we have that $(f\circ \iota)^{-1}(Q_{\geq Z})= \Scom[\MOM]_{\subseteq F}$, which is contractible by \Cref{lem:contractible_sweeps}.
 
 We conclude by Quillen's \Cref{thm:Quillen} that $f\circ \iota:S \to Q$ induces a homotopy equivalence with a homotopy inverse $g: Q \to S$ carried by $C(\sigma)=(f\circ \iota)^{-1}(Q_{\geq \min \sigma})$.

 We will show that $g\circ f : P \to S$ is a homotopy inverse of the inclusion map $\iota: S \hookrightarrow  P$. We trivially have that $g\circ f \circ \iota\sim \id_{S}$ from the fact that $(f\circ \iota)$ and $g$ are homotopy inverses.

It remains to show that $\iota \circ g \circ f\sim \id_{P}$. 
Now, for $\sigma$ in the order complex of $P$, let $C'(\sigma)=\norm{f^{-1}(Q_{\geq \min f(\sigma)})}$. Note that $f^{-1}(Q_{\geq \min f(\sigma)})$ is of the form $\PScom[\LOM,\pluses]_{\subseteq F}$ where $F$ is the first part of the smallest ordered partition in~$\sigma$. It is therefore contractible by \Cref{lem:contractible}.
We claim that $\id_p$ and $\iota \circ g \circ f$ are both carried by $C'$, and thus that they must be homotopy equivalent by \Cref{lem:carrier}. Indeed, $\id_P$ is trivially carried by~$C'$; and so is $\iota \circ g \circ f$ because $g$ is carried by $C$. 

The same proof works if we restrict to non-trivial sweeps in~$S$ and~$P$.
\end{proof}

\section{Allowable graphs of permutations and sweep acycloids}\label{sec:sweepacycloids}

In this section we present an alternative generalization of allowable sequences to high dimensions that is closer to the original formulation, in terms of {moves} between permutations. As we will see, the resulting objects naturally have the structure of {acycloids}, and we recover sweep oriented matroids as a special case.

\subsection{Allowable graphs of permutations}

In this setting it is useful to see a permutation $\sigma \in \Sym$ as the word $[\sigma(1), \ldots, \sigma(n)]$ on the alphabet $[n]$. A \defn{substring} of $\sigma$ is then a contiguous sequence of characters, of the form $[\sigma(j), \sigma(j+1), \ldots, \sigma(k)]$ for certain $1\leq j<k\leq n$. Such a substring is said to be \defn{increasing} if $\sigma(j) < \sigma(j+1) < \ldots < \sigma(k)$.

\begin{definition}
Let $\Pset\subseteq \Sym$ be a set of permutations, and $\sigma, \sigma' \in \Pset$. We define an \defn{allowable sequence} in $\Pset$ from $\sigma$ to $\sigma'$ as a sequence of permutations of $\Pset$: $\sigma = \sigma_0, \ldots, \sigma_l=\sigma'$ such that
\begin{description}
 \item[(M1)\label{it:ASaxiomTRANS}]for each $1 \leq k \leq l$ the \defn{move} from $\sigma_{k-1}$ to $\sigma_k$ consists of reversing a set $m_k$ of one or more disjoint substrings of $\sigma_{k-1}$;
 \item[(M2)\label{it:ASaxiomDISJ}] 
 each pair $i,j$ is reversed at most once along the path. In other words, there is at most one move $m_k$ such that $i$ and $j$ are in the same substring of $m_k$.
\end{description}

A move is \defn{simple} if it consists of a single substring of two elements; and an allowable sequence is \defn{simple} if all its moves are.
\end{definition}

For example, 
$(1, 3, 2, 6, 5, 4)\xrightarrow{[{\color{Green}3, 2}], [{\color{red}6, 5, 4}]} (1,{\color{Green} 2, 3}, {\color{red}4, 5, 6})$  and 
$(6, 5, 4, 3, 1, 2)\xrightarrow{[{\color{Cyan}1, 2}]}(6, 5, 4, 3, {\color{Cyan}2, 1})$ are valid moves, the second being moreover simple. The sequence 
\[(1,2,3,4,5)\xrightarrow{[{\color{olive}1,2,3}]}({\color{olive}3,2,1},4,5)\xrightarrow{[{\color{purple}1,4}]}(3, 2, {\color{purple}4, 1}, 5)\xrightarrow{[{\color{orange}2, 4}], [{\color{blue}1, 5}]}(3, {\color{orange}4, 2}, {\color{blue}5, 1})\] 
is an allowable sequence from $(1,2,3,4,5)$ to $(3, 4, 2, 5, 1)$ in~$\Sym[5]$; whereas 
\[(1,2,3,4,5)\xrightarrow{[{\color{olive}1,2,3}]}({\color{olive}3,2,1},4,5)\xrightarrow{[{\color{purple}1,4}]}(3, 2, {\color{purple}4, 1}, 5)\xrightarrow{[{\color{brown}3,2,4}],[{\color{blue}1,5}]}({\color{brown}4, 2, 3}, {\color{blue}5, 1})\] 
is not an allowable sequence in $\Sym[5]$, because the pair $\{2, 3\}$ is reversed twice. In fact, in an allowable sequence from the identity permutation, only increasing substrings can be reversed. 
Note that if there is a move $m$ from $\sigma$ to $\gamma$, then from $\gamma$ to $\sigma$ there is the \defn{reverse move} $\overline{m}$ whose substrings are $\overline{s}=[s_k, \ldots, s_0]$ for each substring $s=[s_0, \ldots, s_k]$ of $m$.
This way, every allowable sequence can be reversed.

Another way to describe allowable sequences is by looking at the set of pairs that are reversed at each move. For a permutation $\sigma$, we denote by $\inv(\sigma)$ its set of inversions; that is, the set of pairs $(i,j)\in \ipairs$ such that $i<j$ and $\sigma(i)>\sigma(j)$. 
We denote by $\symdif$ the symmetric difference operation on sets.

\begin{definition}\label{def:invset} 
If there is a move $m$ from a permutation $\sigma$ to a permutation $\gamma$, we define the \defn{set of inversions} of the move $m$ by {\defn{$\invset_{m}$}$=\inv(\sigma)\symdif \inv(\gamma)$}.
\end{definition}
For example, for the move $(1, 3, 2, 6, 5, 4)\xrightarrow{[{\color{Green}3, 2}], [{\color{red}6, 5, 4}]} (1,{\color{Green} 2, 3}, {\color{red}4, 5, 6})$ we obtain the set of inversions $\left\{ (2,3), (4,5), (4,6), (5,6) \right\}$.

The conditions defining allowable sequences become:
\begin{description}
\item[(M1')\label{it:NEWASaxiomTRANS}] if $(a,b)$ or $(b,a)$ is in $\invset_{m_k}$ and $(b,c)$ or $(c,b)$ is in $\invset_{m_k}$, then $(a,c)$ or $(c,a)$ is in $\invset_{m_k}$;
\item[(M2')\label{it:NEWASaxiomDISJ}] the inversion sets $\invset_{m_k}$ are pairwise disjoint.
\end{description}
Note also that $\invset_{m}=\invset_{\overline{m}}$.

\begin{remark}
An \defn{allowable sequence} in the sense of Goodman and Pollack in~\cite{GP80,GP82,GP84,GP93} is exactly what we call an allowable sequence from $\idperm=(1, 2, \ldots, n)$ to $\overline \idperm = (n, n-1, \ldots, 1)$ in $\Sym$.
\end{remark}

We need to introduce another concept before our main definition.
\begin{definition}
A set of permutations $\Pset\subseteq\Sym$ is \defn{symmetric} if $\overline{\sigma} \in \Pset$ for all $\sigma \in \Pset$, where
$\overline{\sigma}$ is the \defn{reverse} of $\sigma$, defined by $\sigma(t)=\overline{\sigma}(n-t+1)$ for all~$t\in \ivl$.
\end{definition}

\begin{definition}\label{def:allowablegraphpermutations}
Consider a set of permutations $\Pset\subseteq\Sym$  and a set $\moves$ of moves such that:

\begin{description}
 \item[(P1)\label{it:SPaxiomSYM}] 
    $\Pset$ is symmetric,
\item[(P2)\label{it:SPaxiomAS}] for any $\sigma, \sigma' \in \Pset$, there is an allowable sequence from $\sigma$ to $\sigma'$ whose moves belong to~$\moves$,
 \item[(P3)\label{it:SPaxiomDISJ}] for $m,s\in\moves$, either $\invset_m=\invset_s$ or $\invset_m\cap\invset_s=\emptyset$.
 \end{description}
The graph with vertex set $\Pset$ and whose edges are the pairs of permutations differing by a move in~$\moves$ is an \defn{allowable graph of permutations}.

An allowable graph of permutations is \defn{simple} if $\moves$ consists only of simple moves.

\end{definition}

\begin{lemma}\label{rem:confusionsetgraph}
The graph is completely determined by $\Pset$ and does not depend on $\moves$. 
 More precisely,  {$\sigma, \sigma'\in \Pset$} form an edge if and only if there is no $\sigma''\in \Pset\setminus\{\sigma\}$ 
 such that $\inv(\sigma)\symdif \inv(\sigma'')\subsetneq \inv(\sigma)\symdif \inv(\sigma')$.  
\end{lemma}

\begin{proof}
Suppose that $\sigma, \sigma'\in \Pset$ form an edge. It means that there is a move $m$ in $\moves$ with inversion set $\inv(\sigma)\symdif \inv(\sigma')$. Suppose that $\sigma''\in \moves$ satisfies $\inv(\sigma)\symdif \inv(\sigma'')\subseteq \inv(\sigma)\symdif \inv(\sigma')$. For any move $m'$ in $\moves$ along an allowable sequence from $\sigma$ to $\sigma''$ we have $\invset_{m'}\subseteq \inv(\sigma)\symdif \inv(\sigma'')$, thus $\invset_{m'}\cap \invset_m\neq \emptyset$. 
We deduce from \ref{it:SPaxiomDISJ} that $\invset_{m'}=\invset_m$, thus $\inv(\sigma)\symdif \inv(\sigma'')= \inv(\sigma)\symdif \inv(\sigma')$. 

Reciprocally, suppose that $\sigma, \sigma' \in \Pset$ do not form an edge and let $\sigma''\in \Pset\setminus\{\sigma\}$ be the neighbor of $\sigma$ on an allowable sequence from $\sigma$ to $\sigma'$. By \ref{it:ASaxiomDISJ}, we have $\inv(\sigma)\symdif \inv(\sigma'')\subseteq \inv(\sigma)\symdif \inv(\sigma')$, as if there was a pair in $\inv(\sigma)\symdif \inv(\sigma'')\setminus{\inv(\sigma)\symdif \inv(\sigma')}$, then it would be reversed twice in the allowable sequence: first between $\sigma$ and $\sigma''$ and later between $\sigma''$ and $\sigma'$. Moreover, $\sigma'\neq \sigma''$ because $\sigma$ and $\sigma''$ form an edge, and thus $\inv(\sigma)\symdif \inv(\sigma'')\neq \inv(\sigma)\symdif \inv(\sigma')$.
\end{proof}

We will therefore usually identify~$\Pset$ with the corresponding allowable graph, and directly call $\Pset$ an \defn{allowable graph of permutations}. 

\begin{remark}
The set of moves can be recovered from the graph by gathering all moves between adjacent permutations in the graph.
\end{remark}

\begin{remark}\label{rem:identityinsweeppermset}
If $\Pset$ forms an allowable graph of permutations and $\omega\in\Sym$, then $\omega\circ \Pset = \set{\omega\circ \sigma}{ \sigma \in \Pset}$ is still an allowable graph of permutations.
Sometimes it is convenient to suppose that the identity permutation $\idperm$ belongs to $\Pset$, as Goodman and Pollack did, which can always be obtained by multiplying by an $\omega$ that is the inverse of a permutation in~$\Pset$.
\end{remark}

\begin{remark}\label{rem:axiomDISJ}

Note that in the case of a simple allowable graph of permutations Condition~\ref{it:SPaxiomDISJ} is redundant. 
However, the example of \cref{fig:non_sweep_perm_set} shows that it is necessary in the general case and this is why we needed to fix a set of moves in \cref{def:allowablegraphpermutations}.

In this example, a valid set of moves $\moves$ would necessarily contain all the moves represented with the arrows (and their reverse), which are all the singletons $\{[i,j]\}$ for $(i,j)\in \ipairs$. 
However, in order to satisfy Condition~\ref{it:SPaxiomAS}, $\moves$ also has to contain the move $\{[1, 2], [3, 4]\}$ represented by the dashed segment joining permutations $(3, 4, 5, 2, 1)$ and $(4, 3, 5, 1, 2)$, since there is no other allowable sequence between these two permutations. Indeed, we can see that the edges adjacent to $(3, 4, 5, 2, 1)$ are labeled $[2, 5]$ and $[4, 5]$ but those pairs should not be reversed on an allowable sequence to $(4, 3,  5, 1, 2)$. 
Thus, both conditions~\ref{it:SPaxiomDISJ} and~\ref{it:SPaxiomAS} cannot be satisfied simultaneously. 

We need to have both conditions in order to have the structure of acycloids, as stated in \cref{thm:sweeppermutationsetisacycloid}.
\end{remark}

\begin{figure}[htpb]
\centering
\input{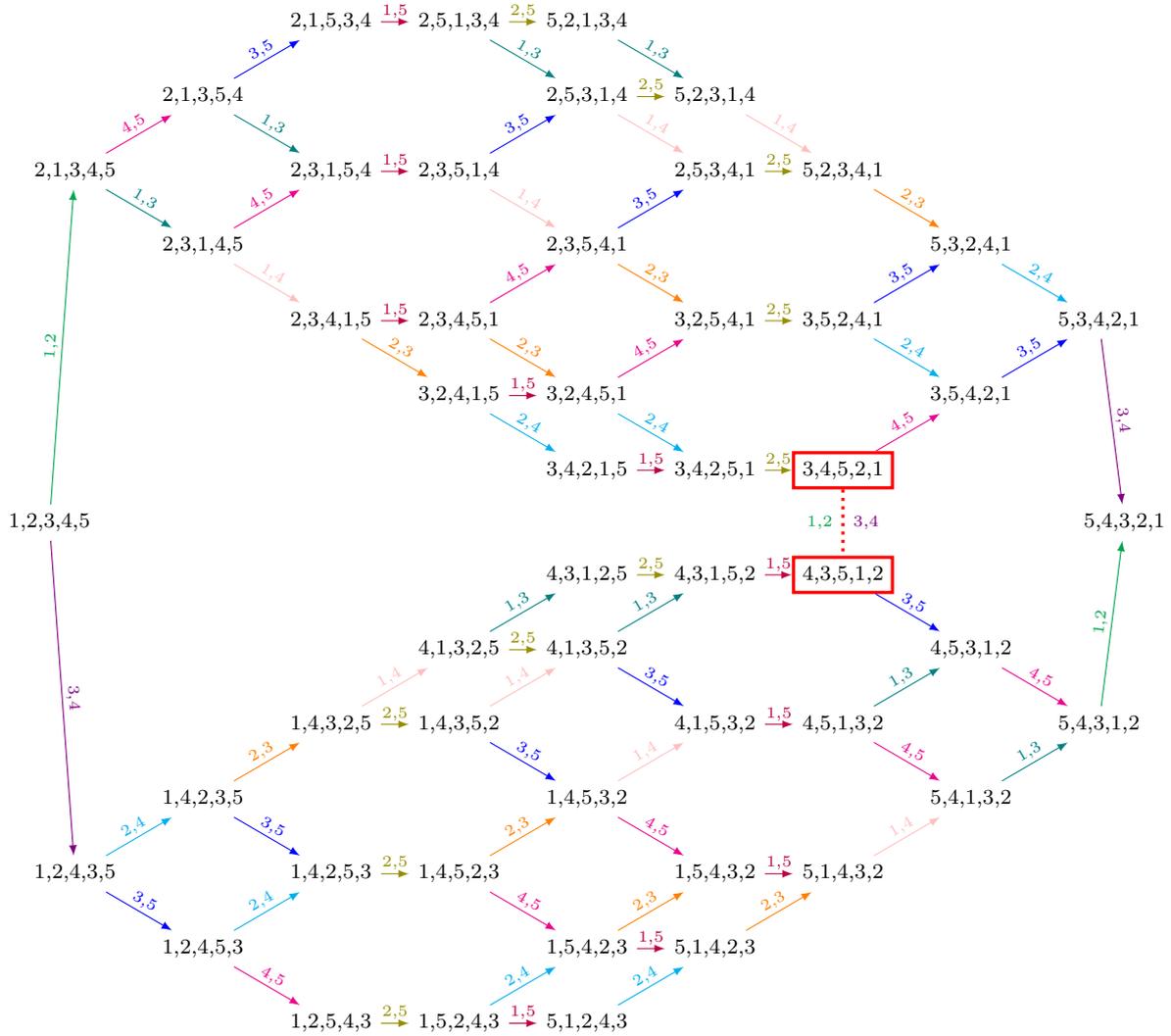}
\caption{
Example of a set of permutations that do not satisfy the definition of allowable graph of permutations.
Neither graphs with or without the dashed segment are partial cubes.}
\label{fig:non_sweep_perm_set}
\end{figure}

\subsection{Sweep acycloids}

Acycloids are combinatorial objects widely studied in connection with the characterization of tope sets of oriented matroids, c.f.~\cite{Handa90,FH93}. They are equivalent to \defn{antipodal partial cubes} (see~\cite{KM20}), a concept well-studied in metric graph theory. A graph is a \defn{partial cube} if it is (isomorphic to) an isometric subgraph of a hypercube graph, and it is \defn{antipodal} (also called \defn{symmetric even}~\cite{BermanKotzig1988}) if for every vertex $v$ there exists exactly one vertex~$\tilde v$, called the antipode of $v$, such that the distance from~$v$ to~$\tilde v$ is larger than the distance from~$v$ to any neighbor of~$\tilde v$.

Following~\cite{Handa90}, we introduce \defn{acycloids} in terms of its topes, which are subsets of sign-vectors. We use the same notation for the notions of reorientation, support and parallelism classes of oriented matroids from \cref{sec:OMdefns}, which carry on verbatim to arbitrary subsets of sign vectors.

\begin{definition}\label{def:axioms_acycloids}
A collection of sign-vectors $\topes \subseteq \{+,-,0\}^E$ is the set of topes of an \defn{acycloid} if and only if  it satisfies the following axioms\footnote{Recall that the parallelism class $\parallelclass{f}$ of $f$ is the set of elements $e\in E$ such that $X_f=X_e$ for all covectors $X$ or $X_f=-X_e$ for all covectors $X$. The reorientation $\reor{X}{F}$ is the signed vector $Z$ such that $Z_f=-X_f$ for all $f\in F$ and $Z_f=X_f$ otherwise. The separation set $S(X, Y)$ of covectors $X, Y$ are the elements $e\in E$ such that $(X_e, Y_e)\in \{(+, -), (-, +)\}$.}:
\begin{description}
 \item[(T1)\label{it:ACaxiomSUP}] 
 $X, Y \in \topes$ implies $\supp{X}=\supp{Y}$ (this set is called the \defn{support} of the acycloid),
 \item[(T2)\label{it:ACaxiomSYM}] $X\in \topes$ implies $-X\in \topes$,
 \item[(T3)\label{it:ACaxiomREOR}] if $X\neq Y \in \topes$ then there exists $f\in \sep$ such that $\reor{X}{\parallelclass{f}}\in\topes$.
 \end{description}
\end{definition}

These three axioms are satisfied by the topes of an oriented matroid but they are not sufficient; there are examples of acycloids that are not oriented matroids, see \cite[Sec.~7]{Handa93}.

To describe the link between allowable graphs of permutations and acycloids, we associate a sign-vector $\sv[\sigma]$ in~$\{+,-\}^{\ipairs}$ to each permutation $\sigma\in \Sym$ via the map~\eqref{eq:covectorfrompartition}. For simplicity, we will sometimes implicitly identify permutations and sign-vectors when it is clear from the context. For a set of permutations $\Pset\subseteq\Sym$, we denote $\topes_\Pset=\set{\sv[\sigma]}{\sigma\in \Pset}\subseteq  \{+,-\}^{\ipairs}$.

\begin{lemma}\label{lem:parallel_classes}
Let $\Pset\subseteq\Sym$ form an allowable graph of permutations, and let $\topes_\Pset\subseteq \{+,-\}^{\ipairs}$ be the set of sign-vectors associated to its permutations. Then the inversion sets of the moves in $\moves$ coincide with the parallelism classes of $\topes_\Pset$.
\end{lemma}

\begin{proof}
First, the fact that $\Pset$ is symmetric and the existence of a valid path between $\sigma$ and $\overline{\sigma}$ for any $\sigma\in \Pset$ implies that any pair $\{i, j\}$ is in the inversion set of at least one move in $\moves$, which is necessarily unique by the disjointness condition~\ref{it:SPaxiomDISJ}. Hence, the inversion sets of the moves in $\moves$ define equivalence classes on the pairs~$\ipairs$. It is straightforward to check that these coincide with the parallelism classes of~$\topes_\Pset$.
\end{proof}

\begin{theorem}\label{thm:sweeppermutationsetisacycloid}
 Let $\Pset\subseteq\Sym$ form an allowable graph of permutations. Then $\topes_\Pset\subset \{+,-,0\}^{\ipairs}$ is the set of topes of an acycloid.
\end{theorem}
\begin{proof}
 The support of all the covectors is~$\ipairs$, and we have symmetry by definition. Hence, it suffices to verify that $\topes_\Pset$ satisfies the reorientation property~\ref{it:ACaxiomREOR}. Let $X, Y \in \topes_\Pset$ and $\sigma, \gamma \in \Pset$ be the associated permutations.  Let $\sigma=\gamma_0, \ldots, \gamma_l=\gamma$ be an allowable sequence from $\sigma$ to~$\gamma$. $\sep$ corresponds to the pairs reversed along this path. Let $Z$ be the sign-vector associated to $\gamma_1$ by the map \eqref{eq:covectorfrompartition}. Then $Z$ is in $\topes_\Pset$ and $Z=\reor{X}{\invset_m}$ where $m$ is the move from $\sigma$ to $\gamma_1$. Lemma \ref{lem:parallel_classes} shows that $\invset_m$ is the parallelism class of any pair $\{i,j\}$ reversed by $m$.
\end{proof}

We can characterize which acycloids arise from allowable graphs of permutations. We do it in a slightly more general context.
\begin{definition}
 A \defn{sweep acycloid} is an acycloid on the ground set~$\ipairs$ such that
\begin{enumerate}[(i)]
\item its topes fulfill the transitivity condition from \cref{lem:OMtransitivity}; namely for every covector~$X$ and every choice of $1\leq i < j < k \leq n$, the triple $(X_{(i,j)},X_{(j,k)},X_{(i,k)})$ is orthogonal to the sign vector $(+,+,-)$, and
\item its parallelism classes verify the transitivity condition~\ref{it:NEWASaxiomTRANS}; namely, 
if $\parallelclass{(i,j)}$ or $\parallelclass{(j,i)}$ coincides with $\parallelclass{(j,k)}$ or $\parallelclass{(k,j)}$, then it also coincides with $\parallelclass{(i,k)}$ or $\parallelclass{(k,i)}$.\end{enumerate}
\end{definition}

As we show in~\cref{prop:sweepacycloidsaresweeppermutationsets} below, sweep acycloids are essentially equivalent to allowable graphs of permutations. The only nuance is that sweep acycloids might have some elements outside its support, which under the map \eqref{eq:covectorfrompartition} would give rise to some partitions that are not permutations. In this case, there would be pairs of elements that belong to the same part in all the partitions. However, up to merging non-singleton parts and relabeling, one can suppose that these maximal ordered partitions are permutations. We recover then an allowable graph of permutations.

These operations of merging and relabeling do not affect the tope-graphs.

\begin{lemma}
 Let $\topes\subseteq \{+,-,0\}^{\ipairs}$ be a sweep acycloid with support $S\subseteq \ipairs$. For $1\leq i<j\leq n$, if $(i,j)\notin S$, then the restriction of $\topes$ to $\ipairs[{\ivl\ssm\{j\}}]$ is a sweep acycloid with isomorphic tope-graph.
\end{lemma}
\begin{proof}
That this restriction is a sweep acycloid is straigthforward from the definition. Moreover, from the characterization in \cref{lem:OMtransitivity} one sees that for $X\in \topes$ and $k\neq i,j$, the values of $X$ on the pairs $(i,k)$ (resp.\ $(k,i)$) and $(j,k)$ (resp.\ $(k,j)$) determine each other uniquely (the sign depending on the relative order of $i,j,k$), because $X_{(i,j)}=0$. 
Therefore, there is a bijection between topes (resp.\ parallelism classes) of $\topes$ and topes (resp.\ parallelism classes) of the restriction.
\end{proof}

If $\topes$ is the tope set of a sweep acycloid, we denote by $\Pset_\topes=\set{\op}{X\in \topes}$ the set of associated ordered partitions.

\begin{theorem}\label{prop:sweepacycloidsaresweeppermutationsets}
If $\Pset\subseteq\Sym$ forms an allowable graph of permutations, then $\topes_\Pset$ is the set of topes of a sweep acycloid. Conversely, if $\topes$ is the tope set of a sweep acycloid of full support~$\ipairs$, then $\Pset_\topes$ forms an allowable graph of permutations.
\end{theorem}
\begin{proof}
 The first claim follows directly from \cref{thm:sweeppermutationsetisacycloid}. Indeed, the topes of the form $\sv[\sigma]$ for a permutation $\sigma\in\Sym$ fulfill the transitivity condition from \cref{lem:OMtransitivity} by construction. Moreover, the parallelism classes of $\topes_\Pset$ are the moves of $\Pset$ by \cref{lem:parallel_classes}, and they
 fulfill condition~\ref{it:NEWASaxiomTRANS} by definition. 
 
 For the second claim, note first that $\Pset_\topes$ is clearly symmetric by \ref{it:ACaxiomSYM}. Following \cref{lem:parallel_classes}, we set $\moves$ to be the moves whose inversion sets are parallelism classes of the topes. By construction, two distinct moves in this family are either disjoint, or they are reverse to each other and have the same set of inversions.
 
 Finally, let $\sigma_X, \sigma_Y\in \Pset_\topes$ be the permutations associated to the topes $X,Y\in \topes$. We will prove that they are joined by an allowable sequence by induction on the cardinality of the symmetric difference of their inversion sets. By the reorientation property \ref{it:ACaxiomREOR}, there is an element $f\in \sep$ such that {$Z=\reor{X}{\parallelclass{f}}\in \topes$.} 
 The parallelism class $\parallelclass{f}$ corresponds to a move $m\in \moves$ such that $\invset_{m}\subseteq \inv_{\sigma_X}\symdif\inv_{\sigma_Y}$. Hence, $Z$ is associated to a permutation $\sigma_Z$ such that $\inv_{\sigma_Z}\symdif \inv_{\sigma_Y} = (\inv_{\sigma_X}\symdif\inv_{\sigma_Y}) \setminus \invset_{m}$. By induction there is an allowable sequence  $\sigma_Z\to \cdots \to\sigma_Y$ with labels in $\moves$. Note that $m$ is not a label of this path because its inversion set is disjoint from $\inv_{\sigma_Z}\symdif \inv_{\sigma_Y}$. Then, $\sigma_X\xrightarrow{m}\sigma_Z\to \cdots \to\sigma_Y$ is an allowable sequence from $\sigma_X$ to $\sigma_Y$.
\end{proof}

\subsection{Sweeps and potential sweeps of sweep acycloids}

With Handa's notation from~\cite{Handa93}, a \defn{face} of an acycloid~$\topes\subseteq \{+,-,0\}^E$ is a sign-vector $X\in\{+,-,0\}^E$ such that $X\circ T\in\topes$ for all $T\in\topes$; and a \defn{coboundary} of $\topes$ is a sign-vector $X\in\{+,-,0\}^E$ that conforms to a tope (which means that there is a tope that refines it) and such that, for every $T\in \topes$ with $X\circ T= T$ we have $X\circ (-T)\in \topes$.
In the language of partial cubes, faces correspond to gated subgraphs, and coboundaries are antipodal subgraphs. In an acycloid, every gated subgraph is antipodal, which shows that every face is a coboundary (see \cite{KM20} for definitions and details). In general, the converse is not true.
However, if $\topes$ is the set of topes of an oriented matroid, then faces and coboundaries coincide, and correspond to the covectors of the oriented matroid.

Augmented with a top element, the set of faces of an acycloid forms a lattice, the \defn{big face lattice } of the acycloid~\cite{Handa93}. Face lattices of acycloids lack many nice properties of those of oriented matroids. In particular, they are not always graded.

We can translate these concepts to sweeps. To this end, define the \defn{composition} $I\circ J$ of two ordered partitions $I=(I_1, \ldots, I_l)$ and $J=(J_1, \ldots, J_{l'})$ of $\ivl$ as 
$$I\circ J = (I_{1, 1}, \ldots, I_{1, r_1}, \ldots, I_{l,1}, I_{l, r_l}),$$
where  for any $k \in \{1, \ldots, l\}$, $(I_{k,1}, \ldots, I_{k,r_k})$ is the sequence $(I_k\cap J_1, I_k\cap J_2, \ldots, I_k\cap J_{l'})$ where the empty parts are removed. That is, the ordered partition of the elements of $I_k$ induced by~$J$. 

\begin{definition} Let $\Pset\subseteq\Sym$ be an allowable graph of permutations.
\begin{itemize}
 \item A \defn{sweep} of~$\Pset$ is an ordered partition $I$ such that $I\circ \sigma\in \Pset$ for all $\sigma\in \Pset$. 

 \item A \defn{potential sweep} of~$\Pset$ is an ordered partition $I$ of $\ivl$ refined by some permutation in~$\Pset$ and such that any sweep permutation $\sigma \in \Pset$ that refines $I$ satisfies $I\circ \overline{\sigma} \in \Pset$. 
\end{itemize}
\end{definition}

\begin{lemma}\label{lem:facescoboundaries}
Let $\Pset\subseteq\Sym$ form an allowable graph of permutations and let $\topes_\Pset$ be its associated sweep acycloid. Then the sweeps of~$\Pset$ are in bijection with the faces of~$\topes_\Pset$ and the potential sweeps of~$\Pset$ are in bijection with the coboundaries of~$\topes_\Pset$.
\end{lemma}

\begin{proof}

 We prove first the equivalence between potential sweeps and coboundaries. It is clear that $\sv$ is a coboundary of~$\topes_\Pset$ for any potential sweep $I$ of~$\Pset$. Indeed, if $\sigma$ refines $I$, it implies that $\sv$ conforms to~$\sv[\sigma]$, i.e.\ $\sv\circ\sv[\sigma]=\sv[\sigma]$. Moreover, $I\circ \overline{\sigma}\in \Pset$ implies that $\sv\circ(-\sv[\sigma])=\sv\circ\sv[\overline{\sigma}]=\sv[I\circ \overline{\sigma}]$ is in~$\topes_\Pset$.
 
 For the converse statement, let $Y$ be a coboundary of $\topes_\Pset$. We need to show that it is of the form~$\sv$ for an ordered partition $I$ of $\ivl$. Then it is clear from the definitions that $I$ is a potential sweep of~$\Pset$. 
 Suppose that there are $1\leq i < j < k\leq n$ such that $(Y_{(i,j)}, Y_{(j,k)}, Y_{(i,k)})$ is one of the forbidden patterns in \cref{lem:OMtransitivity}. Let $\sigma \in \Pset$ be a sweep permutation such that $Z:=Y\circ\sv[\sigma]=\sv[\sigma]$. 
 We denote $\tilde{\sigma}$ the permutation in $\Pset$ such that $\tilde{Z} := Y\circ(-\sv[\sigma]) = \sv[\tilde{\sigma}]$.
 The fact that $Z$ and $\tilde{Z}$ satisfy the transitivity condition implies that the forbidden pattern of $Y$ must be one of the last six ones (with two zeroes). 
 We consider the case $(Y_{(i,j)}, Y_{(j,k)}, Y_{(i,k)})=(0,0,-)$, the other ones are similar. Then we must have $\{(Z_{(i,j)}, Z_{(j,k)}, Z_{(i,k)}), \, (\tilde{Z}_{(i,j)}, \tilde{Z}_{(j,k)}, \tilde{Z}_{(i,k)})\} = \{(+,-,-), (-,+,-)\}$, i.e.\ the elements $i,j,k$ are ordered $k,i,j$ and $j,k,i$ in $\sigma$ and $\tilde{\sigma}$. 
 As a consequence of condition {\ref{it:NEWASaxiomTRANS}}, in any allowable sequence in $\Pset$ from $\sigma$ to $\tilde{\sigma}$, there must be a permutation where the elements $i,j,k$ are ordered $k,j,i$.
 Such $\tau$ satisfies $Y\circ \sv[\tau]=\sv[\tau]$.
 Indeed, any pair $(k,l)$ with $Y_{(k,l)}\neq 0$ satisfies $Z_{(k,l)}=\tilde{Z}_{(k,l)}$, thus it cannot be reversed in an allowable sequence from $\sigma$ to $\tilde{\sigma}$. 
 But then the covector $Y\circ (-\sv[\tau])$ should belong to $\topes_\Pset$ while it has the forbidden pattern $(+,+,-)$. We conclude that any coboundary satisfies the transitivity condition from \cref{lem:OMtransitivity}.

 To finish, it is clear that any sweep $I$ of~$\Pset$ gives a covector $\sv\in \{+,-,0\}^{\ipairs}$ such that for any $\sigma\in \Pset$, $\sv\circ \sv[{\sigma}]=\sv[{I\circ \sigma}]\in \topes_\Pset$, thus $\sv$ is a face of~$\topes_\Pset$.  For the converse, note that any face $Y$ of~$\topes_\Pset$ is a coboundary, and hence it must be of the form $\sv$ associated to a potential sweep~$I$. The condition of being a face shows that this potential sweep is indeed a sweep. 
\end{proof}
Note in particular that the \defn{poset of sweeps} of an allowable graph of permutations, augmented with a top element, is always a lattice, as it is isomorphic to the big face lattice of an acycloid.

\subsection{Sweep oriented matroids from sweep acycloids and allowable graphs of permutations}\label{sec:sweepOMfromsweepacycloids}

The set of topes of an oriented matroid is always an acycloid, but the converse statement is not true. However, the conditions in the definition of sweep acycloid guarantee that, whenever they correspond to an oriented matroid, it is a sweep oriented matroid. 

Note that, for this, the transitivity condition~\ref{it:NEWASaxiomTRANS} on the parallelism classes of sweep acycloids is necessary. Indeed, $(+,+,+), (-,-,-), (-,+,+), (+,-,-)$ satisfy the conditions of \cref{lem:OMtransitivity} (they are orthogonal to $(+,+,-)$) and they are the topes of an oriented matroid, but not a sweep oriented matroid. 
This gives an acycloid whose topes fulfill the transitivity condition from \cref{lem:OMtransitivity} and that arises from an oriented matroid, but that is not a sweep oriented matroid. 
However, thanks to \cref{lem:facescoboundaries}, we know that the conditions on topes and subtopes in the definition of sweep acycloids extend to the whole set of covectors.

\begin{corollary}\label{cor:sweepacycloidmatroidissweepmatroid}
 The set of topes of a sweep oriented matroid is a sweep acycloid. Conversely, if a sweep acycloid is the set of topes of an oriented matroid, then it is a sweep oriented matroid.
\end{corollary}

The following hierarchy summarizes our current knowledge:
\begin{theorem}\leavevmode
\begin{center}
\(\Big\{ \text{Posets of sweeps of point configurations} \Big\} \)

\(\rotatebox{-90}{$\subsetneq$} \)

\(\Big\{ \text{Posets of sweeps of sweep oriented matroids} \Big\}\)

\(\rotatebox{-90}{$\subseteq$}\)

\(\Big\{ \text{Posets of sweeps of sweep acycloids} \Big\}\)

 \end{center}
\end{theorem}

Goodman and Pollack's unrealizable pentagon proves that the first inclusion is strict. For the second inclusion, it is known that there are acycloids that are not oriented matroids, but we do not know of any example that has the additional structure given by the transitivity condition from \cref{lem:OMtransitivity}.

\cref{cor:sweepacycloidmatroidissweepmatroid} allows us to use characterizations of acycloids arising from oriented matroids to characterize which allowable graphs of permutations arise from sweep oriented matroids. We know three families of such characterizations, summarized in~\cite[Cor.~7.2]{KM20}. In the language of permutations, da Silva's characterization {\cite[Thm.~4.1]{daSilva95}} concerns sweeps and potential sweeps. Handa's characterization is stated in terms of contractions. If $\Pset$ is an allowable graph of permutations, and $m\in \moves$ is one of its moves, the \defn{elementary contraction} $\Pset/m$ is obtained by taking all permutations $\gamma\in \Pset$ that are separated from another permutation of~$\Pset$ by $m$, and replacing the substring $m$ by its minimal element. One obtains this way a new set of permutations on the ground set $\ivl\ssm m\cup \{\min(m)\}$. For a collection of moves  $M=\{m_1, \ldots, m_l\}$, the \defn{contraction} \defn{$\Pi/M$}, is defined inductively by $\Pi/M=(((\Pi/m
_1)/m_2)\cdots)/m_l$. The characterization by Knauer and Marc~\cite[Cor.~7.2]{KM20} is in terms of excluded partial cube minors. This operation goes outside the scope of allowable graphs of permutations. We will hence not present its details and refer the reader to the source~\cite{KM20}. 

\begin{corollary}\label{cor:characterizations}
Let $\Pset$ form an allowable graph of permutations. The following conditions are equivalent:
\begin{enumerate}[(i)]
 \item $\Pset$ arises from a sweep oriented matroid,
 \item every potential sweep of $\Pset$ is a sweep,
 \item all its contractions are allowable graphs of permutations,
 \item the graph is in $\cF(\cQ^-)$ in the sense of~\cite{KM20}.
\end{enumerate}
\end{corollary}

These characterizations might be useful to answer the question whether all sweep acycloids are sweep oriented matroids. We have not been able to construct any counterexample, but we do not have any evidence on why the properties defining sweep acycloids should force these conditions to be satisfied.

\begin{question}\label{q:sweepacycloidOM}
Is every sweep acycloid an oriented matroid?
\end{question}

\section{Further directions}\label{sec:further}

\subsubsection*{Elementary homotopies between sweep oriented matroids}

In~\cite{Felsner2004,FelsnerWeil2001} it is proven that if an allowable sequence has two consecutive moves with disjoint support, then these can be merged into a single move and the result is still an allowable sequence; and that conversely, if a move consists of more than one disjoint substrings, these can be split into two disjoint moves. These operations induce an equivalence relation among sweep oriented matroids of rank~$2$ whose equivalence classes are in correspondence with the associated little oriented matroids.

Extending this result to higher rank is closely related to some of the open questions indicated in the paper. First of all, the higher analogue of the operation of merging would consist in collapsing some flats of a sweep oriented matroid to get a flat whose rank is lower than the one expected by~\eqref{eq:Dilworthrank}. The reverse operation would break a flat with unexpected low rank into pieces fulfilling~\eqref{eq:Dilworthrank}. Understanding this procedure would provide a method to prove~\Cref{conj:Dilworth}.

Even if the operations were well described, it is not clear that one could find a connectivity result analogous to that by Felsner and Weil in rank~$2$~\cite{FelsnerWeil2001}. Note that, even if \Cref{thm:retract} goes in this direction, as it shows that all sweep oriented matroids are homotopy equivalent in the complex of pseudo-sweeps, it is not clear that there is a way to do this where all the intermediate steps are also sweep oriented matroids.

\subsubsection*{Are all sweep acycloids oriented matroids?}

Another natural problem that is left open is \Cref{q:sweepacycloidOM}, which asks whether every sweep acycloid is an oriented matroid. The answer would be very interesting in either direction. If it is affirmative, then the two categories of sweep acycloids and sweep oriented matroids would collapse into a single concept. This would make allowable graphs of permutations a useful alternative characterization of sweep oriented matroids. If, on the contrary, the answer is negative, then it would be interesting to understand the gap between the two categories.

We do not have any good reason to conjecture that every sweep acycloid is an oriented matroid, beyond the fact that we could not find any. This does not tell much, because the naive approaches to computationally generate all allowable graphs of permutations of a certain size fail badly very soon because of the rapid growth of these objects. 

\subsubsection*{Allowable graphs in Coxeter groups}

We already saw the hyperoctahedral group~$B_n$ naturally appear before. First, in \Cref{sec:crosspolytope}, because the permutahedron of type~$B$ is the sweep polytope of the crosspolytope. Then also in \Cref{ex:braidBOM} to explain the supersolvability of the associated matroid.
In fact, the definition of allowable graph extends naturally to any Coxeter group, specially in the simple case; namely, a \defn{simple allowable graph of Coxeter permutations} is a symmetric set~$\Pi$ of elements of the Coxeter group, in which for every pair of elements $w,w'\in\Pi$ there is a path from $w$ to $w'$ following a reduced decomposition of~$w^{-1}  w'$. For the non-simple case one has to partition the generators into a collection of disjoint subsets to define the allowable moves.

\subsubsection*{Higher sweep oriented matroids and permutahedra}

\begin{figure}[htbp]
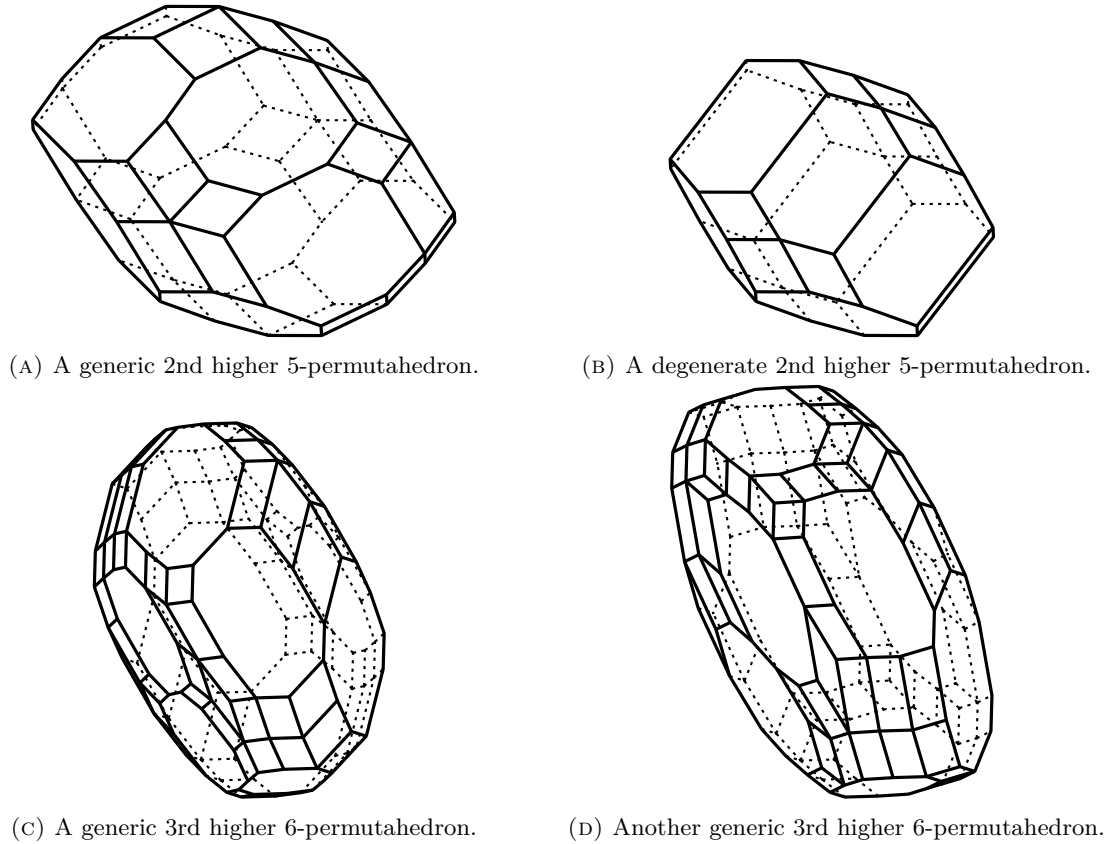

    \centering
    \begin{subfigure}[b]{0.48\textwidth} %
    
        \centering
   \input{figures/TikzFigures/genericP52}
\caption{A generic $2$nd higher $5$-permutahedron.}\label{fig:genericP52}
    \end{subfigure}
    ~
    \begin{subfigure}[b]{0.48\textwidth}

        \centering

\input{figures/TikzFigures/nongenericP52}
        \caption{A degenerate $2$nd higher $5$-permutahedron.}\label{fig:nongenericP52}
    \end{subfigure}

    \begin{subfigure}[b]{0.48\textwidth}
         \centering

\input{figures/TikzFigures/generic3rdhigher6permutahedron}
        \caption{A generic $3$rd higher $6$-permutahedron.}\label{fig:generic3rdhigher6permutahedron}
    \end{subfigure}
    ~
    \begin{subfigure}[b]{0.48\textwidth}
         \centering

\input{figures/TikzFigures/anothergeneric3rdhigher6permutahedron}
        \caption{Another generic $3$rd higher $6$-permutahedron.}\label{fig:anothergeneric3rdhigher6permutahedron}
    \end{subfigure}
    \caption{The first row shows a generic and a degenerate $2$nd higher $5$-permutahedra. The second row depicts two combinatorially different generic $3$rd higher $6$-permutahedra.}\label{fig:higherpermutahedra}
\end{figure}

As we saw in \Cref{sec:Dilworth}, sweep oriented matroids are closely related to the first Dilworth truncation. What about higher truncations? In the realizable case, instead of studying the intersection of the lines spanned by the points of~$\pc$ with a hyperplane (at infinity), we would study the intersection of a flat $F$ of codimension~$k$ (playing the role of hyperplane at infinity) with every flat spanned by $k+1$ points of~$\pc$. In~\cite[Thm.~8]{Stan15}, Stanley states (in the polar formulation) that for a sufficiently generic choice of the flat, this gives rise to an arrangement whose lattice of flats is the $k$th Dilworth truncation of the original arrangement. Let's call this operation the \defn{$k$th Dilworth truncation} of $\pc$ with respect to~$F$.
Doing the $k$th Dilworth truncation of an standard $(n-1)$-simplex gives rise to ``higher'' analogues of braid arrangements, which are the normal fans of the \defn{$k$th higher $n$-permutahedra}. However, in comparison with the $k=1$ case, there is no $\Sym$-invariant subspace that gives a canonical choice for~$F$. Indeed, different choices for~$F$ can give rise to different combinatorial types of hyperplane arrangements and zonotopes, even if the flats are sufficiently generic in the sense of Stanley. See \Cref{fig:higherpermutahedra} for some examples. Nevertheless, every zonotope associated to a $k$th Dilworth truncation of a point configuration still arises as the projection of some $k$th higher permutahedron.

\subsubsection*{Which matroids are little oriented matroids?}

In \Cref{sec:extendability} we proved that not every oriented matroid is a little oriented matroid. This begs the question of which are the oriented matroids that are sweepable, in the sense that they can be extended to a big oriented matroid. Or, at least, to find sufficient conditions. For example, we know that realizable oriented matroids are sweepable, and also all oriented matroids of rank~$3$, by \Cref{thm:sweepabilityRank3}.

As shown in~\cite{Hochstattler2016}, Euclidean oriented matroids (see \cite[Section 10.5]{BLSWZ99}) always admit topological sweepings (see \cref{sec:terminology}). Is there a relation between being Euclidean and being sweepable? Our example of non-sweepable oriented matroid in \cref{sec:extendability} is based on a well-known example of non-Euclidean oriented matroid.

\section*{Acknowledgements}

We are very grateful to Keiichi Handa, who sent us a copy of his Ph.D. thesis manuscript. We also want to thank 
Raul Cordovil, Kolja Knauer, Jean-Philippe Labb\'e, Germain Poullot, Francisco Santos, and Raman Sanyal for their helpful comments on previous versions of this manuscript. 
Finally, we would like to thank the anonymous reviewers for their detailed comments and suggestions to improve our presentation.

\bibliographystyle{amsalpha}
\bibliography{biblio_sweepotopes} 
\label{sec:biblio}

\appendix

\section{Another fiber polytope construction}\label{sec:alternativefiber}
In this section we present another way to construct sweep polytopes as fiber polytopes. As we will see, it is strongly related to the monotone path construction we gave in \cref{sec:asFiberPolytope}.

Define the \defn{Lawrence polytope} of a point configuration $\pc\in \RR^{d\times\ivl}$ as
\[\Law=\conv\set{\p e_i\times (-\h {\p a}_i) ,\p e_i\times \h {\p a}_i}{i\in \ivl}\subset\RR^{n+d+1}.\]
Then the intersection of $\Law$ with the subspace $\p x_{1}=\cdots =\p x_{n}$ is a homothety of the zonotope~$\Z$, and the \defn{Cayley trick} provides a bijection between (regular) subdivisions of $\Law$ and (coherent) zonotopal tilings of~$\Z$, see~\cite[Sec.~9.2]{DRS10}. In fact, the fiber polytopes associated to the canonical projections $\simp[2n-1]\to \Law$ and $\cube\to \Z$ are normally equivalent~\cite[Thm.~5.1]{Sturmfels1994}.

Consider (the vertex set of) the standard $(n-1)$-simplex~$\simp[n-1]$ and the $0$-dimensional configuration~$\pc[O]\in \RR^{0\times\ivl}$ consisting of~$n$ copies of a point. The chain of linear maps
\begin{center}
\begin{tikzcd}[column sep=huge]
\simp[n-1] \arrow[r, "{\lm[\h \pc]}"] & \h \pc \arrow[r, "\height"] &\h {\pc[O]},
\end{tikzcd}
\end{center}
induces chains of projections between the corresponding Lawrence polytopes and associated zonotopes, respectively, that can be arranged in the following commutative diagram:
\begin{center}
\begin{tikzcd}[column sep=huge]
\Z[{\simp[n-1]}] \arrow[r, "{\lm[\h \pc]}"] \arrow[hook,d] \arrow[bend left]{rr}{ s } & \Z\arrow[r, "\height"]\arrow[hook,d] & \Z[{\h {\pc[O]}}]\arrow[hook,d]\\
\Law[{\simp[n-1]}] \arrow[r, "{\id\times\lm[\h \pc]}"] \arrow[bend right]{rr}{\id\times s }& \Law\arrow[r, "\id\times\height"] & \Law[{\h {\pc[O]}}]\\
\end{tikzcd}
\end{center}

Note that $\Z[{\simp[n-1]}]$ is just the cube~$\cube$, and~$\Z[{\h {\pc[O]}}]$ a segment, and hence $\fib[{\Z[{\simp[n-1]}]}][ s ]$ and $\fib[\Z][\height]$
are the $n$-permutahedron and the sweep polytope~$\Sp$ by \Cref{ex:permutahedronfiber} and \Cref{prop:spasmonotonepath}, respectively.

Moreover, $\Law[{\simp[n-1]}]$ and~$\Law[{\h {\pc[O]}}]$ are the (non-standard) $(2n-1)$-simplex~$\conv\set{\p e_i\pm \p e_{i+n}}{i\in\ivl}$ and a prism over~$\simp[n-1]$, respectively.
The same proof as in {{\cite[Thm.~6.2.6]{DRS10}}} shows that $\fib[{\Law[{\simp[n-1]}]}][\id\times s ]$ is a homothety of the $n$-permutahedron embedded into~$\RR^{2n}$. By \Cref{lem:proj_fiber_pol} we obtain that:

\begin{corollary}
The fiber polytope $\fib[\Law][\id\times\height]$
is a homothety of the sweep polytope~$\Sp$ embedded into
$\RR^{n+d+1}$.
\end{corollary}

\section{Proofs of \cref{thm:BOMisOM,cor:decorations}}\label{sec:appendix}

We include below the technical details of the proof of \cref{thm:BOMisOM}. We first recall the notations and the statement of the theorem.

For a covector~$X$ of a sweep oriented matroid, let $\sur[X]:\ivl\to [l_X]$ be the surjection associated to the corresponding ordered partition.
For each $1\leq k \leq  2l_X+1$, let $X^k\in \{+,-,0\}^{\ivl\cup\ipairs}$ be the sign-vector:
\begin{align*}
{X}^k_i &=
\begin{cases}
- &\text{ if } p_X(i)\leq \lfloor \frac{k-1}{2} \rfloor, \\
+ &\text{ if } p_X(i)>\lfloor \frac{k}{2} \rfloor, \\
0 &\text{ if $k$ is even and } p_X(i)=\frac{k}{2}.
\end{cases}&&\text{ for }1\leq i\leq n;\\
{X}^k_{(i,j)} &= X_{(i,j)} &&\text{ for all } 1\leq i < j \leq n.
\end{align*}

\begin{theorem*}[\ref{thm:BOMisOM}]
If $\cov$ is the set of covectors of a sweep oriented matroid, then
\[\BOM[\cov] = \set{{X}^k}{ X\in \cov, \, 1\leq k\leq 2l_X+1}\] is the set of covectors of an oriented matroid.
\end{theorem*}

\begin{proof}
We have to check that $\cov$ satisfies the axioms of \cref{def:axioms_OM}, namely:
\begin{description}
 \item[(V0)] $\zero \in \BOM[\cov]$,
 \item[(V1)] $X\in \BOM[\cov]$ implies $-X\in \BOM[\cov]$,
 \item[(V2)] $X,Y \in \BOM[\cov]$ implies $X\circ Y \in \BOM[\cov]$,
 \item[(V3)] if $X,Y\in \BOM[\cov]$ and $e\in \sep[X][Y]$ then there exists $Z \in \BOM[\cov]$ such that $Z_e=0$ and $Z_f=(X\circ Y)_f$ for all $f\notin S(X,Y)$.
 \end{description}

$(V0)$ $\zeros \in \cov$, associated to the one part ordered partition $(\{1, 2, \ldots, n\})$. Then $(\zeros)^2$ is the zero vector and it is in $\BOM[\cov]$.

$(V1)$ Let $X^k$ be an element of $\BOM[\cov]$. Then, $-X^k= (-X)^{2l_X+2-k}$, so it is still in $\BOM[\cov]$.

$(V2)$ Let $X^k, Y^h$ be two elements of $\BOM[\cov]$. Then $X^k \circ Y^h = (X\circ Y)^t$, where $t= 2(r_1 + \ldots + r_{\frac{k-1}{2}-1})+1$ if $k$ is odd (with the same notations as in the definition of the composition between two ordered partitions), $t=2(r_1 + \ldots + r_{\frac{k}{2}-1})+ j$ if $k$ is even and  $j$ is the index corresponding to $h$ when the elements of $I_k$ are ordered according to $Y$ (that is to say, for all $i \in I_k$, $p_{X\circ Y}(i)\leq \lfloor \frac{t-1}{2} \rfloor \Leftrightarrow p_Y(i)\leq \lfloor \frac{h-1}{2} \rfloor$ and $p_{X\circ Y}(i)> \lfloor \frac{t}{2} \rfloor \Leftrightarrow \lfloor p_Y(i) \rfloor > \lfloor \frac{h}{2} \rfloor$).

$(V3)$ Let $X^k, Y^h$ be two elements of $\BOM[\cov]$, and $e \in \sep[X^k][Y^h]$. It remains to find $Z\in \cov$ and $r \in \{1, \ldots, 2l_Z+1\}$ such that $(Z^r)_e=0$ and $(Z^r)_f=(X^k \circ Y^h)_f$ for any $f \notin \sep[X^k][Y^h]$. $e$ can be of two types: $e=(i,j)$ or $e=i$.

In both cases, it will be convenient to define
\begin{align*}
E_-&=\Big\{ p \mid 1\leq p \leq n\text{ and }
\{(X^k)_p, (Y^h)_p\} \in \{ \{-,-\}, \{0,-\}\} \Big\} \\
&=\Big\{p \in \{1, \ldots, n\} \setminus  \sep[X^k][Y^h] \mid (X^k \circ Y^h)_p=-\Big\},\\
E_+ &= \Big\{p\mid 1\leq p \leq n\text{ and } \{(X^k)_p, (Y^h)_p\} \in \{ \{+,+\}, \{0,+\}\Big\},\\
E_0 &= \Big\{p\mid 1\leq p \leq n\text{ and } \{(X^k)_p, (Y^h)_p\} =\{0,0\}\Big\}.
\end{align*}

Then $E_-\cup E_+ \cup E_0 = \{1, \ldots, n\}\setminus \sep[X^k][Y^h]$ and part of the condition is that $(Z^r)_p=\varepsilon$ for all $p \in E_{\varepsilon}$, $\varepsilon\in \{-, +, 0\}$.

1) If $e=(i,j)$, up to exchanging $X^k$ and $Y^h$, one can suppose that $X_{(i,j)} = -$ and $Y_{(i,j)}=+$.
Let $Z\in \cov$ be given by $(V3)$ on $\cov$. For any $r$ we will have that $(Z^r)_e=0$ and $(Z^r)_f=(X^k \circ Y^h)_f$ for any $f \notin \sep[X^k][Y^h]$ of the form $f=(p,q)$, because in that case, $f$ is an index for $X$ and $Y$ that is not in $\sep[X][Y]$. Can we find $r$ such that $(Z^r)_p=(X^k \circ Y^h)_p$ for any $ 1\leq p \leq n$ such that $p \notin \sep[X^k][Y^h]$ ?
It is sufficient to check that $p_Z(p)<p_Z(q)$ for all $(p,q) \in E_-\times E_+\cup E_-\times E_0 \cup E_0\times E_+$ and $p_Z(p)=p_Z(q)$ for all $(p,q) \in E_0\times E_0$.
$(p,q) \in E_0\times E_0$ and $p<q$ implies that $X_{(p,q)}=Y_{(p,q)}=0$, hence $Z_{(p,q)}=0$ and $p_Z(p)=p_Z(q)$. 

If $E_0\neq \emptyset$, we take $r = 2 p_Z(q)$ for any $q \in E_0$.
Then, we treat the case $(p,q) \in E_-\times E_0$, since the case $(p,q) \in E_0\times E_+$ is similar. If $p<q$, then $\{(X^k)_{(p,q)}, (Y^h)_{(p,q)}\} \in \{\{+,+\}, \{+,0\}\}$ and $Z_{(p,q)}=+$. If $p>q$, then $\{(X^k)_{(q,p)}, (Y^h)_{(q,p)}\} \in \{\{-,-\}, \{-,0\}\}$ and $Z_{(q,p)}=-$. In any case, $p_Z(p)<p_Z(q)$, thus $(Z^r)_p=-$.

If $E_0 = \emptyset$, there may be several possibilities for $r$. The same reasoning as precedently shows that for any $(p,q) \in E_-\times E_+$, $p_Z(p)<p_Z(q)$. Hence there is at least one appropriate $r$ which separates the parts that contain elements in $E_-$ from parts that contain elements in $E_+$.

2) If $e=i$ for some $1\leq i \leq n$, up to exchanging $X^k$ and $Y^h$, one can suppose that $(X^k)_{i} = -$ and $(Y^h)_{i}=+$.

First, we consider the case where $E_0=\emptyset$.
We take $Z=X\circ Y$ and $r=2p_{X\circ Y}(i)$ (corresponding to the part of $i$ in $Z$). It only remains to check that if $p\in E_{-}$ (resp. $E_+$), than $(Z^r)_p=-$ (resp. $+$).

$p \in E_- \Rightarrow p_Y(p)<p_Y(i) \Rightarrow p_Z(p)<p_Z(i) \Rightarrow (Z^r)_p=-$,

$p \in E_+ \Rightarrow p_X(p)>p_X(i) \Rightarrow p_Z(p)>p_Z(i) \Rightarrow (Z^r)_p=+$.

If $E_0\neq \emptyset$, let $j$ be the smallest element of $E_0$. Than $p_X(i)<p_X(j)$ and $p_Y(i)>p_Y(j)$, thus $(i,j) \in \sep[X][Y]$. Let $Z \in \cov$ be given by axiom $(V3)$ applied to $\cov$ with $X, Y$ and $(i,j)$. Than, for any $k \in E_0$ other than $j$, $Z_{(j,k)}=0$ because $X_{(j,k)}=0$ and $Y_{(j,k)}=0$ (resp.\ $Z_{(k,j)}=0$ because $X_{(k,j)}=0$ and $Y_{(k,j)}=0$), and thus $Z_{(i,k)}=0$ (resp.\ $Z_{(k,i)}=0$), because $Z_{(i,j)}=0$ and $\cov$ satisfies the transitivity condition from \cref{lem:OMtransitivity}. We choose $r=2p_{Z}(i)$ (corresponding to the part of $Z$ that contains $i$ and all $k\in E_0$). Then:
\begin{align*}
p \in E_- &&\Rightarrow&&
\begin{cases}
p_X(p)<p_X(j) \\
p_Y(p)\leq p_Y(j)
\end{cases} \text{or }
\begin{cases}
p_X(p)= p_X(j) \\
p_Y(p)< p_Y(j)
\end{cases}
&&\Rightarrow&& p_Z(p)<p_Z(j) &&\Rightarrow&& (Z^r)_p=-,\\
p \in E_+ &&\Rightarrow&&
\begin{cases}
p_X(p)>p_X(i) \\
p_Y(p)\geq p_Y(j)
\end{cases} \text{or }
\begin{cases}
p_X(p)=p_X(i) \\
p_Y(p)>p_Y(j)
\end{cases}
&&\Rightarrow&& p_Z(p)>p_Z(j) &&\Rightarrow&& (Z^r)_p=+.
\end{align*}
\end{proof}

For the proof of \cref{cor:decorations}, recall that for any simple oriented matroid~$\OM'$ on the ground set~$F$, we call a \defn{valid decoration} a couple of maps~$\delta:F\to 2^{\ipairs}$ and $\epsilon:\ipairs\to\{+,-\}$
for a certain~$n$, such that:
\begin{itemize}[leftmargin=0.5cm]
\item the decorations form a partition of $\ipairs$, with empty parts accepted: $\ipairs = \bigcup_{f\in F} \delta(f)$ with $ \delta(f)\cap \delta(f')=\emptyset$ whenever $f\neq  f'$; and
\item the covectors $X\in\OM$, seen as elements of $\{+,-,0\}^{\ipairs}$ by considering $X_{(i,j)}=\epsilon(i,j)X_f$ if $(i,j)\in \delta(f)$, satisfy the transitivity condition from \cref{lem:OMtransitivity}.
\end{itemize}

\begin{corollary*}[\ref{cor:decorations}]
If $\OM'$ is a simple oriented matroid on~$F$ with a valid decoration~$(\delta,\epsilon)$, then $\OM'$ can be extended to a unique oriented matroid~$\OM$ for which $F$ is a modular hyperplane and $(\delta,\epsilon)$ is the decoration of $F$ induced by $\OM$.

In particular, an oriented matroid $\OM$ with a modular hyperplane~$F$ is completely determined by $\restr{\OM}{F}$ together with the decoration of~$F$ induced by~$\OM$.
\end{corollary*}
\begin{proof}
The proof is very simple, as it relies entirely on \cref{thm:BOMisOM}, but it involves some auxiliary oriented matroids and some cumbersome notation to identify them.

With the help of the decoration, we will first add to $\OM'$ the elements of $\ipairs$ to get a new oriented matroid $\tilde \OM'$ on $F\cup \ipairs$. We do this by adding for each $f\in F$ the parallel elements $(i,j)= \epsilon(i,j) f$ for $(i,j)\in \delta(f)$. The restriction of $\tilde \OM'$ to $\ipairs$ is a sweep oriented matroid, as it fulfills the transitivity condition from \cref{lem:OMtransitivity} by hypothesis.
We want to apply \cref{thm:BOMisOM} to find the associated big oriented matroid. While \cref{thm:BOMisOM} is only stated to extend a matroid from $\ipairs$ to $\ivl\cup \ipairs$, the same proof carries on almost verbatim to extend a matroid from $F\cup \ipairs$ to $F\cup \ivl\cup \ipairs$. We associate a family of covectors $X^k$ on $F\cup\ivl\cup \ipairs$ to every covector $X$ of~$\tilde \OM'$ in the very same way, just ignoring the entries in~$F$ when generating the values for $\ivl$ in~$X^k$. These are the covectors of an oriented matroid~$\tilde \OM$ (by the same argument as in \cref{thm:BOMisOM}), and its restriction to $\ivl\cup F$ is the desired oriented matroid~$\OM$.
\end{proof}

\end{document}